\documentclass{article}
\usepackage{fullpage}
\usepackage{hyperref}
\usepackage{xcolor}
\usepackage{amsthm}
\usepackage{amsmath,amsfonts,amssymb}
\usepackage{graphicx}
\usepackage{subcaption}
\usepackage{mathrsfs}
\usepackage[margin=1.0in]{geometry}
\usepackage{appendix}
\usepackage{accents}

\usepackage{lipsum}
\usepackage{amsfonts}
\usepackage{graphicx}
\usepackage{epstopdf}
\usepackage{algorithmic}
\ifpdf
  \DeclareGraphicsExtensions{.eps,.pdf,.png,.jpg}
\else
  \DeclareGraphicsExtensions{.eps}
\fi
\usepackage[shortlabels]{enumitem}

\usepackage{mathrsfs}
\usepackage{dsfont}
\usepackage[margin=1.0in]{geometry}

\usepackage[shortlabels]{enumitem}
\newtheorem{theorem}{Theorem}[section]
\newtheorem{remark}[theorem]{Remark}
\newtheorem{proposition}[theorem]{Proposition}

\newtheorem{lemma}[theorem]{Lemma}
\newtheorem{corollary}[theorem]{Corollary}

\usepackage{accents}

\newcommand{\jay}{{\mathfrak{j}}}
\newcommand{\que}{{\mathfrak{q}}}
\newcommand{\sfk}{{\mathsf k}}
\newcommand{\sfl}{\mathsf{l}}
\newcommand{\sfs}{{\mathsf s}}

\newcommand{\sfu}{{\mathsf u}}
\newcommand{\sfv}{{\mathsf v}}

\usepackage[export]{adjustbox}
\newcommand{\tb}[1]{\textcolor{black}{#1}}

\newcommand{\mS}{\mathcal{S}}

\newcommand{\mV}{{\mathcal{V}}}

\newcommand{\mG}{\mathcal{G}}
\newcommand{\mI}{\mathcal{I}}\newcommand{\mJ}{\mathcal{J}}
\newcommand{\IS}{{It\^o-Schr\"odinger}}
\newcommand{\fC}{{\mathfrak C}}

\newcommand{\E}{\mathbb E}
\newcommand{\Rm}{\mathbb R}\newcommand{\Nm}{\mathbb N}
\newcommand{\Xm}{\mathbb X}
\newcommand{\Zm}{\mathbb Z}
\newcommand{\Tm}{\mathbb T}
\newcommand{\dint}{\displaystyle\int}
\newcommand{\mE}{\mathcal E}\newcommand{\mF}{\mathcal F}

\newcommand{\aver}[1]{\langle {#1} \rangle}
\newcommand{\dsum}{\displaystyle\sum}

\newcommand{\us}{u}
\newcommand{\up}{{u_\sharp}}
\newcommand{\uc}{{u_c}}
\newcommand{\ut}{u^\theta}
\newcommand{\hut}{\hat u^\theta}
\newcommand{\utc}{{u^\theta_c}}
\newcommand{\ud}{u^\Delta}
\newcommand{\udc}{u^\Delta_c}

\newcommand{\udp}{{u_\sharp^\Delta}}

\newcommand{\utd}{u^{\theta \Delta}}
\newcommand{\hatutd}{\hat{u}^{\theta\Delta}}
\newcommand{\hatutdcon}{\hat{u}^{\theta\Delta *}}
\newcommand{\utdc}{u^{\theta\Delta}_c}

\newcommand{\udd}{u^\Delta_{\delta}}

\newcommand{\utp}{u^{\theta}_{\sharp}}
\newcommand{\utdp}{u^{\theta\Delta}_{\sharp}}
\newcommand{\utdd}{u^{\theta\Delta}_{\delta}}

\newcommand{\mudpq}{\mu^{ \Delta}_{p,q}}
\newcommand{\hmudpq}{\hat{\mu}^{ \Delta}_{p,q}}

\newcommand{\Psidpq}{\Psi^{\Delta}_{p,q}}

\newcommand{\psit}{\psi^{\theta}}
\newcommand{\psitd}{\psi^{\theta\Delta}}

\newcommand{\chid}{ \chi^{\Delta}}

\newcommand{\mL}{\mathcal{L}}
\newcommand{\mP}{\mathcal{P}}

\newcommand{\Ltpq}{\mathcal{L}_{p,q}^{\theta}}
\newcommand{\Ed}{E^{\Delta}}

\newcommand{\TP}{|\!\!\>|\!\!\>|}

\newcommand{\sk}{{\rm k}}
\newcommand{\vs}{\Vec{\mathsf{s}}}
\newcommand{\vk}{\Vec{\mathsf{k}}}
\newcommand{\sss}{\Vec{s}}
\newcommand{\ttt}{\Vec{t}}
\newcommand{\kkk}{\Vec{k}}

\newcommand{\qqq}{\Vec{q}}

\newcommand{\rG}{{\rm G}}
\newcommand{\rV}{{\rm V}}

\newcommand{\sfc}{\tb{\mathrm{C}}}
\newcommand{\tbb}[1]{\textcolor{black}{#1}}

\title{Splitting algorithms for paraxial and It\^o-Schr\"odinger models of wave propagation in random media}
\author{Guillaume Bal \thanks{Departments of Statistics and Mathematics and Committee on Computational and Applied Mathematics, University of Chicago, Chicago, IL 60637; guillaumebal@uchicago.edu} \and Anjali Nair \thanks{Department of Statistics and Committee on Computational and Applied Mathematics, University of Chicago, Chicago, IL 60637; anjalinair@uchicago.edu}}
\date{\today}

\begin{document}

\maketitle

\begin{abstract}
 This paper introduces a full discretization procedure to solve wave beam propagation in random media modeled by a paraxial wave equation or an \IS\ stochastic partial differential equation.  This method bears similarities with the phase screen method used routinely to solve such problems. The main axis of propagation is discretized by a centered splitting scheme with step $\Delta z$ while the transverse variables are treated by a spectral method after appropriate spatial truncation. The originality of our approach is its theoretical validity even when the typical wavelength $\theta$ of the propagating signal satisfies $\theta\ll\Delta z$. More precisely, we obtain a convergence of order $\Delta z$ in mean-square sense while the errors on statistical moments are of order $(\Delta z)^2$ as expected for standard centered splitting schemes. This is a surprising result as splitting schemes typically do not converge when $\Delta z$ is not the smallest scale of the problem. The analysis is based on equations satisfied by statistical moments in the \IS\ case and on integral (Duhamel) expansions for the paraxial model. Several numerical simulations illustrate and confirm the theoretical findings.

\end{abstract}


\section{Introduction}\label{sec:intro}
This paper concerns the numerical simulation of wave beams propagating in an oscillatory random environment and described by either a paraxial or an \IS\ equation.  The paraxial equation is given by 
\begin{equation}\label{eqn:PWE}
    \begin{aligned}        \partial_z\ut&=i\kappa_1(z)\Delta_x\ut+ i\kappa_2(z)\frac1{\sqrt{\theta}}\nu\Big(\frac{z}{\theta},x\Big)\ut,\quad z>0, x\in\mathbb{R}^d ;\qquad 
        \ut(0,x)=u_0(x).
    \end{aligned}
\end{equation}

Here, 
$u_0(x)$ is the incident beam profile and $\kappa_{1,2}(z)$ are smooth positive functions of $z$ with bounded inverse. \tb{The functions $\kappa_{1,2}$ are related to the mean value of the refractive index, and we refer the reader to the Supplementary Materials for a detailed derivation.}  We assume that $\nu$ is a \tb{mean-zero real-valued} stationary Gaussian random process with covariance function \tb{$C(z-z',x-x'):=\mathbb{E}[\nu(z,x)\nu(z',x')]$ for $(z,x)$ and $(z',x')$ in $\Rm^{d+1}$, where $\E$ denotes mathematical expectation}.
The parameter $\theta$ represents the ratio of the typical wavelength of the propagating signal with respect to the correlation length in the medium. It satisfies $\theta\ll1$ in laser light propagation in turbulent atmospheres, which is our primary application. For justification and analyses of the paraxial model, see, e.g., \cite{bailly1996parabolic,BKR-KRM-10,bal2011asymptotics,garnier2009coupled,gu2021gaussian} and the Supplementary Materials section.

Solving \eqref{eqn:PWE} numerically is challenging when $\theta\ll1$. The simulation is, however, significantly simplified when $\kappa_2=0$, since $\Delta_x\ut$ is local in the Fourier domain, or when $\kappa_1=0$ since the equation is then local in $x$. It is therefore natural to use a splitting algorithm, which treats the transport and interaction terms in turn over small intervals $\Delta z$, and has been used to partially discretize deterministic and random Schr\"odinger equations in various contexts \tbb{\cite{anton2018exponential,bal2004time,bao2002time,bao2024explicit, brehier2022strong, bruned2022resonance, cui2018analysis, de2001theoretical,gomez2014asymptotics,liu2013mass,liu2013order,marty2006splitting}}. 
Standard convergence results are obtained when the interval $\Delta z$ is sufficiently smaller than the smallest scale in the problem, which here is $\theta\ll1$. 

Choosing $\Delta z$ larger than the smallest scale in the system typically leads to inaccurate simulations \cite{bal2004time,bao2002time}. Yet, splitting techniques called phase screen methods are routinely used in the numerical simulations of paraxial wave propagation in random media \cite{ferlic2023synchronous, gbur2014partially, martin1988intensity, nair2023scintillation, rabinovich2023optical, schmidt2010numerical,spivack1989split}. The reason for this fact, whose justification is one of the main objectives of this paper, is that $\theta^{-\frac12}\nu(\frac z\theta,x){dz}$ is well-approximated by its white noise limit $dB(z,x)$. 

The It\^o-Schr\"{o}dinger equation is the white noise approximation of \eqref{eqn:PWE} given by
\begin{equation}\label{eqn:Ito}
    du=i\kappa_1(z)\Delta_xudz-\frac{\kappa_2^2(z)R(0)}{2}udz+i\kappa_2(z)udB,\quad u(0,x)=u_0(x)\,.
\end{equation}
Here, $B(z,x)$ is a mean-zero Gaussian process characterized by the covariance function $\mathbb{E}[B(z,x)B(z',x')]=\min(z,z')R(x-x')$ with $R(x) = \int_{\Rm} C(s,x)ds$.
See, e.g., \cite{fannjiang2004scaling} for details of the derivation of \eqref{eqn:Ito} from \eqref{eqn:PWE}, which shows that $\ut$ solution of \eqref{eqn:PWE} converges to $\us$ solution of \eqref{eqn:Ito} in distribution.

The splitting approximations $\utd$ and $\ud$ to $\ut$ and $\us$, respectively, are introduced in Section \ref{sec:subsplitting} below with $\Delta\equiv \Delta z$ the splitting step. We also aim to analyze a full discretization of the transverse variables $x\in\Rm^d$, with $d=2$ in practical applications. This is done in three steps detailed in Section \ref{sec:spatialdisc}. We first discretize the random medium $\nu(z,x)$ by a finite dimensional approximation we denote by $\nu_c(z,x)$. A similar procedure approximates $B(s,x)$ by $B_c(s,x)$. The splitting approximations in these modified random environments are then denoted by $\utdc$ and $\udc$. A second step shows that the solutions of \eqref{eqn:PWE} and \eqref{eqn:Ito} decay rapidly in the $x$ variable when the incident beam profile $u_0(x)$ also decays rapidly. \tb{As we use trigonometric polynomials to represent spatial discretization in $x$, we next periodize the solutions.} We denote by $\utdp$ and $\udp$ the periodizations of $\utdc$ and $\udc$, respectively, on the torus $\Tm_L^d:[-\frac L2,\frac L2]^d$. In a final step, the torus is discretized by a uniform grid of spacing $\delta\equiv \Delta x$. The solution is represented by a trigonometric polynomial on which applications of functions of the Laplacian may be carried out explicitly. Fully discretized solutions of \eqref{eqn:PWE} and \eqref{eqn:Ito} are then denoted by $\utdd$ and $\udd$, respectively.

This paper provides error estimates for the aforementioned approximations summarized in Section \ref{sec:submain}. In particular, we show that the splitting algorithm displays a second-order accuracy in $\Delta z$ (when $\theta\leq\Delta z$) for statistical moments whereas it is only first order in a \tb{(strong)} mean\tb{-}square sense.  The convergence in $\delta=\Delta x$ is always super-algebraic under smoothness conditions on the incident beam. We also show in the Supplementary Material that all discretizations converge in distribution to the \IS\ solution $\us$ borrowing tools from \cite{bal2024complex,bal2024long}. Several numerical simulations of wave propagation for both paraxial and \IS\ models confirm the theoretically predicted rates of convergence.

The rest of the paper is structured as follows. The rest of this introduction section collects assumptions and notation used throughout the paper. Section \ref{sec:main} describes the splitting algorithm, the periodization step, and the fully discrete splitting algorithm, and then states our main results of convergence and error estimates. The \tb{proofs} of the results of convergence for the It\^o-Schr\"odinger equation are given in Section \ref{sec:ITO} while the corresponding proofs for the paraxial equation are given in Section \ref{sec:paraxial}. 
Finally, several numerical simulations are presented in Section \ref{sec:num} to illustrate the theoretical findings. Additional information on convergence in distribution and further numerical simulations are provided in the Supplementary Material.

\paragraph{Assumptions on incident beam and random medium} We assume that the incident source $u_0$ satisfies
\begin{equation}\label{eqn:u0_bound}
    \sup_{|\beta|\leq M} \int_{\Rm^d} 
    \aver{\xi}^{2N} |\partial^\beta \hat u_0(\xi)|^2 d\xi  = {\rm C}_{M,N}[u_0]<\infty.
\end{equation}
\tb{Here, $\beta$ is a standard multi-index and }the values of $(M,N)$ are different for different estimates. \tb{We also use the notation} $\aver{x}=\sqrt{1+|x|^2}$ for a vector $x$.

The  random environment is modeled by $\nu(z,x)$, a mean-zero real-valued stationary Gaussian field characterized by the covariance function
\[
 \E \nu(t,x)\nu(s,y) = C(t-s,x,y) = \dint_{\Rm^{2d}} \hat C(t-s,\xi,\zeta)e^{i(\xi\cdot x-\zeta\cdot y)} \frac{d\xi d\zeta}{(2\pi)^{2d}},
 \quad
 \E \hat\nu(t,\xi) \hat\nu^\ast(s,\zeta) =\hat C(t-s,\xi,\zeta)
\]
with $\hat C(s,\xi,\zeta)d\xi d\zeta$ a measure \tb{of bounded total variation} and $\hat{\nu}(z,\xi)$ the Fourier transform of $\nu(z,x)$ w.r.t $x$. \tb{We use the convenient notation $^\ast$ for complex conjugation.} 
In the \IS\ model, the random medium is characterized by the power spectrum $\hat R(\xi,k) := \int_{\Rm} \hat C(s,\xi,k) ds$. \tb{Since $\nu(z,x)$ is stationary, we have $\hat C(s,\xi,\zeta)= (2\pi)^d\hat C(s,\xi)\delta(\xi-\zeta)$  and $\hat R(\xi,\zeta)= (2\pi)^d\hat R(\xi)\delta(\xi-\zeta)$ (using the same symbols in both contexts to simplify notation).}
We assume total variation properties of the form
\begin{equation}\label{eq:intcorr}
    \tb{ \int_{\Rm^{d}} \hat R(\xi) \aver{\xi}^{M}d\xi= \fC_{M}<\infty,\qquad  \int_{\Rm^{d+1}} |\partial^\alpha_{t}\partial^\beta_\xi\hat C(t,\xi)| \aver{t}^{N}\aver{\xi}^{M} dtd\xi  = \fC_{M,N}[\nu] <\infty\,,}
\end{equation}
\tb{where $\alpha=\{0,1\}$ and $|\beta|\in\{0,1,2\}$. We note that from the definition of $R$, the first equality is a consequence of the second. } 

\paragraph{Discretized random medium} We now construct a random medium strongly correlated to $\nu(z,x)$ and characterized by a finite number of independent Gaussian random variables. 
We start from $\hat\nu(z,k)$ for the above stationary process with covariance function given by $\hat C(s,k)$.  We introduce a high-frequency cut-off $K_k\gg1$ and a discretization step $\Delta k\ll1$. We define a smooth cut-off function $\chi_c\in C^\infty_c(\Rm^{d})$ with $\chi_c(k)=1$ on $[-\frac12 K_k,\frac12 K_k]^d$ and $\chi_c(k)=0$ on $\mathbb{R}^d\setminus[- K_k, K_k]^d$. We also define the lattice $\Zm_\Delta^d:=(\Delta k \,\Zm)^d$.  For $q\in\Zm_\Delta^d$, let $\square_q$ be the (Cartesian) cube centered at $q$ of volume $(\Delta k)^d$. We define
\begin{equation}\label{eq:discretevar}
    \hat \nu_q(z) =  
\int_{\square_q} \hat \nu(z,k) dk,\qquad  \hat B_q(z) = \int_{\square_q} \hat B(z,k) dk.
\end{equation}
These are independent mean-zero Gaussian processes with
$\E \hat \nu_q(z_1)\hat \nu^\ast_p(z_2) = \mathds{1}_{p=q}\int_{\square_q} \hat C(z_1-z_2,k)dk$ and $\E \hat B_q(z_1)\hat B^\ast_p(z_2) = \mathds{1}_{p=q} \min{(z_1,z_2)}\int_{\square_q} \hat R(k)dk$.
We then construct
\begin{equation}\label{eq:discreterandom}
\hat \nu_c(z,\xi) = \chi_c(\xi) \sum_{q\in\Zm_\Delta^d} \hat \nu_q(z) \delta(\xi-q), \qquad 
\hat B_c(z,\xi) = \chi_c(\xi) \sum_{q\in\Zm_\Delta^d} \hat B_q(z) \delta(\xi-q).
\end{equation}
Note that $\nu_c(z,x)$ and $B_c(z,x)$ obtained by inverse Fourier transform are real-valued since, for instance, $\hat \nu_c(z,-\xi) =\hat \nu_c^*(z,\xi)$.
For (a.e.) $\xi\in\Rm^d$, denote $q_\xi\in\Zm_\Delta^d$ the unique element such that $\xi\in \square_q$. Then we find
\begin{equation}\label{eq:mediumcor}
\begin{aligned}
\E \hat \nu(z_1,\xi) \hat \nu^\ast_c(z_2,\zeta) &= (2\pi)^d\chi_c(\zeta) \delta(\zeta- q_\xi)\hat C(z_2-z_1,\xi)   ,
\\
\E \hat \nu_c(z_1,\xi) \hat \nu^\ast_c(z_2,\zeta) &=(2\pi)^d\chi_c^2(\xi) \delta(\xi-q_\xi)\delta(\zeta-q_\xi)\int_{\square_{q_\xi}} \hat C(z_2-z_1,k)dk,
\end{aligned}
\end{equation}
and corresponding expressions for the white noise limits.
We thus obtain highly correlated continuous and discrete random media on the domain where $\chi_c=1$.  Assuming that $M=M(r)$ is sufficiently large in \eqref{eq:intcorr}, we verify that for any $r>0$, then $ \int_{\Rm^{d+1}} (1-\chi_c(k)) |\hat C(s,k)| dk ds \leq \sfc K_k^{-r}$. This shows how $K_k$ may be chosen to capture most of the correlation in the medium. The number of Gaussian variables needed to describe the discrete random medium in \eqref{eq:discreterandom} is therefore of order $(\frac{K_k}{\Delta k})^d$.

\begin{remark}\label{rem:nu_c_disc}
Alternatively, we could define $\hat{\nu}_q(z)=\hat{\nu}_{-q}(z)$ \tb{mean-zero} Gaussian processes with correlation $\E \hat \nu_q(z_1)\hat \nu^\ast_p(z_2)=\mathds{1}_{p=q} (2\pi\Delta k)^d \hat C(z_1-z_2,p)$ with $\hat \nu_c$ defined from these Gaussian random variables as above. This medium has almost the same statistics as the previous discrete one, but is uncorrelated to it. We use this medium for numerical simulations as the random variables are more easily described.
\end{remark}

\paragraph{General notation} We summarize here the main notation used throughout the paper. We recall that $z\geq0$ is the coordinate along the main axis of propagation while $x\in\Rm^d$ denotes transverse spatial variables. We denote by $X=(x_1,\ldots\tb{,} x_p)$ a collection of $p\geq0$ points in $\Rm^d$ and $Y=(y_1,\ldots \tb{,}y_q)$ a collection of $q\geq0$ points in $\Rm^d$. We denote by $Z$ a target distance of propagation. We aim to solve the wave beam problem for $(z,x)\in [0,Z]\times\Rm^d$. The interval $[0,Z]$ is discretized into $N_z+1$ points $z_k=k\Delta z$ for $0\leq k\leq N_z$ and $N_z \Delta z=Z.$

We denote by $u$ the solution to the It\^o-Schr\"{o}dinger equation \eqref{eqn:Ito} and $\ud$ its time splitting solution with step size $\Delta z$. When the random potential $\nu(z,x)$ is periodized and discretized by $\nu_c(z,x)$, we denote the corresponding solutions by $\uc$ and $\udc$ respectively. The period of the random medium is $L=2\pi/(\Delta k)$ corresponding to $O(K_k/(\Delta k))$ Fourier modes along each dimension. 
We denote by $\up$ the periodization of $\uc$ on $\Tm^d_L\equiv [-\frac L2,\frac L2]^d$. For the time splitting case we denote by $\udp$ the periodic approximation of $\udc$. We set the solution periodization box length $L=2\pi/\Delta k$ although this can be any integer multiple of $2\pi/\Delta k$. 
Finally we denote by $u_\delta$ the spatial discretization of $\up$. In the time splitting case, the spatially discrete version of $\udp$ is denoted by $\udd$. The spatial discretization of the solution is $\Delta x=L/N_x$ for $N_x$ grid points along each dimension of $x$.

For the paraxial model, we denote by $\ut$ the solution to the paraxial equation \eqref{eqn:PWE}, $\utd$ its time splitting solution and $\utdc$ the time splitting solution when the potential is replaced by $\nu_c$. We denote by $\utdp$ the periodic extension of $\utdc$ and by $\utdd$ the spatial discretization of $\utdp$. 

We denote the potential in the rescaled coordinates by $\nu^\theta(z,x)=\theta^{-\frac12}\nu(\theta^{-1}z,x)$ and the corresponding correlations by $C^\theta(z,x,y)=\theta^{-1}C(\theta^{-1}z,x,y)$.

For $u(z,x)$ a continuous random field, and $(X,Y)$ a collection of points $p+q$ points, we define the \tb{$(p+q)$}th statistical moment of $u$ in the physical and complex symmetrized Fourier variables as
\begin{eqnarray}
    \label{eq:mupq}\mu_{p,q}[u](z,X,Y) &=& \mathbb{E}\Big[\prod_{j=1}^pu(z,x_j)\prod_{l=1}^q u^\ast(z,y_l)\Big],\\
    \label{eq:hatmupq}
    \hat{\mu}_{p,q}[u](z,v)&=&\int_{\mathbb{R}^{(p+q)d}}\mu_{p,q}[u](X,Y)e^{-i(\sum_{j=1}^px_j\cdot\xi_j-\sum_{l=1}^qy_l\cdot\zeta_l)}dx_1\cdots dx_pdy_1\cdots dy_q,
\end{eqnarray}
where $v=(\xi_1,\cdots\tb{,}\xi_p,\zeta_1,\cdots,\zeta_q)$ denotes the vector of (symmetrized) variables dual to $(X,Y)$.  The standard Fourier transform is defined by $\mF f(\xi) = \int_{\Rm^d} e^{-ix\cdot\xi} f(x) dx$ with inverse $\mF^{-1}\hat f(x) = \int_{\Rm^d} e^{ix\cdot\xi} \hat f(\xi) \frac{d\xi}{(2\pi)^d}$.

We use $\|\cdot\|$ to denote the total variation (TV) norm of a Borel measure on $\Rm^n$. When the measure $\rho d\xi$ has a density $\rho(\xi)$, we also denote by $\|\rho\|=\|\rho d\xi\|$ the $L^1(\Rm^n)$ norm. The $L^2(X)$ norm is denoted $\|\cdot\|_{L^2(X)}$ or $\|\cdot\|_2$. The uniform norm is denoted as $\|\cdot\|_\infty$ or $\|\cdot\|_{L^\infty(X)}$ while the k-Lipschitz norm is denoted as $\|\cdot\|_{\mathrm{k},\infty}$. We also define a $\TP\cdot\TP$ \tb{root-mean-square (rms)}-norm in \eqref{eq:rmsnorm} below.

\section{Splitting algorithms and convergence results} \label{sec:main}
\subsection{Splitting scheme for paraxial and It\^o-Schr\"odinger models}
\label{sec:subsplitting}
We start with a discretization of the wave solutions in the axial variable $z$. We consider a final distance $Z=N_z\Delta z$ and a grid $z_n=n\Delta z$ for $0\leq n\leq N_z$. Integrals on each interval $[n\Delta z,(n+1)\Delta z]$ are approximated by a one-point collocation method. 
Let \tb{$\gamma\in[0,1)$} parametrize the collocation points. Splitting schemes are then defined by a choice of
\begin{equation}\label{eqn:tau}
    \tau_\gamma(z):=\Delta z\sum_{n\ge 0}\delta(z-(n+\gamma)\Delta z).
\end{equation}
The end-point \tb{choice $\gamma=0$ corresponds} to the Lie splitting while the midpoint choice $\gamma=\frac12$ corresponds to the centered Strang splitting. We next define the integrals for $0\leq z_1\leq z_2\leq Z$:
\begin{equation}\label{eq:intchi}
\begin{aligned}
    \chi_{z_1}(z_2)&=\int_{z_1}^{z_2}\kappa_1(s)ds,\qquad   \chid_{z_1}(z_2)&=\int_{z_1}^{z_2}\tau_\gamma(s)\kappa_1(s)ds  = \sum_{j=\lfloor z_1/\Delta z\rfloor+1}^{\lfloor z_2/\Delta z\rfloor}\kappa_1((j+\gamma\tb{-1})\Delta z)\Delta z.
\end{aligned}
\end{equation}
Here, $\lfloor \cdot \rfloor$ is defined with $\lfloor z/\Delta z\rfloor=n$ when $z_{n}\le z< z_{n+1}$.  It is well-known \cite{geometric2006} that the centered splitting may offer second-order accuracy compared to the first-order accuracy of Lie splitting. In the splitting of stochastic equations, the concentration phenomenon $\E dB(s) dB(t)=\delta(t-s)$ renders the advantage of centered trapezoidal integration ineffective and only first-order convergence is expected for any value of $\gamma$ (see \cite{liu2013order} and subsequent calculations). However for moment calculations, we do obtain second-order accuracy for the centered splitting scheme $\gamma=\frac12$.

The splitting scheme for the paraxial model is defined as follows. Recalling the notation $\nu^\theta(z,x)=\theta^{-\frac12}\nu(\theta^{-1}z,x)$, an approximation $\utd$ to the solution of \eqref{eqn:PWE} is defined as the solution to:
\begin{equation}\label{eqn:PWE_split}
    \begin{aligned}    & \partial_z\utd
    =i\tau_\gamma(z)\kappa_1(z)\Delta_x\utd+
    i\kappa_2(z) \nu^\theta(z,x) \utd,
    \qquad 
       \utd(0,x)=u_0(x).
    \end{aligned}
\end{equation}
\tb{The Dirac comb $\tau_\gamma$ ensures that the Laplacian is `active' at discrete points in time.  This is a convenient notation ensuring that the evolution equation sees the effect of the Laplacian and the potential successively. Alternatively, }the splitting scheme may be split into a succession of simple steps. Define the operators:
\begin{eqnarray}
    \rV^\theta_{z_1}(z_2):\  && \psi(x)\mapsto 
    (\rV^\theta\psi)(x) = \exp{\Big(\int_{z_1}^{z_2}i\kappa_2(s)\nu^\theta(s,x)ds}\Big) \psi(x) \label{eq:Vtheta}
    \\
    \mG(z):\  && \psi(x) \mapsto \mathcal{G}\psi(x)=\int_{\mathbb{R}^d}G(x-x',z)\psi(x')dx',\qquad  G(x,z):=\frac{1}{(4\pi i z)^{\frac d2}}e^{\frac{i|x|^2}{4z}}.\label{eq:mG}
\end{eqnarray}
The solution to \eqref{eqn:PWE_split} is then given explicitly for $z\in(z_{n},z_{n+1}]$ by the more standard form:
\begin{equation}\label{eq:splittingexplicit}
\utd(z,x) = \left\{ \begin{array}{cl}
 \rV^\theta_{z_n}(z) \utd(z_n,x), & z_n<z\le z_n+\gamma\Delta z,
 \\[3mm]
\rV^\theta_{z_n+\gamma \Delta z}(z)\circ \mG(\chi_{z_n}(z_{n+1})) \circ \rV^\theta_{z_n}(z_n+\gamma \Delta z) \utd(z_n,x), &z_n+\gamma\Delta z<z\le z_{n+1}.
\end{array}\right.
\end{equation}
The splitting algorithm applied to the interval $[0,Z]$ is therefore a composition of operators that are straightforward to apply, with $\rV^\theta$ a multiplicative operator in the physical domain and $\mG(t)=\mF^{-1} e^{-it|\xi|^2} \mF$ a multiplicative operator in the Fourier domain.

For the It\^o-Schr\"odinger model, the splitting approximation $\ud$ to \eqref{eqn:Ito} is given by
\begin{equation}\label{eqn:Ito_split}
d\ud=i\tau_\gamma(z)\kappa_1(z)\Delta_x\ud dz-\frac{\kappa_2^2(z)R(0)}{2}\ud dz+i\kappa_2(z)\ud dB,\quad \ud(0,x)=u_0(x)\,.
\end{equation}
As can be verified by a standard application of the It\^o formula, the solution of this stochastic partial differential equation is also given by the explicit procedure \eqref{eq:splittingexplicit}, where $\rV^\theta$ is replaced by the (local) multiplicative operator \tb{(the phase screen)} $ \rV_{z_1}(z_2):\psi(x)\mapsto 
    (\rV\psi)(x) = e^{\int_{z_1}^{z_2}i\kappa_2(s)dB(s,x)} \psi(x)$.

\subsection{Spatial periodization and discretization}
\label{sec:spatialdisc}

A full discretization in the transverse variable $x$ of the above approximations $\utd$ and $\ud$ \tb{is done using a spectral method and} requires the following three steps. 

The first step approximates the stationary random medium $\nu(z,x)$ by $\nu_c(z,x)$ described in \eqref{eq:discreterandom}, which involves an order of $(K_k/\Delta k)^d$ Gaussian random variables and is periodic on the torus $\Tm^d_L = [-\frac L2,\frac L2]^d$ for $L\Delta k=2\pi$. The corresponding solutions for the paraxial and \IS\ models on $(z,x)\in\Rm_+\times\Rm^d$, before and after splitting in $z$, are denoted by $\utc$, $\uc$, and $\utdc$ as well as $\udc$, respectively. 

The second step defines periodized solutions on $\Tm^d_L$. Periodization is defined by:
\begin{equation}\label{eq:periodization}
     \sharp:\  u(x) \,\mapsto\, u_\sharp(x) \,:= \dsum_{h\in (L\Zm)^d} u(x+h)
\end{equation}
for $u(x)$ an integrable function on $\Rm^d$. 
We will show that for incident conditions $u_0(x)$ with sufficiently fast decay, then $u(z,x)$ also has sufficiently fast decay in $x$ so that it is well approximated by its periodization on $\Tm_L^d$. Starting from $\utdc$ and $\udc$, the corresponding periodizations are called $\utdp$ and $\udp$, respectively.

The final step is a full spatial discretization based on approximating periodic functions on $\Tm_L^d$ by trigonometric polynomials of degree $N_x$ in each of the $d$ transverse spatial dimensions.  We denote by $\Pi=\Pi_{N_x}$ the $L^2(\Tm_L^d)-$orthogonal projection onto such modes.  For $f\in L^2(\Tm_L^d)$, we have:\begin{equation}\label{eq:discFourier}
    f(x)=\sum_{l\in\Zm^d}  e^{i\Delta k l\cdot x} \hat f_l,\qquad 
\Pi f(x) = \sum_{|l_j|\leq N_x}  e^{i\Delta k l\cdot x} \hat f_l,
\qquad \hat f_l=L^{-d}\int_{\Tm_L^d} e^{-i \Delta k l\cdot x} f(x) dx.
\end{equation}
The polynomials $\Pi f$ are equivalently characterized by their evaluation on a discrete uniform grid with mesh size $\Delta x= L/N_x= 2\pi/(N_x\Delta k)$ with a discrete (fast) Fourier transform. 

We define $\sfv^\theta_\delta=\Pi \sfv^\theta_\delta$ and $\sfv_\delta=\Pi \sfv_\delta$ as the spatially discretized solutions to
\begin{equation}\label{eq:discreteeq}
\partial_z \sfv^\theta_\delta = i \phi(z) \Delta_x \sfv^\theta_\delta + i \kappa_2(z) \Pi(\nu^\theta_c \sfv^\theta_\delta),\qquad
d \sfv_\delta = i \phi(z) \Delta_x \sfv_\delta dz + \Pi \Big(-\kappa_2^2 \frac12 R_c(0) \sfv_\delta dz + i\kappa_2 \sfv_\delta dB_c \Big).
\end{equation}
Here, $\phi(z)$ is either $\kappa_1(z)$ or $(\tau_\gamma\kappa_1)(z)$. We impose the incident conditions $\sfv^\theta_\delta(0,x)=\sfv_\delta(0,x)=\Pi u_0 (x)$.
This defines $\sfv^\theta_\delta=:u^\theta_\delta$ as solution of a finite system of stochastic differential equations and $\sfv^\theta_\delta=:\utdd$ as a fully discretized system since $(\partial_z-i\Pi \nu^\theta_c)$ is solved explicitly on the discrete spatial grid points while $(\partial_z-i\kappa_1(z)\Delta_x)$ is solved explicitly for each of the finitely many Fourier coefficients. This similarly defines $\sfv_\delta=:u_\delta$ for $\phi=\kappa_1$ and $\sfv_\delta:= \udd$ for $\phi=\tau_\gamma\kappa_1$.

\subsection{Main convergence results} 
\label{sec:submain}
We now compare the exact solutions $\us$ and $\ut$ to their semi-discrete approximations $\ud$ and $\utd$ and fully discrete approximations $\udd$ and $\utdd$. Our convergence rates are essentially uniform in $\theta\leq1$. In particular, we show that splitting algorithms provide accurate solutions independent of $0<\theta$ even when the splitting step $\Delta z\gg\theta$. 


We consider two types of convergence. The first type is a \tb{\em strong} \tb{(root-mean-square)} estimate of the form $(\E\|\us-\ud\|_{L^2(\Rm^d)}^2)^{\frac12}$ at grid points $z=z_n$, with a rate of convergence at most first-order in $\Delta z$ even for the centered splitting scheme $\gamma=\frac12$. The second type is a \tb{\em (stochastically) weak} estimate for spatial moments (uniform in the spatial variables) of the form $\| \mu_{p,q}[\us]-\mu_{p,q}[\ud]\|_\infty$. We will see that such estimates are first-order for any splitting algorithm and second-order in $\Delta z$ when $\gamma=\frac12$. The various constants that appear in the estimates are of the form $\sfc {\rm C}_{M,N}[u_0] e^{\sfc\aver{z}^2 (p+q)^2\fC_{P,Q}[\nu]}$ and $\sfc \mathfrak{C}_{M,1}[\hat C(z,\cdot)] e^{\sfc\aver{z}^2 (p+q)^2\fC_{P,Q}[\nu]}$ with universal constants $\sfc$ and constants of regularity $(M,N,P,Q)$ in \eqref{eqn:u0_bound} and \eqref{eq:intcorr} that depend on the estimate of interest. We do not keep track of these constants explicitly.

In the Supplementary Material, we obtain a third type of convergence showing that for a fixed value of $z$, all processes $\ud$, $\utd$, $\udd$, and $\utdd$ converge in law as continuous processes to the law generated by the \IS\ model $x\mapsto \us(z,x)$. This proximity of $\ut$ to $\us$ in law heuristically justifies why we may obtain convergence even when $\theta\ll \Delta z$.

For $u(s,x)$ sufficiently smooth, we define the root-mean-square norm
\begin{equation}\label{eq:rmsnorm}
    \TP u \TP_{\Xm} := \sup_{0\le s\le Z} (\mathbb{E}\|u(s,\cdot)\|^2_{L^2(\Xm)})^{\frac12}
\end{equation}
where $\Xm$ is either $\Rm^d$ or $\Tm^d_L$. Then we have the following strong estimates.
\begin{theorem}[\tb{Strong} estimates]\label{thm:pathwise_error}
\begin{enumerate}
    \item Let $\Xm=\Rm^d$ and $\sfv\in\{\us,\ut\}$. Then $ \TP \sfv-\sfv^\Delta \TP_{\Xm} \leq \sfc \Delta z$.
    \item Let $\Xm=\Rm^d$ and $\sfv\in\{\us,\ud,\ut,\utd\}$. Then for  $N\geq1$, $\TP \sfv-\sfv_c \TP_{\Xm} \leq \sfc_N [(\Delta k)^2 + K_k^{-N}]$.
    \item Let $\Xm=\Tm^d_L$ and $\sfv\in\{\us,\ud,\ut,\utd\}$. Then for  $N\geq1$, $\TP \sfv_c-\sfv_\sharp \TP_{\Xm} \leq \sfc_N L^{-N}$.
    \item Let $\Xm=\Tm^d_L$ and $\sfv\in\{\us,\ud,\ut,\utd\}$. Then for  $N\geq1$, $\TP \sfv_\sharp-\sfv_\delta \TP_{\Xm} \leq \sfc_N (\Delta x)^N$.
\end{enumerate}
The above estimates are uniform in $\theta\in(0,1]$. In the last three estimates, the regularity assumptions on $u_0$ and $\hat C$ depend on $N$.
\end{theorem}
The errors $u-\udd$ and $\ut-\utdd$ on $\Tm_L^d$ are therefore given by the sum of the above four contributions. For (essentially) first-order approximation results for non-linear \IS\ models, see \cite{liu2013mass,liu2013order}.


We now consider estimates for spatial moments. We assume that $0\leq s\leq Z$ and $(X,Y)\in\Xm^p\times\Xm^q$ for $\Xm=\Rm^d$ or $\Xm=\Tm^d_L$. We define the $\|\cdot\|_\infty$ norm as the supremum over $[0,Z]\times\Xm^p\times\Xm^q$.  For the \IS\ model, we define $\beta=\beta(\gamma)$  equal to $1$ when \tb{$\frac12\not=\gamma\in[0,1)$} and $\beta$ equal to $2$ when $\gamma=\frac12$.  For the paraxial model, $\beta=\beta(\gamma,\theta)$ is defined a bit differently. We still define $\beta=1$ when \tb{$1/2\not=\gamma\in[0,1)$}. When $\gamma=\frac12$ and $\theta\leq\Delta z$, define $\beta=2$. However, when $\gamma=\frac12$ and $\Delta z\leq\theta$, we define $\beta$ such that \tb{$(\Delta z)^\beta= \min (\theta\Delta z+(\Delta z)^2,\frac{(\Delta z)^2}\theta)$}. We observe that $\frac32\leq \beta\leq2$ and that $\beta=\frac32$ when $\Delta z=\theta^{2}$. 

Then we have the following result for the \IS\ and paraxial models.

\begin{theorem}[Moment estimates]\label{thm:momentestim} Let $D=[0,Z]\times\Xm^{p}\times\Xm^q$\tb{, for non-negative integers $p,q$}. Then we have the following.
\begin{enumerate}
    \item Let $\sfv\in\{\us,\ut\}$ and $\Xm=\Rm^d$. Then $\|{\mu}_{p,q}[\sfv]-{\mu}_{p,q}[\sfv^\Delta] \|_{L^\infty(D)} \leq \sfc (\Delta z)^\beta$.
    \item Let  $\sfv\in\{\us,\ud,\ut,\utd\}$ and $\Xm=\Rm^d$. Then for $N\geq1$, $\|{\mu}_{p,q}[\sfv]-{\mu}_{p,q}[\sfv_c] \|_{L^\infty(D)} \leq \sfc_N  [(\Delta k)^2+ K_k^{-N}]$.
    \item Let $\Xm=\Tm^d_L$ and $\sfv\in\{\us,\ud,\ut,\utd\}$. Then for  $N\geq1$, $ \| {\mu}_{p,q}[\sfv_c]-{\mu}_{p,q}[\sfv_\sharp] \|_{L^\infty(D)} \leq \sfc_N L^{-N}$.
    \item Let $\Xm=\Tm^d_L$ and $\sfv\in\{\us,\ud,\ut,\utd\}$. Then for  $N\geq1$, $\| {\mu}_{p,q}[\sfv_\sharp]-{\mu}_{p,q}[\sfv_\delta] \|_{L^\infty(D)} \leq \sfc_N (\Delta x)^N$.
\end{enumerate}
\end{theorem} 

\begin{remark}
    As in \cite[Theorem 4.3]{bal2024long}, we may show that  the statistical moments of the paraxial equation are well approximated by those of It\^o-Schr\"{o}dinger with $\|\mu_{p,q}[u^\theta](z,X,Y)-\mu_{p,q}[u](z,X,Y)\|_{L^\infty(\Rm^{d(p+q)})}\le \sfc\theta$ uniformly in $z\in[0,Z]$. Along with Theorem \ref{thm:momentestim}, this shows that the statistical moments of the fully discrete solution of the paraxial model $\utdd$ are well approximated by those of $\us$. Since $\utdd$ lives on a torus, we define its extension to $\mathbb{R}^d$ as $\sfu^{\theta\Delta}_\delta(z,x)=\utdd(z,x)\chi_L(x)$, where $\chi_L$ is a smooth window function which equals $1$ for $x\in[-L/4,L/4]^d$ and $0$ for $x\in\mathbb{R}^d\setminus[-L/2,L/2]^d$. This leads to the following corollary for $D=[0,Z]\times\mathbb{R}^{(p+q)d}$ (see also the Supplementary Material):
    \begin{corollary}\label{coro:PWE_to_IS}
    We have $\|\mu_{p,q}[\sfu^{\theta\Delta}_\delta]-\mu_{p,q}[u]\|_{L^\infty(D)}\le \sfc_N[\theta + (\Delta z)^\beta +(\Delta k)^2 + K_k^{-N}+ (\Delta x)^N]$.
\end{corollary}
\end{remark}

\begin{remark}
    For $\sfv\in\{u,\ut\}$, we have used the triangle inequality several times in order to show the convergence $\sfv^{\Delta}_\delta $ to $\sfv^\Delta_\sharp$ to $\sfv^\Delta_c$ to $\sfv^\Delta$ to $\sfv$. For $\sfv\in\{u,u_c,u_\sharp,\ut,\utc,\ut_\sharp\}$, a similar analysis as in the proof for the first part of Theorems~\ref{thm:pathwise_error} and~\ref{thm:momentestim} can be extended to show a convergence of $\sfv^\Delta\to \sfv$.
\end{remark}

Even though our convegence results are similar for the paraxial and \IS\ models, the latter is simpler to analyze technically as statistical moments of its solution satisfy closed form equations. Such equations are not available for the paraxial model, which is analyzed using a full Duhamel expansion characterizing all orders of interaction of the propagating field with the underlying random medium. The proofs of Theorems \ref{thm:pathwise_error} and \ref{thm:momentestim} are presented in Section \ref{sec:ITO}  for the \IS\ model and Section \ref{sec:paraxial} for the paraxial model. Since $\kappa_2(z)$ plays no essential role in what follows, with bounds in \eqref{eq:intcorr} multiplied by $\|\kappa_2\|^2_\infty$, we set $\kappa_2(z)=1$ for the rest of the paper. 


Before presenting these proofs, we conclude this section by a useful lemma at the core of our analysis of the splitting algorithms. Splitting schemes of the form $e^{t(A+B)}\approx e^{tA}e^{tB}$ (for time independent $A$ and $B$) are often analyzed by estimating the commutator of the generators $[A,B]$ \cite{geometric2006}. Since $\Delta z$ is not the smallest scale in the problem, such commutator techniques cannot be directly applied to our problem. Our approach is based on the estimation of phase differences that appear between the exact and approximate schemes.
 
Let $0<\Delta$. For $0\le k\le N_z-1$, we recall standard collocation and midpoint-rule estimates
\begin{equation}\label{eq:orderthree} 
\Big| \! \dint_{k\Delta}^{(k+1)\Delta} \hspace{-0.8cm} \psi(t) \big(1-\Delta\delta\big(t-(k+\gamma)\Delta\big)\big)dt \Big| \leq \sfc \Delta^2 \|\psi\|_{1,\infty},\ \,
\Big| \! \dint_{k\Delta}^{(k+1)\Delta} \hspace{-0.8cm} \psi(t) \big(1-\Delta\delta\big(t-(k+\frac{1}{2})\Delta\big) \big)dt \Big| \leq \sfc \Delta^3 \|\psi\|_{2,\infty}.
\end{equation}
In particular for the integrals defined in \eqref{eq:intchi}, we thus have:
\begin{equation}\label{eq:errorchi}
\sup_{0\leq s\leq z}|\chi_0(s)-\chid_0(s)| \leq \sfc \langle z\rangle\Delta z \|\kappa_1\|_{1,\infty},\quad 
\sup_k |\chi_0(k\Delta z)-\chid_0(k \Delta z)| \leq \sfc (\Delta z)^2 \|\kappa_1\|_{2,\infty}.
\end{equation}
The form of the splitting approximation \eqref{eq:splittingexplicit} shows that the evolution associated to the Laplace operator is carried out at discrete times rather than continuously for the paraxial model. It is therefore natural to compare the two phases associated to such evolutions. This is the role of the following lemma, which we will use a number of times.
\begin{lemma}\label{lem:Strang}
Let $\psi(s)\in W^{1,\infty}[0,Z]$ and $\tb{g}\in\Rm$. Then, for $\gamma=\frac12$, we have uniformly in $0<z<Z$:
\begin{equation}\label{eq:errorStrang} |I(z)| \leq \sfc \langle z\rangle(\Delta z)^2 [ \|\psi\|_\infty |\tb{g}|^2 + |\tb{g}| \|\psi\|_{\infty} + \|\psi\|_{1,\infty}],\qquad I(z) :=\int_0^z \psi(s) \big( e^{i\chi_0(s)\tb{g}} - e^{i\chid_0(s)\tb{g}}\big) ds.
\end{equation}
\begin{proof}
We write $z=k\Delta z+\delta z$ with $0\leq \delta z<\Delta z$. The integral over $[k\Delta z,z]$ is bounded by $\Delta z$ times $|\chi_0(s)-\chid_0(s)| |\psi|$ and hence of order $O((\Delta z)^2)$ using \eqref{eq:errorchi}.  The rest of the integral defining $I(z)$ is a sum over an order of $(\Delta z)^{-1}$ terms of the following form.  Define $\tilde\psi(s)=\psi(s)e^{i\chi_0(s)\tb{g}}$. Since $|e^{ia}-1-ia|\leq \sfc |a|^2$, the integral over $s\in[k\Delta z,(k+1)\Delta z]$ satisfies:
\[ \Big|\int_{k\Delta z}^{(k+1)\Delta z} \tilde\psi(s) \big( 1 - e^{i(\chid_0(s)-\chi_0(s))\tb{g}}\big) ds \Big| \leq \Big|\int_{k\Delta z}^{(k+1)\Delta z} \tilde \psi(s) \big(\chi_0(s) - \chid_0(s)\big)\tb{g} ds \Big| + \sfc |\tb{g}|^2  \|\psi\|_\infty (\Delta z)^3.\]
The first term on the above right-hand side is 
\[\int_{k\Delta z}^{(k+1)\Delta z} ds\tilde\psi(s) \Big( \int_0^{k\Delta z} + \int_{k\Delta z}^s \Big) dt \kappa_1(t) (1-\tau_{\frac12}(t)).\]
The first contribution for $t\in[0,k\Delta z]$ is of order $\|\kappa_1\|_{2,\infty}\|\psi\|_\infty(\Delta z)^3$ by \eqref{eq:errorchi} while the second contribution for $\tilde\Psi$ an antiderivative of $\tilde\psi$ is given explicitly by 
\[\int_{k\Delta z}^{(k+1)\Delta z} dt (\tilde\Psi((k+1)\Delta z)-\tilde\Psi(t)) \kappa_1(t) (1-\Delta z\delta\big(t-(k+\frac{1}{2})\Delta z\big)\big).\]

We use \eqref{eq:orderthree} again to conclude.   
\end{proof}
\end{lemma}

Note that the estimate \eqref{eq:errorStrang} may easily be replaced by an estimate of order $O(\Delta z)$ for all \tb{$\gamma\in[0,1)$}.

\section{Convergence results for the It\^o-Schr\"{o}dinger model} 
\label{sec:ITO}
We first recall that ensemble averages of products of wavefields solving the \IS\ equation satisfy closed-form partial differential equations. This also holds for the semi-discrete splitting solution $\ud$. More precisely, \tb{from stochastic calculus, it can be shown that} the \tb{$(p+q)$th} moments of the solution to the It\^o-Schr\"{o}dinger equation $\mu_{p,q}=\mu_{p,q}[u]$ and its time splitting approximation $\mudpq=\mu_{p,q}[\ud]$ satisfy the (deterministic) partial differential equations \cite{fouque1998forward,garnier2014scintillation,miyahara1982stochastic}:
\begin{eqnarray} \label{eqn:mu_pq_PDE} 
    \partial_z\mu &=&i\phi(z)\big(\sum_{j=1}^p\Delta_{x_j}-\sum_{l=1}^q\Delta_{y_l}\big)\mu+\mathcal{U}_{p,q}(X,Y)\mu,\quad \mu(0,X,Y)=\prod_{j=1}^pu_0(x_j)\prod_{l=1}^qu_0^\ast(y_l)
\end{eqnarray}
\tb{where} $$\mathcal{U}_{p,q}(X,Y)=\sum_{j=1}^p\sum_{l=1}^qR(x_j-y_l)-\sum_{1\le j<j'\le p}R(x_j-x_{j'})-\sum_{1\le l<l'\le q}R(y_l-y_{l'})-\frac{p+q}{2}R(0),$$ $\phi(z)=\kappa_1(z)$ for $\mu=\mu_{p,q}:=\mu_{p,q}[\us]$ while $\phi(z)=\tbb{\tau_\gamma(z)}\kappa_1(z)$ for $\mu=\mu^\Delta_{p,q}:=\mu_{p,q}[\ud]$. 

All estimates are carried out in the Fourier domain for the transverse spatial variables. Denote by $\hat{\mu}_{p,q}:=\hat{\mu}_{p,q}[u]$ and $\hat{\mu}^\Delta_{p,q}=\hat{\mu}_{p,q}[u^\Delta]$ the partial complex symmetrized Fourier transform~\eqref{eq:hatmupq} of $u$ and $u^\Delta$. As in~\cite{bal2024complex}, it is convenient to construct phase compensated moments of the form
\begin{equation}\label{eqn:psi_pq_def}
    \Psi_{p,q}(z,v)=\hat{\mu}_{p,q}(z,v)e^{i\chi_0(z)\big(\sum_{j=1}^p\!|\xi_j|^2-\sum_{l=1}^q\!|\zeta_l|^2\big)},
    \ \
    \Psidpq(z,v)=\hmudpq(z,v)e^{i\chid_0(z)\big(\sum_{j=1}^p\!|\xi_j|^2-\sum_{l=1}^q\!|\zeta_l|^2\big)}\,.
\end{equation}
These phase compensated moments satisfy the evolution equations 
\begin{equation}\label{eqn:Psi_pq}
    \begin{aligned}
        \partial_z\Psi_{p,q}=\mathcal{L}_{p,q}[\chi_0(z)]\Psi_{p,q}(z,v),\qquad  \partial_z\Psidpq=\mathcal{L}_{p,q}[\chid_0(z)]\Psidpq(\tb{z,}v)\,,
    \end{aligned}
\end{equation}
with initial condition $\Psi_{p,q}(0,v)=\Psidpq(0,v)=\prod_{j=1}^p\hat{u}_0(\xi_j)\prod_{l=1}^q\hat{u}_0^\ast(\zeta_l)$. Here, the operator $\mathcal{L}_{p,q}$ is given by 
\begin{eqnarray}
    \label{eqn:L_pq_def}
\mathcal{L}_{p,q}[\Phi]\psi (z,v)&=& \dint_{\mathbb{R}^d}\hat{R}(k) \big(L_{p,q}[\Phi] \psi\big)(z,v;k,k) \frac{dk}{(2\pi)^d},\\
L_{p,q}[\Phi] \psi(z,v;k_1,k_2) &=&
\dsum_{j=1}^p\sum_{l=1}^q \psi(z,\xi_j-k_1,\zeta_l-k_1)e^{i\Phi(z)[g(\xi_j,k_2)-g(\zeta_l,k_2)]}
\nonumber \\
&&-\dsum_{1\le j<j'\le p}\psi(z,\xi_j-k_1,\xi_{j'}+k_1)e^{i\Phi(z)[g(\xi_j,k_2)+g(\xi_{j'},-k_2)]} \nonumber\\
&&-\dsum_{1\le l<l'\le q}\psi(z,\zeta_l-k_1,\zeta_{l'}+k_1)e^{-i\Phi(z)[g(\zeta_l,k_2)+g(\zeta_{l'},-k_2)]}-\frac{p+q}{2}\psi(z,v),\nonumber
\end{eqnarray}
where we defined $g(\xi,k) := |\xi|^2-|\xi-k|^2$. We start with the following regularity result.
\begin{lemma}\label{lemma:Psi_pq_bound}
    \tb{Let $r\in\Nm$, \eqref{eqn:u0_bound}  hold with $N=r$ and \eqref{eq:intcorr} hold with $M=r$.} 
    Let $\Psi\in \{\Psi_{p,q},\Psi_{p,q}^\Delta\}$. Then for $|\alpha|\leq r$, $\|\prod_{j=1}^{p+q}\langle v_j\rangle^\alpha \sup_{0\le s\le Z}\Psi(s,v)\| 
        + \|\sup_{0\le s\le Z}\partial^\alpha\Psi(s,v)\| \le \sfc$.
\end{lemma}
\begin{proof} From the equation for $\Psi$ and noting that $\langle
     v_m\rangle, \langle v_m+k\rangle\le 2\langle v_m\rangle\langle k\rangle$, $ \|\prod_{j=1}^{p+q}\langle v_j\rangle^\alpha\Psi(z,v)\|\le\|\prod_{j=1}^{p+q}\langle v_j\rangle^\alpha\Psi(0,v)\|+(p+q)^2\int_0^z\int_{\mathbb{R}^d} \! \hat{R}(k)\langle k\rangle^{2\alpha}\|\prod_{j=1}^{p+q}\langle v_j\rangle^\alpha\Psi(s,v)\|dkds$.
    From Gr\"{o}nwall's inequality \tb{ and smoothness assumptions on the medium in~\eqref{eq:intcorr}}, it follows that $\sup_{0\le s\le Z}\|\prod_{j=1}^{p+q}\langle v_j\rangle^\alpha\Psi(s,v)\| \le \sfc$. The bound for $\Psi$ now follows as for $0\le z\le Z$, $|\Psi(z,v)|\le |\Psi(0,v)|+\sum_m\int_0^Z\int_{\mathbb{R}^d}\hat{R}(k)|\Psi_{p,q}(s,v-A_mk)|dkds$ where $m$ sums over all the terms in the definition~\eqref{eqn:L_pq_def} of $\mathcal{L}_{p,q}$ for some linear operator $A_m$. Derivatives are bounded in a similar manner using standard regularity estimates.
\end{proof}

\subsection{\tb{Strong} estimates for the splitting scheme}
\label{sec:pathsplittingIS}

We now prove the first estimate in Theorem~\ref{thm:pathwise_error} for \tb{$\sfv=\us$, where $u$ and its time discretization $u^\Delta$ are} solutions of \eqref{eqn:Ito} and \eqref{eqn:Ito_split}, respectively. Define
$\mG(t)=\mF^{-1} e^{-it|\xi|^2}\mF$ as in \eqref{eq:mG} so that
\[ \sfv(z,x) = \mG(\chi_0(z)) u_0 (x) + \int_0^z \mG(\chi_s(z)) (\mS \sfv)(s),\quad \mS \sfv(s) :=-\frac {R(0)}2 \sfv ds + i \sfv dB(s)\]
with a similar expression for $\ud$ where $\chi_s(z)$ is replaced by $\chid_s(z)$. Thus, with $E=\us-\ud$:
\[ E(z) = \int_0^z \mG(\chid_s(z)) \mS E(s) + \delta\mG(0,z)u_0
+\int_0^z \delta\mG(s,z) \mS u (s),\qquad \delta\mG(s,z) := \mG(\chi_s(z))-\mG(\chid_s(z)),\]
which implies for some universal constant $\sfc$:
\[ \frac 1\sfc |E(z)|^2 \leq \Big| \int_0^z \! \!\mG(\chid_s(z)) E ds\Big|^2 +\Big| \int_0^z \! \!\mG(\chid_s(z)) E dB\Big|^2 +  |\delta\mG(0,z)u_0|^2 +\Big| \int_0^z \!\! \delta \mG(s,z) u ds\Big|^2 +\Big| \int_0^z \!\!  \delta \mG(s,z) u dB\Big|^2.  \]
\tb{The strategy is to} integrate in $x$ over $\Rm^d$ and take ensemble average of the above expression to obtain 
\begin{eqnarray}\label{eqn:Expec_E}
\E \|E(z)\|^2_2  \leq  \sfc \Big( \int_0^z \E \|E(s)\|^2_2 ds + (\Delta z)^2\Big)\,.    
\end{eqnarray}
The proof~\tb{of Theorem 2.1 with $\sfv=u$} then follows from the Gr\"onwall inequality. The~\tb{rest of this section is devoted to the proof of the} above inequality~\eqref{eqn:Expec_E}. 

First, from \eqref{eq:errorchi}, we have $|\chi_0(z)-\chid_0(z)|\leq \aver{z}\|\kappa_1\|_{1,\infty} \Delta z$ so that by $\delta \mG(s,z) =\mF^{-1} (e^{-i\chi_s(z)|\xi|^2}-e^{-i\chid_s(z)|\xi|^2})\mF$, then $\mathbb{E}\|\delta\mG(0,z)u_0\|^2_2 \le (\aver{z}\|\kappa_1\|_{1,\infty})^2(\Delta z)^2
\||\xi|^4|\hat{u}_0|^2(\xi)\|\leq \sfc (\Delta z)^2$.

Combining the Parseval relation, the \tb{Cauchy-Schwarz} inequality, and the above regularity result yields
\[
\mathbb{E}\|\int_0^z \delta \mG(s,z) u(s,x) ds\|^2_2\le (\aver{z}^2\|\kappa_1\|_{1,\infty})^2(\Delta z)^2\sup_{0\le s\le z}\||\xi|^4\hat{\mu}_{1,1}(s,\xi,\xi)\| \leq \sfc (\Delta z)^2.
\]
In the second inequality, we have used the regularity of the second moment from Lemma~\ref{lem:regulmu11} below with $2r=4$. 
From the unitarity of $\mathcal{G}$ and the Parseval relation, we have
\begin{equation}\label{eq:errorParseval}
\mathbb{E}\|\int_0^{z}\mathcal{G}\big(\chid_s(z)\big)[E(s,x)]ds\|^2_{2}= \mathbb{E}\int_{\mathbb{R}^d}\big|\int_0^z\tb{e^{-i\chid_s(z)|\xi|^2}}\hat{E}(s,\xi)ds\big|^2\frac{d\xi}{(2\pi)^d}\le z\int_0^z\mathbb{E}\|E(s,x)\|^2_{2}ds.
\end{equation}
From the martingale property of the stochastic integral and the It\^o formula, we also have \cite{miyahara1982stochastic}
\begin{equation}\label{eqn:dB_conc}
\begin{aligned}
\mathbb{E}\|\int_0^z\mathcal{G}\big(\chid_s(z)\big)[E(s)dB(s)]\|^2_{2}&=\int_0^z\int_{\mathbb{R}^{2d}}\mathbb{E}|\hat{E}(s,\xi-k)|^2\frac{\hat{R}(k)}{(2\pi)^d}dkd\xi ds
= R(0) \int_0^z\mathbb{E}\|E(s)\|^2_{2}ds.
\end{aligned}
\end{equation}
Finally, the bound for the last term $\mathbb{E}\|\int_0^z[\delta\mathcal{G}(s,z)u(s,x)dB(s,x)]\|^2_{2}$ is (with $dq=\frac{1}{(2\pi)^{3d}}d\xi dk dk' ds ds'$)
\begin{equation*}
\begin{aligned}
&\mathbb{E}\int_{\mathbb{R}^{3d}}\! \int_{0}^z \!\! \int_0^z \!
[e^{-i \chi_s(z)|\xi|^2} - e^{-i \chid_s(z)|\xi|^2}]
[e^{i \chi_{s'}(z)|\xi|^2} - e^{i \chid_{s'}(z)|\xi|^2}]
\hat{u}(s,\xi-k)\hat{u}^\ast(s',\xi-k')d\hat{B}(s,k)d\hat{B}^\ast(s',k')
dq
\\
=&\int_{\mathbb{R}^{2d}}\int_0^z|e^{-i \chi_s(z)|\xi|^2} - e^{-i \chi^\Delta_s(z)|\xi|^2}|^2\mathbb{E}|\hat{u}(s,\xi-k)|^2\hat{R}(k)\frac{dsdkd\xi}{(2\pi)^{2d}}
\\
\le& \int_{\mathbb{R}^{2d}}\int_0^z|\chi_s(z)-\chid_s(z)|^2|\xi|^4\mathbb{E}|\hat{u}(s,\xi-k)|^2\hat{R}(k)dsdkd\xi
\le \sfc\|\langle k\rangle^4\hat{R}(k)\|(\Delta z)^2\sup_{0\le s\le z}\|\langle \xi\rangle^4\hat{\mu}_{1,1}(s,\xi,\xi)\|\,.
\end{aligned}
\end{equation*}
This concludes the proof of the first part in Theorem~\ref{thm:pathwise_error} for $\sfv=u$.

We note that due to the concentration property of correlations of the measure $dB$ as in~\eqref{eqn:dB_conc}, we do not expect to obtain convergence rates better than $O(\Delta z)$ even for a centered splitting scheme. However as is shown next, the statistical moments of such approximations are still sufficiently regular which makes higher order discretization schemes possible.

\subsection{Convergence of moments for the splitting scheme}
This section is devoted to the proof of the first estimate in Theorem \ref{thm:momentestim} for $\sfv=\us$. We define $\Psidpq(z,v)=\Psi_{p,q}(z,v)+\Ed(z,v)$ to obtain that 
\begin{equation*}
    \partial_z \Ed=\mathcal{L}_{p,q}[\chid_0(z)]\Ed(z,v)  +  (\mathcal{L}_{p,q}[\chi_0(z)]-\mathcal{L}_{p,q}[\chid_0(z)])\Psi_{p,q}(z,v),\quad \Ed(0,v)=0,
\end{equation*}
which \tb{in integral form is} $$\Ed(z,v) = \int_0^z \mathcal{L}_{p,q}[\chid_0(s)]\Ed(s,v) ds + \int_0^z (\mathcal{L}_{p,q}[\chi_0(s)]-\mathcal{L}_{p,q}[\chid_0(s)])\Psi_{p,q}(s,v)ds.$$

We observe that $\|\mathcal{L}_{p,q}[\chid_0(z)]\|$ is uniformly bounded by a constant times $R(0)(p+q)^2$. As an application of the Gr\"onwall inequality, we thus obtain that $\sup_{z\in[0,Z]}\|\Ed(z)\| \leq \sfc  
 \sup_{z\in[0,Z]} \big\| \int_0^z (\mathcal{L}_{p,q}[\chi_0(s)]-\mathcal{L}_{p,q}[\chid_0(s)])\Psi_{p,q}(s,v)ds\big\|$.
We next observe that 
\[ \int_0^z ds (\mathcal{L}_{p,q}[\chi_0(s)]-\mathcal{L}_{p,q}[\chid_0(s)]) \Psi_{p,q}(s,v) = 
\dsum_{m} \dint_{\Rm^{d}} dk \hat R(k) \int_0^z ds  (e^{i\chi_0(s)G_m}-e^{i\chid_0(s)G_m})\Psi_{p,q}(s,v-A_mk),
\]
for some matrices $A_m$ (with two non-vanishing coefficients) and real-valued phases $G_m$ that are at most quadratic in $(v,k)$, and where $m$ runs over all $(p,q)$ terms in the definition \eqref{eqn:L_pq_def}.

For general \tb{$\gamma\in[0,1)$}, the above term is bounded by $O(\Delta z)$. Assume now that $\gamma=\frac12$ and apply Lemma \ref{lem:Strang} with $\psi(s)=\Psi_{p,q}(s,v-A_mk)$ and $\phi=G_m$ to obtain a bound in total variation for each $m$ given by
\[
 \langle Z\rangle\dint_{\Rm^{(p+q)d}} dv \dint_{\Rm^{d}} dk \hat R(k) \sup_s \big( |G_m||\partial_s\Psi_{p,q}|+|G_m|^2|\Psi_{p,q}|\big)(s,v-A_mk) .
\]
We verify that $|G_m| \leq \sfc(\aver{v}^2+\aver{k}^2) \leq \sfc (\aver{v-A_mk}^2+\aver{k}^2)$ and $|G_m|^2 \leq \sfc (\aver{v-A_mk}^4+\aver{k}^4)$.
Since $\partial_z \Psi_{p,q} = \mathcal{L}_{p,q}[\chi_0(z)]\Psi_{p,q}(z,v)$, we obtain similar bounds for both terms so that
\[
\sup_{z\in[0,Z]}\|\Ed(z)\| \leq \sfc (\Delta z)^2 \int_{\Rm^d} \aver{k}^4 \hat R(k) dk \int_{\Rm^{(p+q)d}} \aver{v}^4 \sup_s |\Psi_{p,q}|(s,v) dv \leq \sfc (\Delta z)^2
\]
by Lemma \ref{lemma:Psi_pq_bound}. This concludes the proof of the first part of Theorem~\ref{thm:momentestim} for $\sfv=u$.
\subsection{Error estimates for discrete, periodic random media}
In this section, we prove the second part in Theorems~\ref{thm:pathwise_error} and~\ref{thm:momentestim}  for $\sfv\in \{\us,\ud\}$. The wavefield in the discretized medium solves $dv_c=i\phi \Delta_x \sfv_c dz -\frac12 R_c(0) \sfv_c dz + i\sfv_c dB_c$ where $\phi=\kappa_1$ for $\sfv=u$ and $\phi=\tau_\gamma\kappa_1$ for $\sfv=\ud$ with
\[ \E d\hat B_c(s,\xi) d\hat B_c(t,\zeta) = \delta(t-s) \hat R_c(\xi,\zeta),\qquad \hat R_c(\xi,\zeta):= \chi_c^2(\xi) \delta(\xi-\zeta) \delta(\xi-q_\xi) \int_{\square_{q_\xi}} \hat R(\zeta) d\zeta.\]
So the difference $E=\sfv-\sfv_c$ satisfies with $\Phi_s(z)=\int_s^z \phi(t)dt$:
\[E(z,x) = \int_0^z \mG(\Phi_s(z)) [E(s) dB(s) + i\sfv_c(s)(dB-dB_c)(s) -\frac12 R(0) E(s)] ds.\]
We neglected the term $R_c(0)-R(0)=O(K_k^{-N})$ of comparable order to the third estimate in Theorem~\ref{thm:momentestim} to simplify the presentation. As in the preceding section this shows that $\E\| E \|^2_2$ is bounded by the following terms. Let $\mF_s$ be the filtration generated by $B(s)$, which is finer than that generated by $B_c(s)$ so that both $v$ and $\sfv_c$ are adapted for this filtration. As a consequence, we have for instance as in \eqref{eq:errorParseval} the following inequality that $\tb{\mathbb{E}}\big\|\int_0^z \mG(\Phi_s(z)) E(s) dB(s) \big\|_2^2
 \leq \sfc z  \int_0^z \E\| E(s) \|^2_2 ds$.
The term involving $R(0)$ is treated similarly. The norm of the second term is by the Parseval relation
\[\begin{aligned}
    &\E \int_0^z\int_0^z\dint_{\Rm^{3d}} e^{i(\Phi_s(z)-\Phi_{t}(z))|\xi|^2} \hat \sfv_c(s,\xi-k) \hat \sfv_c^*(t,\xi-q) d(\hat B-\hat B_c)(s,k) d(\hat B-\hat B_c)(t,q) dk d\xi dq  
    \\
    = &\int_0^z \int_0^z \dint_{\Rm^{3d}} \E[\hat \sfv_c(s,\xi-k) \hat \sfv_c^*(s,\xi-q)] \E [d(\hat B-\hat B_c)(s,k) d(\hat B-\hat B_c)(t,q)] dk d\xi dq = : I_c
\end{aligned}
\]
where we evaluated $\sfv$ and $\sfv^\ast$ and the complex exponentials at the same location $s$ anticipating the $\delta(t-s)$ contribution in 
\[
\begin{aligned}
    &\E [d(\hat B-\hat B_c)(s,k) d(\hat B-\hat B_c)(t,q)] = (2\pi)^d\delta(t-s) \Big( \hat R(k) \delta(k-q) \\
    + &\chi_c^2(k) \delta(k-q_k)\delta(q-q_k) \int_{\square_{q_k}} \hat R(\zeta)d\zeta - \chi_c(k) \delta(q_q-k) \hat R(q) - \chi_c(q) \delta(q_k-q) \hat R(k) \Big).
\end{aligned}
\]
Here $q_k$ denotes the center of the cube containing $k$. Then 
\[\begin{aligned}
    I_c&=\int_0^z\int_{\mathbb{R}^{2d}}\hat{R}(k)\big(\hat{\mu}_{1,1}[\sfv_c](s,\xi-k,\xi-k)-\hat{\mu}_{1,1}[\sfv_c](s,\xi-k,\xi-q_k)\chi_c(q_k)\big)dkd\xi ds\\
    &+\int_0^z\int_{\mathbb{R}^{2d}}\hat{R}(k)\big(\hat{\mu}_{1,1}[\sfv_c](s,\xi-q_k,\xi-q_k)-\hat{\mu}_{1,1}[\sfv_c](s,\xi-q_k,\xi-k)\chi_c(q_k)\big)\chi_c(q_k)dkd\xi ds\,.
    \end{aligned}
\]
The first integral can be bounded by $O\big((\Delta k))^2+K_k^{-N}\big)$ provided that $\E[\hat \sfv_c(s,\xi-k) \hat \sfv_c^*(s,\xi-k)]$ is integrable in $\xi$ with values in $C^2$ in the variable $k$. This is a consequence of \eqref{eq:concentrationmu11} in Lemma \ref{lem:regulmu11} below for \tb{$r\ge 1$} and by noting that for $\psi(\xi,k)\in C^2(\mathbb{R}^{2d})$, $\hat{R}(q_k)\cdot(k-q_k)$ and $\partial\psi(\xi-q_k,\xi-q_k)\cdot(k-q_k)$ integrate to $0$ on the cube $\square_{q_k}$ due to the mid point rule and expanding 
\tb{
\begin{equation*}
\begin{aligned}
 &\hat{R}(k)\big(\psi(\xi-k,\xi-k)-\psi(\xi-k,\xi-q_k)\big)=\big(\hat{R}(k)-\hat{R}(q_k)\big)\big(\psi(\xi-k,\xi-k)-\psi(\xi-k,\xi-q_k)\big)\\
&+\hat{R}(q_k)\big(\psi(\xi-k,\xi-k)-\psi(\xi-q_k,\xi-q_k)+\psi(\xi-q_k,\xi-q_k)-\psi(\xi-k,\xi-q_k)\big)   .    
\end{aligned}
\end{equation*}
}
The second integral is dealt with in a similar manner.
Combined with the above estimates and a Gr\"onwall inequality, we obtain the second part of the strong estimates in Theorem~\ref{thm:pathwise_error} for $\sfv\in\{\us,\ud\}$.

It remains to extend this result to the moment problem. 
Our analysis includes the case $\hat R_c(\xi,\zeta):= (2\pi)^d\chi_c^2(\xi) \delta(\xi-\zeta) \delta(\xi-q_\xi) |\square_{q_\xi}| \hat R(q_\xi)$, as implemented numerically as they have the same approximation properties (see Remark~\ref{rem:nu_c_disc}). Define $\Psi_{c,p,q}$ as we did $\Psi_{p,q}$ with $\hat R(k)$ replaced by $\hat R_c(k)$.
Let $\Psi=\Psi_{p,q}$ and $\Psi_c=\Psi_{c,p,q}$ as well as $\mL=\mL_{p,q}[\chi_0]$ while $\mL_c$ is defined as $\mL$ with $\hat R(k)$ replaced by $\hat R_c(k)$. Then $ \partial_z(\Psi-\Psi_c) = (\mL-\mL_c)\Psi+ \mL_c(\Psi-\Psi_c)$ 
so that we need to bound in the TV sense the source term 
\[ (\mL-\mL_c)\psi = \int_{\Rm^d} (\hat R(k)-\hat R_c(k)) (L_{p,q}[\chi_0(z)] \Psi (z,v;k,k))\frac{dk}{(2\pi)^d}.\]
It thus remains to show that $L_{p,q}[\chi_0(z)] \Psi (z,v;k,k)$ is integrable in $v$ when taking values in $C^2$ functions in $k$. From the form of $L_{p,q}$ and the quadratic nature of $g(\xi,k)$, this term is bounded by Lemma~\ref{lemma:Psi_pq_bound} with $r=4$.  This shows the second estimate in Theorem~\ref{thm:momentestim} for $\sfv\in\{\us,\ud\}$.

\subsection{Spatial concentration and periodization estimates} \label{sec:periodIS}
We now aim to prove the concentration and periodization estimate in the third parts of Theorem~\ref{thm:pathwise_error} and \ref{thm:momentestim} for $\tb{\sfv}\in\{\us,\ud\}$. We have the following estimate on the spatial concentration of the second moments.
\begin{lemma}\label{lem:regulmu11}
    Let $\mu(z,x,y)=\mu_{1,1}[\sfv](z,x,y)$ for $\sfv\in \{\us,\ud,\uc,\udc\}$ and $\hat \mu(z,\xi,\zeta)$ be its Fourier transform. Then for $r>0$ and \tb{\eqref{eqn:u0_bound}-\eqref{eq:intcorr}} valid for $M$ sufficiently large, we have
    \begin{equation}\label{eq:concentrationmu11}
        \sup_{0\leq z\leq Z}  \Big( \tb{\|\aver{x}^{2r}\mu(z,x,x)\|_{L^\infty(\mathbb{R}^d)}} +  \|\aver{\xi}^{2r} \hat\mu(z,\xi,\xi)\| 
        + \| \! \! \sum_{|\alpha|+|\beta| \leq 2r} \!\!\!\!  \big( |\partial^\alpha_\xi\partial^\beta_\zeta \hat \mu| + \aver{\xi}^{|\alpha|} \aver{\zeta}^{|\beta|} |\hat \mu| \big) \| \Big) \leq \tb{\sfc(r)}.
    \end{equation}
\end{lemma}
\begin{proof} 
From Parseval's theorem, the first and third estimates are equivalent. In all instances of $\sfv \in \{\us,\ud\}$, we have $ \partial_z \mu (z,x,y) = i\phi(z) (\Delta_x-\Delta_y) \mu + (R(x-y)-R(0)) \mu$ where $\phi=\kappa_1$ for $\sfv=u$ and $\phi=\tau_\gamma\kappa_1$ for $\sfv=\ud$ so that in the Fourier domain and for $\Phi(z)=\int_0^z \phi(s)ds$ we have
\[ 
\hat \mu(z,\xi,\zeta) = e^{i\Phi(z)(|\xi|^2-|\zeta|^2)} \hat \mu(0,\xi,\zeta) + \int_0^z \int_{\Rm^d} \hat R(k) e^{i\Phi(z-s)(|\xi|^2-|\zeta|^2)} (\hat \mu(s,\xi-k,\zeta-k)-\hat \mu(s,\xi,\zeta)) \frac{dk}{(2\pi)^d}.
\]
Using that $\aver{\xi}\leq \aver{\xi-k}\aver{k}$ and that $\aver{k}^N\hat R(k)$ is integrable for $N$ sufficiently large, we find that $I(z) \leq \sfc \Big(1+ \int_0^z I(s) ds\Big) $, $ 
I(s):= \| \sum_{|\alpha|+|\beta| \leq 2r} \big( |\partial^\alpha_\xi\partial^\beta_\zeta \hat \mu| + \aver{\xi}^{|\alpha|} \aver{\zeta}^{|\beta|} |\hat \mu| \big) \| $. The third term is bounded by Gr\"{o}nwall's inequality. Evaluating at $\zeta=\xi$ gives the second estimate in \eqref{eq:concentrationmu11} by the same mechanism. The proof for $\sfv=\{u_c,u^\Delta_c\}$ is similar, with $\hat{R}(k)dk$ replaced by the discrete measure $\sum_{q\in\mathbb{Z}_\Delta^d}\hat{R}(k)\delta(k-q_k)\chi_c^2(k)$.
\end{proof} 

We know from \eqref{eq:concentrationmu11} in Lemma \ref{lem:regulmu11} that the corresponding moment $|\mu_{1,1}[\sfv](z,x,x)|\leq \sfc \aver{x}^{-2N}$ uniformly in $z\in [0,Z]$. This implies that $\aver{x}^{2N}\E|\sfv|^2\leq c$ uniformly in $(z,x)\in[0,Z]\times\Rm^d$.

We recall the periodization of $\sfv$ on $\Tm_L^d$ given by $\sfv_\sharp(z,x) = \sum_{n\in \mathbb{Z}^d} \sfv(z,x-nL)$, $ \sfv_\sharp(z,x)-\sfv(z,x) = \sum_{0\not=n\in \mathbb{Z}^d} \sfv(z,x-nL)$. The above decay implies that 
\begin{equation}\label{eq:errorwsharp}
\E |\sfv_\sharp-\sfv|^2(z,x)  \leq  \sum_{m\not=0}\sum_{n\not=0} \frac \sfc{\aver{x-mL}^{N}\aver{x-nL}^{N}}\leq \sfc L^{-2N}\,.
\end{equation}
Using the Cauchy-Schwarz inequality and provided that $N>d$, after integration over $\Xm=\Tm_L^d$ we obtain the third estimate in Theorem~\ref{thm:pathwise_error} for $\sfv\in \{\us,\ud\}$.

Before presenting the proof of the moment estimates in the third part of Theorem~\ref{thm:momentestim}, we verify the following lemma. \tb{We recall that $(X,Y)=(x_1,\cdots,x_p,y_1,\cdots,y_q)$ is a collection of $p+q$ points in $\mathbb{R}^d$.}
\begin{lemma}\label{lem:reg_torus}
    Let $\sfv\in\{u,\ud\}$. \tb{Let $r\in\Nm$, \eqref{eqn:u0_bound}  hold with $N=r$ and \eqref{eq:intcorr} hold with $M=r$.} For $|\alpha|\le r$, the spatial moments satisfy $\sup_{0\le s\le Z}|\partial^\alpha\mu_{p,q}[\sfv_\sharp](z,X,Y)|\le \sfc$.
    \begin{proof}
         Let $\phi=\kappa_1$ for $\sfv=u$ and $\phi=\kappa_1\tau_\gamma$ when $\sfv=\ud$ with antiderivative $\Phi_s(z)=\int_s^z\phi(t)dt$. Expanding $\sfv_\sharp$ in terms of its discrete Fourier coefficients as $\sfv_\sharp(z,x)=\sum_{l\in\mathbb{Z}^d}\hat{\sfv}_l(z)e^{i\Delta k l\cdot x}$, we find that $\hat{\sfv}_l$ satisfies the equation
\[
d\hat{\sfv}_l=-i\phi(z)|\Delta k l|^2\hat{\sfv}_ldz-\frac{R_c(0)}{2}\hat{\sfv}_ldz+i\sum_{k\in\mathbb{Z}^d}\hat{\sfv}_{l-k}d\hat{B}_{c,k}\,,
\]
for the discrete Fourier coefficients $d\hat{B}_{c,l}(z)=d\hat{B}_c(z,\Delta k l)$. For $\sfl=(l_1,\cdots,l_{p+q})$, $l_j\in\mathbb{Z}^{d}$ we define $\hat{\mu}_{\sfl}(z)=\mathbb{E}\prod_{m=1}^p\hat{\sfv}_{l_m}(z)\prod_{n=p+1}^{p+q}\hat{\sfv}^\ast_{l_n}(z)$. 
Also, let $\epsilon_m=1$ when $1\le m\le p$ and $\epsilon_m=-1$ when $p+1\le m\le p+q$. As in~\eqref{eqn:psi_pq_def}, we define the phase compensated coefficients $\Psi_\sfl(z)=\hat{\mu}_\sfl(z)e^{i\Phi_0(z)\sum_{m=1}^{p+q}\epsilon_m|\Delta kl_m|^2}$. Then $\Psi_\sfl$ satisfies
\[
\partial_z\Psi_\sfl=-\sum_{m=1}^{p+q}\sum_{m\neq n=1}^{p+q}\epsilon_m\epsilon_n\sum_{j\in\mathbb{Z}^d}\hat{R}_j\Psi_{l_m-\epsilon_mj,l_n+\epsilon_nj}e^{i\Phi_0(z)|\Delta k|^2[\epsilon_mg(l_m,\epsilon_m j)+\epsilon_ng(l_n,-\epsilon_n j)]}-\frac{p+q}{2}R_c(0)\Psi_\sfl\,,
\]
with initial condition $\Psi_\sfl=\hat{\mu}_\sfl(0)$. We observe that the right hand side is absolutely summable with weight $\langle\Delta k \sfl\rangle^{|\alpha|}$, which gives $\sup_{0\le s\le Z}\sum_{\sfl\in\mathbb{Z}^{(p+q)d}}\langle \Delta k\sfl\rangle^{|\alpha|}|\Psi_\sfl|(s)\le \sfc$. The bound for $\partial^\alpha\mu_{p,q}[\sfv_\sharp]$ now follows after summing the discrete Fourier coefficients.
    \end{proof}
\end{lemma}

Again, for $\sfv\in\{u,\ud\}$ consider the moments $\mu_{p,q}[\sfv](z,X,Y)$ and $\mu_{p,q}[\sfv_\sharp](z,X,Y)$. From Lemma~\ref{lem:reg_torus}, the latter moment satisfies an equation over $\Tm_L^d$ with bounded TV norm in the Fourier domain implying that $|\mu_{p,q}[\sfv_\sharp](z,X,Y)|\leq \sfc$ uniformly in $(X,Y)$ and $z\in [0,Z]$.
Now
\begin{equation}\label{eq:decmoments}
|\mu_{p,q}[\sfv]-\mu_{p,q}[\sfv_\sharp]|(z,X,Y) \leq \sum_{j=1}^{p+q}  \Big| \E \prod_{k=1}^{j-1}w(x_k) (w-w_\sharp)(x_j) \prod_{k=j+1}^{p+q} w_\sharp(x_k) \Big|
\end{equation}
where we denote by $x_k$ all variables in $(X,Y)$ and with $w=\sfv$ or $\sfv^*$ depending on the considered variable. This implies $|\mu_{p,q}[\sfv]-\mu_{p,q}[\sfv_\sharp]|(z,X,Y) \leq \sum_j (\E |\prod_{k=1}^{j-1}\sfv(x_k)|^4)^{\frac14} (\E |\prod_{k=j+1}^{p+q} \sfv_\sharp(x_k)|^4)^{\frac14} (\E|\sfv-\sfv_\sharp|^2(x_j))^{\frac12} \leq \sfc(p+q) L^{-N}$. 
This proves the third part of Theorem~\ref{thm:momentestim} for $\sfv=\{u,\ud\}$.

\subsection{Spatial discretization of solution} 
This section obtains the fourth estimates in Theorems~\ref{thm:pathwise_error} and \ref{thm:momentestim} for $\sfv=u$ and $\sfv=\ud$.  Consider $\sfv_\sharp\in\{\up,\udp\}$. The fully discretized solution $\sfv_\delta$ is defined in \eqref{eq:discreteeq}. We find
\[
 d\Pi \sfv_\sharp = i \phi \Delta \Pi \sfv_\sharp dz + \Pi (\sfv_\sharp \circ dB_c)
 =  i \phi \Delta \Pi \sfv_\sharp  dz+ \Pi (\Pi\sfv_\sharp \circ dB_c) + \Pi ((I-\Pi)\sfv_\sharp \circ dB_c)
\]
using $\sfv\circ dB_c\equiv-\frac12 R_c(0)\sfv dz+ \sfv dB_c$. Defining $\delta \sfv=\Pi\sfv_\sharp-\sfv_\delta$, we thus have
\[ 
d \delta \sfv =i \phi \Delta \delta \sfv dz + \Pi (\delta \sfv \circ dB_c)  + \Pi ((I-\Pi)\sfv_\sharp \circ dB_c).
\]
\tb{As in earlier sections, we obtain for $I_\delta(z):=\E \int_{\Tm_L^d} |\delta \sfv(z,x)|^2 dx$ the following estimate:}
\[
I_\delta(z) \leq \sfc \int_0^z I_\delta(s) ds + \E \int_{\Tm^d_L} \Big| \int_0^z \mG_\sharp[\Phi_{s}(z)] \Pi ((I-\Pi)\sfv_\sharp \circ dB_c)\Big|^2 dx.
\]
Here $\mG_\sharp[\Phi_s(z)]$ is the solution operator of $(\partial_z-i\phi \Delta_x)$ on $[s,z]\times\Tm_L^d$. By unitarity of $\mG$ and contraction of $\Pi$, as well as adaptivity of $\sfv_\sharp$ for the filtration generated by $B_c(z)$, we deduce as in Section \ref{sec:pathsplittingIS} that the above source term is bounded by a constant times $\E\|(I-\Pi)\sfv_\sharp\|_2^2$. We compute
\[
\E\int_{\Tm^d_L} |(I-\Pi)\sfv_\sharp|^2  dx\leq\sum_{|l|>N_x} \E |\hat \sfv_l|^2 \leq (\Delta k N_x)^{-2N} \sum_{l} \E \aver{\Delta kl}^{2N} |\hat \sfv_l|^2\leq  \sfc (\Delta x)^{2N},
\]
since $\E \|\partial^\alpha \sfv_\sharp\|^2_{L^2(\Tm_L^d)}\leq \sfc$ for $|\alpha|\leq N$ from Lemma~\ref{lem:reg_torus}. This proves the final estimate in Theorem~\ref{thm:pathwise_error} for $\up$ and $\udp$ and their fully discrete approximations. The extension to errors for moments in the fourth part of Theorem~\ref{thm:momentestim} related to $\up$ and $\udp$ then follows the same steps as for~\eqref{eq:errorwsharp} and \eqref{eq:decmoments}.
\section{Convergence results for the paraxial equation}
\label{sec:paraxial}

In this section, we prove the estimates presented in Theorems \ref{thm:pathwise_error} and \ref{thm:momentestim} for the paraxial model $\sfv=\ut$.   Unlike the \IS\ model, statistical moments of the paraxial solution $\ut$ do not satisfy closed form equations as in \eqref{eqn:mu_pq_PDE}. However, such equations are {\em almost} satisfied when $\theta\ll1$ \cite{bal2024long}, which intuitively explains why we should expect similar results. Since no closed form equation is available, we construct and analyze Duhamel expansions for such moments. As in the \IS\ model, they form the main technical representation we use to derive our convergence results. Unlike the \IS\ model, these Duhamel expansions need to be defined directly for products of random fields rather than their ensemble averages.

%
%
%
\paragraph{Duhamel expansion} We recall that  $\ut(z,x)$ is the solution of the paraxial model \eqref{eqn:PWE}.
Introduce the notation $\sss=(s_1,\ldots,s_n)$, $\kkk=(k_1,\ldots,k_n)$ and 
\begin{equation}\label{eq:notationDuhamel}\begin{aligned}
    &d\sss:=\prod_j d s_j,\quad  d\kkk :=\prod_j \frac{dk_j}{(2\pi)^d},\quad
  \sk_j(\xi,\kkk):= \xi- \sum_{l=1}^j k_l,\quad \hat \mV^\theta(\sss,\kkk):=\prod_{j=1}^n \hat \nu^\theta(s_j,k_j)  
  \\
  &\rG_n(z,\sss,\xi,\kkk) := \frac {n\pi}2+ \sum_{j=0}^n \chi_0(s_j)(|\sk_{j-1}|^2-|\sk_{j}|^2),\qquad \mI_n(z,\sss,\xi,\kkk):= e^{i\rG_n(z,\sss,\xi,\kkk)} \hat u_0(\sk_n(\xi,\kkk))\,,
\end{aligned}
\end{equation}
with $s_0=z$, $\sk_0(\xi,\kkk)=\xi$ and $\sk_{-1}(\xi,\kkk)=0$. We also define the simplex $[0,z]^n_<$  given by $0\le s_n\le\cdots\le s_1\le z$.
Let $\hut(z,\xi)$ be the (partial) Fourier transform of $\ut(z,x)$. We thus obtain the Duhamel expansion
\begin{equation}\label{eq:Duhamelhatu}
\hut(z,\xi) = \sum_{n\geq0} \hut_n(z,\xi),\qquad 
\hut_n(z,\xi) := \int_{[0,z]_<^n} \hspace{-.4cm} d\sss \int_{\Rm^{nd}} \hspace{-.3cm} d\kkk \
\mI_n(z,\sss,\xi,\kkk) \hat \mV^\theta(\sss,\kkk)\ \mbox{ for }\ n\geq1,
\end{equation}
and $\hat u^\theta_0(z,\xi)=e^{-i\chi_0(z)|\xi|^2}\hat u_0(\xi)$. The Duhamel expansion for $\hat{u}^{\theta\Delta}$ is defined in the same manner, with $\chi_0$ replaced by $\chid_0$. We extend the above Duhamel expansion to $(p,q)$ field products. Define the compensation phases $\varphi (z,v) = e^{i\chi_0(z)\big(\sum_{j=1}^p|\xi_j|^2-\sum_{l=1}^q|\zeta_l|^2\big)}$, $
\varphi^\Delta (z,v) = e^{i\chid_0(z)\big(\sum_{j=1}^p|\xi_j|^2-\sum_{l=1}^q|\zeta_l|^2\big)}$\tb{, where $v=(\xi_1,\cdots,\xi_p,\zeta_1,\cdots,\zeta_q)$}. We construct the phase compensated fields and their ensemble averages:
\begin{equation*}
\begin{aligned}
    \psi_{p,q}^{\theta}(z,v)= \varphi (z,v) \Big(\prod_{j=1}^p\hat{u}^\theta(z,\xi_j)\prod_{l=1}^q\hat{u}^{\theta\ast}(z,\zeta_l)\Big),\quad 
    & \psi_{p,q}^{\theta\Delta}(z,v)=\varphi^\Delta (z,v) \Big(\prod_{j=1}^p\hatutd(z,\xi_j)\prod_{l=1}^q\hatutdcon(z,\zeta_l)\Big)
    \\
    \Psi_{p,q}^{\theta}(z,v):=\mathbb{E}\psi_{p,q}^{\theta}(z,v) = \varphi(z,v)\hat{\mu}_{p,q}[\ut](z,v), \quad &
    \Psi_{p,q}^{\theta \Delta}(z,v):=\mathbb{E}\psi_{p,q}^{\theta\Delta}(z,v) = \varphi^\Delta(z,v)\hat{\mu}_{p,q}[\utd](z,v).
\end{aligned}
    \end{equation*}

The compensated moments solve the evolution equations $\partial_z{\psi}^\theta_{p,q}= \Ltpq[\chi_0(z),\nu^\theta_z]{\psi}^\theta_{p,q}$ and $\partial_z\psitd_{p,q}=
        \Ltpq[\chid_0(z),\nu^\theta_z]\psitd_{p,q}$, with the same initial condition ${\psi}^\theta_{p,q}(0,v)=\psitd_{p,q}(0,v)=\hat{\mu}_{p,q}(0,v)$ and with the definition $\nu^\theta_z(x):=\nu^\theta(z,x)=\frac{1}{\sqrt\theta}\nu\big(\frac{z}{\theta},x\big)$. The operator $\Ltpq$, parametrized by the phase $\phi$ and the random medium $\nu$ is defined by
\begin{equation}\label{eq:Ltpq}
    \Ltpq[\phi,\nu] \psi\,(z,v) = i\int_{\mathbb{R}^d}\hat{\nu}(k)\big[\sum_{j=1}^p\psi(\xi_j-k)e^{i\phi g(\xi_j,k)}-\sum_{l=1}^q\psi(\zeta_l+k)e^{-i\phi g(\zeta_l,-k)}\big]\frac{dk}{(2\pi)^d}
\end{equation}
where we recall $g(\xi,k)=|\xi|^2-|\xi-k|^2$. As in~\cite[Lemma 4.1]{bal2024long}, we have the following Duhamel expansion:
\begin{equation}\label{eqn:chaos_series2}
    \begin{aligned}       
    \psi^\theta_{p,q}(z,v) = \sum_{n\geq0} \psi^\theta_n(z,v),\qquad 
    \psi^\theta_n(z,v) :=\int_{[0,z]^n_<} \hspace{-.4cm} d\sss 
    \int_{\mathbb{R}^{nd}} \hspace{-.4cm} d\kkk\ 
    \sum\limits_{m=1}^{(p+q)^n}e^{iG_m(\sss,v,\kkk)}\hat\mu_{p,q}(0,v-A_m\Vec{k}) \, \hat \mV^\theta(\sss,\kkk),
    \end{aligned}
\end{equation}
with $\psi^\theta_0(z,v)=\hat \mu_{p,q}(0,v)$. $A_m$ are $(p+q)d\times nd$ block matrices of $d\times d$ blocks with exactly one non zero entry (either $1$ or $-1$) per column block. $G_m$ are real valued phases which we briefly recall here for completeness. 

Let \tb{$\Vec{r}=(r_1,\cdots,r_n)$} with $r_j\in\{1,\cdots,p+q\}$, $\epsilon_j=1$ if $1\le j\le p$ and $\epsilon_j=-1$ if $p+1\le j\le p+q$. Let $B_{r_1}=\epsilon_{r_1}\chi_0(s_1)g(v_{r_1},\epsilon_{r_1}k_1)$ and define recursively $B_{r_1,\cdots,r_j}(s_1,\cdots,s_j,v,k_1,\cdots,k_j)=\epsilon_{r_j}\chi_0(s_j)g(v_{r_j},\epsilon_{r_j}k_j)+B_{r_1,\cdots,r_{j-1}}(s_1,\cdots,s_{j-1},v_1,\cdots,v_{r_j}-\epsilon_{r_j}k_j,\cdots,v_{p+q},k_1,\cdots,k_{j-1})$. Finally for the $m$th combination determined by the entries of $\Vec{r}$, set $G_m(\sss,v,\kkk)$ to be $B_{\Vec{r}}(\sss,v,\kkk)+{\frac{\pi}{2}\sum_{j=1}^n\epsilon_{r_j}}$. Similarly set $A_m$ to be the $(p+q)d\times nd$ block matrix with the $j$th entry of the column block being $-\epsilon_{r_j}$.

The semi-discretized splitting field $\psi_{p,q}^{\theta\Delta}(z,v)$ is defined in the same way with $\chi_0(z)$ replaced by $\chid_0(z)$. We denote by $G_m^\Delta$ the corresponding phases appearing in the Duhamel expansion. Since $\nu$ is a mean-zero Gaussian field, then $\E \psi^\theta_n=\E\psi^{\theta\Delta}_n=0$ whenever $n$ is odd.  We start with the following lemmas.
\begin{lemma}\label{lem:estims}
    For each $0\leq l\leq n$, we have
    \begin{equation}\label{eq:controlkl} 
     |\sk_l|^2 \leq 2 |\xi-\sum_{j=1}^nk_j|^2 + 2 |\sum_{j=l+1}^n k_j|^2 \leq 2 |\sk_n|^2 +2(\sum_{j=1}^n |k_j|)^2\leq 2^{n+1}\aver{\sk_n}^2 \prod_{j=1}^n\aver{k_j}^2.
    \end{equation}
    For $1\le r\le n+m$, let $\hat{\nu}_{\epsilon(r)}$ be either $\hat{\nu}$ or $\hat{\nu}_c$. We have for any $n+m=2N$ and $M\ge 0$,
    \begin{equation}\label{eq:controlaver}  
        \int_{[0,z]^n_<}\hspace{-.4cm}d\sss\int_{[0,z]^m_<}\hspace{-.4cm}d\ttt\int_{\mathbb{R}^{nd}}\hspace{-.4cm}d\kkk\int_{\mathbb{R}^{md}}\hspace{-.4cm}d\qqq|\mathbb{E}\prod_{j=1}^n\langle k_j\rangle^M\hat{\nu}^\theta_{\epsilon(j)}(s_j,k_j)\prod_{l=1}^m\langle q_l\rangle^M\hat{\nu}^\theta_{\epsilon(n+l)}(t_l,q_l)|\leq \frac{z^N(2N)!!}{n!m!}\Big(\int_{\mathbb{R}}\|\langle k\rangle^{2M}\hat{C}(s,k)\|ds\Big)^N\,\tb{.}
    \end{equation}
    Moreover for constants $0\le \alpha_{1,2}\le \tb{\bar{\alpha}}$,
    \begin{equation}\label{eq:controlaver_sum}
    \begin{aligned}
     &\sum_{m,n\ge 0}\alpha_1^n\alpha_2^m\int_{[0,z]^n_<}\hspace{-.5cm}d\sss\int_{[0,z]^m_<}\hspace{-.5cm}d\ttt\int_{\mathbb{R}^{nd}}\hspace{-.4cm}d\kkk\int_{\mathbb{R}^{md}}\hspace{-.4cm}d\qqq|\mathbb{E}\prod_{j=1}^n\langle k_j\rangle^M\hat{\nu}^\theta_{\epsilon(j)}(s_j,k_j)\prod_{l=1}^m\langle q_l\rangle^M\hat{\nu}^\theta_{\epsilon(n+l)}(t_l,q_l)|\leq   e^{2z\tb{\bar{\alpha}}^2\int_{\mathbb{R}}\hspace{-0.05cm}\|\langle k\rangle^{2M}\hat{C}(s,k)\|ds}.
    \end{aligned}
            \end{equation}
\end{lemma}
\begin{proof} The first relation~\tb{\eqref{eq:controlkl}} is  immediate. 
We note that after integrating in $\kkk$ and $\qqq$,~\tb{thanks to the Gaussianity of the random potential,} the expectation in~\eqref{eq:controlaver} is invariant under permutations of the elements of $\sss$ and $\ttt$ which allows us to symmetrize the simplexes. Let $\vs=(\sss,\ttt)$ and $\vk=(\kkk,\qqq)$. Using the moment formula for Gaussian variables \cite{janson1997gaussian}, \eqref{eq:controlaver} is
\begin{equation}\label{eq:estimgeom}
    \begin{aligned}
&\int_{[0,z]^n_<}\hspace{-.4cm}d\sss\int_{[0,z]^m_<}\hspace{-.4cm}d\ttt\int_{\mathbb{R}^{nd}}\hspace{-.4cm}d\kkk\int_{\mathbb{R}^{md}}\hspace{-.4cm}d\qqq|\mathbb{E}\prod_{j=1}^n\langle k_j\rangle^M\hat{\nu}^\theta_{\epsilon(j)}(s_j,k_j)\prod_{l=1}^m\langle q_l\rangle^M\hat{\nu}^\theta_{\epsilon(n+l)}(t_l,q_l)|\\
&=\frac{1}{n!}\frac{1}{m!}\int_{[0,z]^{2N}}\hspace{-.4cm}d\vs\int_{\mathbb{R}^{2Nd}}\hspace{-.4cm}d\vk\prod_{j=1}^{2N}\langle \sfk_j\rangle^M|\sum_{P^{2N}}\prod_{(j,\jay)}\hat{C}^\theta_{j,\jay}(\sfs_j-\sfs_\jay,\sfk_j,\sfk_\jay)|,\quad \hat{C}^\theta_{j,\jay}=\mathbb{E}\hat{\nu}^\theta_{\epsilon(j)}(\sfs_j,\sfk_j)\hat{\nu}^\theta_{\epsilon(\jay)}(\sfs_\jay,\sfk_\jay)\,.
\end{aligned}
\end{equation}
 Here, $P^{2N}$ denotes all possible ways $\pi$ to construct $N$ pairs of elements out of $2N$ elements while $(j,\jay)$ above denote the corresponding pairings in $\pi$ resulting from Gaussian statistics \cite{janson1997gaussian}. The total number of pairings is given by $|P^{2N}|=(2N)!!=\frac{(2N)!}{N!2^N}$. 
Since in all pairings $\hat{C}^\theta_{j,\jay}$ has the same upper bound in total variation, we can bound the term above by
\[
\frac{|P^{2N}|}{n!m!}\int_{[0,z]^{2N}}\int_{\mathbb{R}^{Nd}}\prod_{j=1}^N\langle k_j\rangle^{2M}|\hat{C}^\theta(s_j-t_{j},k_j)|dk_jds_jdt_j\le\frac{|P^{2N}|}{n!m!}\Big(\int_{[0,z]^2}\int_{\mathbb{R}^d}\langle k_j\rangle^{2M}|\hat{C}^\theta(s-t,k)|dkdsdt\Big)^{N}\,.
\]
We arrive at~\eqref{eq:controlaver} after a change of variables $s-t\to \theta s$. We obtain the bound in~\eqref{eq:controlaver_sum} for the sum by observing that the expectation of terms with $m+n$ odd is zero so that for $m+n=2N$ and $\lambda=\tb{\bar{\alpha}}^2\int_{\mathbb{R}}\|\langle k\rangle^{2M}\hat{C}(s,k)\|ds$, 
\[
\sum_{N\ge 0}\sum_{n=0}^{2N}\frac{|P^{2N}|}{n!m!}\lambda^N=\sum_{N\ge 0}\sum_{n=0}^{2N}\frac{(2N)!}{N!n!(2N-n)!2^N}\lambda^N=\sum_{N\ge 0}\frac{2^{N}\lambda^N}{N!}\,.
\]
\end{proof}

We also have the following stability result for the statistical moments.
\begin{lemma}\label{lem:stabpsitpq}
    Let $\psi_{p,q}^{\theta}(z,v)$  and $\psi_{p,q}^{\theta\Delta}(z,v)$ be defined as above. Let $N\geq0$. \tb{Suppose ${\rm C}_{M,R}[u_0]$ in \eqref{eqn:u0_bound} and $\fC_{P,Q}[\nu]$ in \eqref{eq:intcorr} are bounded for $M\ge 0, R\ge N/2, P\ge (p+q)N$ and $Q\ge 0$. We} have $\sup_{0\leq z\leq Z} \| \prod_{j=1}^{p+q} \aver{v_j}^N \E\psi_{p,q}^{\theta\Delta}(z,v) \| \leq \sfc$ and $ \sup_{0\leq z\leq Z} \| \prod_{j=1}^{p+q} \aver{v_j}^N \E\psi_{p,q}^{\theta}(z,v) \| \leq \sfc$. 
\end{lemma} 
\begin{proof}
    We recall that $\|\cdot\|$ denotes the total variation norm of measures. As $\aver{v_j}\leq 2\aver{v_j-[A_r\kkk]_j}\aver{[A_r\kkk]_j}$ $\leq 2^{n+1}\aver{v_j-[A_r\kkk]_j}\prod_{l=1}^n \aver{k_l}$, $ |\E  \prod_{j=1}^{p+q}\aver{v_j}^N  \psi^\theta_{p,q}(z,v)| $ is upper bounded by
    \[
    \sum_{n\geq0}\sfc^n \sum_{r=1}^{(p+q)^n} \int_{[0,z]^n_<} \hspace{-.5cm} d\sss 
    \int_{\mathbb{R}^{nd}} \hspace{-.4cm} d\kkk\ 
   \prod_{j=1}^{p+q} \aver{v_j-[A_r\kkk]_j}^N |\hat\mu_{p,q}(0,v-A_r\Vec{k})| \Big( \prod_{j=1}^n \aver{k_j}^{(p+q)N} \Big)|\E \hat \mV^\theta(\sss,\kkk)|.
   \]
    Using \tb{the bounds in \eqref{eqn:u0_bound} and \eqref{eq:intcorr},} and \eqref{eq:controlaver} with $\alpha_2=0$ gives the result after integration in the $v$ variables. $\psi^{\theta\Delta}_{p,q}$ is bounded in the same way.
\end{proof}

\subsection{Strong stability estimates for the splitting scheme}
\label{sec:pathwisesplitting}
This section proves the first part in Theorem~\ref{thm:pathwise_error} for $\sfv=\ut$. Using the Parseval equality, this is equivalent to proving a bound for $\mathbb{E}\|\hat u^\theta(z,\xi)-\hat u^{\theta\Delta}(z,\xi)\|^2_{L^2(\mathbb{R}^d)}$. The Duhamel expansion for $\hat u^{\theta\Delta}(z,\xi)$ is the same as that for $\hat u^{\theta}(z,\xi)$ in \eqref{eq:Duhamelhatu} with the exception that $\chi_0(s)$ is replaced by $\chid_0(s)$. As a consequence, we find that
\[    (\hat u^\theta-\hat u^{\theta\Delta})(z,\xi)  
    = 
    \sum_{n\geq1}\sum_{l=0}^n 
    \int_{[0,z]_<^n} \hspace{-.4cm} d\sss\, 
    \int_{\Rm^{nd}} \hspace{-.3cm} d\kkk e^{i \tilde \rG_l(z,\sss,\xi,\kkk)} [e^{i\chi_0(s_l)(|\sk_{l-1}|^2-|\sk_l|^2)} -e^{i\chid_0(s_l)(|\sk_{l-1}|^2-|\sk_l|^2)}]
      \hat u_0(\sk_n) \hat \mV^\theta(\sss,\kkk)
\]
for real-valued phases $\tilde \rG_l=\frac{n\pi}{2}+\sum_{j=0}^{l-1}\chi_0(s_j)(|\sk_{j-1}|^2-|\sk_j|^2)+\sum_{j=l+1}^{n}\chid_0(s_j)(|\sk_{j-1}|^2-|\sk_j|^2)$ so that:
\[\begin{aligned}
    |\hat u^\theta-\hat u^{\theta\Delta}|^2(z,\xi) = \sum_{m,n\geq1}\sum_{l=0}^n \sum_{p=0}^m \int_{[0,z]_<^n} \!\!\!\!\!\! d\sss \int_{[0,z]_<^m} \!\!\!\!\!\! d\ttt K^\theta_m(\ttt) 
    \int_{\Rm^{(n+m)d}} \hspace{-.6cm} d\kkk d\qqq\ 
    e^{i \tilde \rG_{lp}(z,\sss,\ttt,\xi,\kkk,\qqq)} 
    \hat u_0(\sk_n(\xi,\kkk))\hat u_0^*(\sk_m(\xi,\qqq))
    \\
    [e^{i\chi_0(s_l)(|\sk_{l-1}|^2-|\sk_l|^2)} -e^{i\chid_0(s_l)(|\sk_{l-1}|^2-|\sk_l|^2)}]
    [e^{-i\chi_0(t_p)(|\sk_{p-1}|^2-|\sk_{p}|^2)} -e^{-i\chid_0(t_p)(|\sk_{p-1}|^2-|\sk_{p}|^2)}]
    \hat \mV^\theta(\sss,\kkk) \hat \mV^{\theta*}(\ttt,\qqq),
\end{aligned}\]
where we use to simplify $\sk_p=\sk_p(\xi,\qqq)$ and $\tilde\rG_{lp}(z,\sss,\ttt,\xi,\kkk,\qqq)=\tilde\rG_l(z,\sss,\xi,\kkk)-\tilde\rG_p(z,\ttt,\xi,\qqq)$. We deduce from \eqref{eq:errorchi} that $|e^{i\chi_0(s_l)(|\sk_{l-1}|^2-|\sk_l|^2)} -e^{i\chid_0(s_l)(|\sk_{l-1}|^2-|\sk_l|^2)}| \leq \sfc \Delta z (|\sk_{l-1}|^2+|\sk_l|^2)$.
This shows that
\[\begin{aligned}
&\E|\hat u^\theta-\hat u^{\theta\Delta}|^2(z,\xi) \leq \sfc (\Delta z)^2 \sum_{m,n\geq1}\sum_{l=0}^n \sum_{p=0}^m \int_{[0,z]_<^n} \!\!\!\!\!\! d\sss \int_{[0,z]_<^m} \!\!\!\!\!\! d\ttt \\ 
&\int_{\Rm^{(n+m)d}} \hspace{-1cm} d\kkk d\qqq\ 
(|\sk_l(\xi,\kkk)|^2+|\sk_{l-1}(\xi,\kkk)|^2) (|\sk_p(\xi,\qqq)|^2+|\sk_{p-1}(\xi,\qqq)|^2)
|\hat u_0(\sk_n(\xi,\kkk))||\hat u_0(\sk_m(\xi,\qqq))|\ 
\Big| \E  \hat \mV^\theta(\sss,\kkk) \hat \mV^{\theta*}(\ttt,\qqq)\Big|.
\end{aligned}\]
Using \eqref{eq:controlkl} and the change of variables $q\to-q$ (knowing that $\hat \nu^*(q)=\hat \nu(-q)$),
\[\begin{aligned}
&\E|\hat u^\theta-\hat u^{\theta\Delta}|^2(z,\xi) \leq  (\Delta z)^2
\sum_{m,n\geq1} \sfc^{n+m}(n+1)(m+1)\int_{[0,z]_<^n} \!\!\!\!\!\! d\sss \int_{[0,z]_<^m} \!\!\!\!\!\! d\ttt  
\\
&\times\int_{\Rm^{(n+m)d}}\hspace{-1cm} \aver{\sk_n(\xi,\kkk)}^2|\hat u_0(\sk_n(\xi,\kkk))| \aver{\sk_m(\xi,-\qqq)}^2|\hat u_0(\sk_m(\xi,-\qqq))|\ \Big| \E  \prod_{j=1}^n 
    \aver{k_j}^2\hat \nu^\theta(s_j,k_j) \frac{dk_j}{(2\pi)^d}
    \prod_{i=1}^m \aver{q_j}^2\hat \nu^{\theta}(t_j,q_j) \frac{dq_j}{(2\pi)^d} \Big|.
\end{aligned}\]
This implies, using the Cauchy-Schwarz inequality that $\E 
\int_{\Rm^d}|\hat u^\theta-\hat u^{\theta\Delta}|^2(z,\xi) d\xi $ is upper bounded by
\[
\begin{aligned}
(\Delta z)^2
\int_{\Rm^d} \aver{\xi}^4|\hat u_0(\xi)|^2 d\xi  
\sum_{m,n\geq1}\sfc^{n+m} \int_{[0,z]_<^n}\hspace{-.4cm}d\sss \int_{[0,z]_<^m} \hspace{-.4cm} d\ttt \int_{\Rm^{(n+m)d}}\hspace{-.4cm}  |\E  \prod_{j=1}^n 
    \aver{k_j}^2\hat \nu^\theta(s_j,k_j)
    \prod_{i=1}^m \aver{q_j}^2\hat \nu^{\theta}(t_j,q_j) \Big|.
   \end{aligned}
\]
The series is summable from~\eqref{eq:controlaver_sum} and using the Parseval equality concludes the proof of the first part in Theorem~\ref{thm:pathwise_error} for $\sfv=\ut$.

\subsection{Convergence of moments for the splitting scheme}
This section proves the first estimate in Theorem~\ref{thm:momentestim} for $\sfv=\ut$. We focus on the centered splitting scheme $\gamma=\frac12$ as the first-order estimates are significantly simpler to obtain. As in the proof for the \IS\ model, Lemma \ref{lem:Strang} is the central estimate used to obtain the convergence results. Unlike the \IS\ model, such estimates need to be applied to the Duhamel expansion of the moments, which is combinatorially significantly more involved.  \tb{We also note that for the first order splitting, Lemma 2.6 gives an upper bound of $O(\Delta z)$ instead of $O((\Delta z)^2)$ as below, which is consistent with the error analysis in the \IS\ regime.}

Let $(p,q)$ be fixed. We observe that $\|\E(\psit_{p,q}-\psitd_{p,q})(z,\cdot)\| \leq \sum_{n\geq1}\|\E(\psit_{2n}-\psitd_{2n})(z,\cdot)\|$, with
\[
\E(\psi^\theta_{2n}-\psi^{\theta\Delta}_{2n})(z,v) = \sum_{l=1}^{2n} \sum_{l\neq\ell=1}^{2n}\sum_{m=1}^{(p+q)^{2n}}
\int_{\mathbb{R}^{2nd}} \hspace{-0.5cm} d\kkk \hat\mu_{p,q}(0,v-A_m\kkk)  \Pi_{m,l,\ell} (z,v,\kkk),\quad \Pi_{m,l,\ell}=\int_{[0,z]_<^{l-1}}\hspace{-0.9cm}d\sss_{l>}\int_0^{s_{l-1}}\hspace{-0.5cm}\delta\phi_{l}(s_l)Q_{l\ell} ds_l\,,
\]
where we have defined
\[
\delta\phi_l(s)=e^{i\chi_0(s_l)\phi_l}-e^{i\chid_0(s_l)\phi_l}, 
\ \ 
Q_{l\ell}=\int_{[0,s_l]^{\ell-l-1}_<}\hspace{-1cm}ds_{l+1}\cdots ds_{\ell-1}\int_0^{s_{\ell-1}}\hspace{-0.3cm}e^{i\psi_\ell}\hat{C}^\theta(s_l-s_\ell,k_l,k_\ell)ds_\ell\int_{[0,s_\ell]^{2n-\ell}_<}\hspace{-0.5cm}e^{i\psi_{l\ell}}\mP_{l\ell}d\sss_{\ell<}.
\]
Here we use the notation $d\sss_{l>}=\prod_{j=1}^{l-1}ds_j,d\sss_{\ell<}=\prod_{j=\ell+1}^{2n}ds_j, \mP_{l\ell}=\mathbb{E}\prod_{j\neq l,\ell}\hat{\nu}^\theta(s_j,k_j)$. The phases $\phi_l,\phi_\ell$ do not depend on $\sss$ and $\psi_{l\ell}$ is independent of $s_l,s_\ell$. If $l>\ell$, $\psi_\ell=\chi_0(s_{\ell})\phi_{\ell}$ and if $l<\ell$, $\psi_\ell=\chid_0(s_{\ell})\phi_{\ell}$. 

For some linear operators $A_{lm}$ and $B_{lm}$, the phases $|\phi_l|\le \sum_{j=1}^{p+q}2|v_j-[A_{lm}\kkk]_j|^2+2|v_j-[B_{lm}\kkk]_j|^2\le \sfc^n(|v-A_m\kkk|^2+\sum_{j=1}^n||k_j|^2\le \sfc^n\langle v-A_m\kkk\rangle^2\prod_{j=1}^n\langle k_j\rangle^2$. 
We find that for any $r\geq0$
\begin{equation}\label{eq:bdtv2n1}
 \| \E(\psi^\theta_{2n}-\psi^{\theta\Delta}_{2n})(z,\cdot)\| \leq \|\hat\mu_{p,q}(0,\cdot)\aver{\cdot}^r \|\sum_{l=1}^{2n}\sum_{l\neq\ell=1}^{2n} \sum_{m=1}^{(p+q)^{2n}} 
 \sup_v \| \aver{v-A_m\cdot}^{-r} \Pi_{m,l,\ell} (z,v,\cdot)\|.
\end{equation}
We aim to bound $|\Pi_{m,l,\ell}|$ uniformly in $(v,\kkk)$. We assume $l<\ell$ with the proof of $l>\ell$ following the same strategy writing integrals over the simplex as $\int_0^z ds_{2n} \int_{s_{2n}}^z ds_{2n-1}\ldots$. We use Lemma \ref{lem:Strang} repeatedly to obtain $O(\Delta z)^2$ contributions to $|\Pi_{m,l,\ell}|$ so that \eqref{eq:bdtv2n1} still makes sense and thus need to bound $\|\psi\|_{1,\infty}$ and $|\phi|$ appearing there. We write $e^{i\psi_\ell} =  \delta \phi_{\ell }(s_\ell) + e^{i\chi_0(s_{\ell})\phi_{\ell}}$, where $e^{i\chi_0(s_{\ell})\phi_{\ell}}$ will be differentiable in $s_\ell$. This comes with an error term of the form $e^{i\chid_0(s_\ell)\phi_\ell}-e^{i\chi_0(s_\ell)\phi_\ell}$, which is bounded by $\sfc\Delta z |\phi_\ell|$.  This involves a contribution of the form, choosing $r=4$:
\[
 \sup_v |\delta\Pi_{lm}(z,v,\kkk)| \leq \sfc (\Delta z)^2 \dsum_{P^{2n}}\int_{[0,z]_<^{2n}} \!\!\!\!\!\!\! d\sss\,  \prod_{(j,\jay)\in P^{2n}} \aver{k_j}^2\aver{k_\jay}^2 \hat C^\theta(s_j-s_\jay,k_j,k_\jay),
\]
which is bounded in the TV sense following Lemma \ref{lem:estims} and \eqref{eq:estimgeom}. In what follows, we thus assume that $e^{i\psi_\ell}$ involves $\chi_0(s_\ell)$ whenever it needs to be differentiated in the variable $s_\ell$. 
When $\ell\ge l+2$, we have
\[
\partial_{s_l}Q_{l\ell}=\int_{[0,s_l]^{2n-l-1}_<}\hspace{-1.2cm}\hat{C}^{\theta}(s_l-s_\ell)e^{i\chi_0(s_\ell)\phi_\ell}e^{i\psi_{l\ell}}\mP_{l\ell}\mathds{1}(s_{l+1}=s_l)d\sss_{l+1<}+\int_{[0,s_l]^{2n-l}}\hspace{-1cm}\partial_{s_l}\hat{C}^\theta(s_l-s_\ell)e^{i\chi_0(s_\ell)\phi_\ell}e^{i\psi_{l\ell}}\mP_{l\ell}d\sss_{l<}\,.
\]
\tb{Let $\que$ be defined as the pairing associated with $l+1$ (see proof of Lemma \ref{lem:estims} for definitions of pairs of indices).} $|\partial_{s_l}Q_{l\ell}|$ is bounded uniformly in $s_l$ by
\begin{equation*}
    \begin{aligned}
&|\hat{C}^\theta(0,k_{l+1},k_\que)|\,
\sup_{s_l} \!\int_{0}^{s_l}\hspace{-0.3cm}|\hat{C}^\theta(s_l-s_\ell)|ds_\ell
\int_{0}^z\hspace{-0.2cm}\int_0^{s_{l+3}}\hspace{-0.6cm}\cdots\int_{0}^{s_{2n-1}}\hspace{-0.8cm}|\mP'_{l\ell}|\hspace{-0.4cm}\prod_{\que \neq j\ge l+3}\hspace{-0.3cm}ds_j
+\sup_{s_l}\int_{0}^{s_l}\hspace{-0.4cm}|\partial_{s_l}\hat{C}^\theta(s_l-s_\ell)|ds_\ell
\int_{[0,z]^{2n-l-1}_<}\hspace{-1.2cm}|\mP_{l\ell}|\prod_{\ell \neq j\ge l+1}\hspace{-0.3cm}ds_j.       
    \end{aligned}
\end{equation*}
Here, \tb{$\mP'_{l\ell}=\mathbb{E}\prod_{j\neq l,l+1,\ell,\que}\hat{\nu}^\theta(s_j,k_j)$}. Lemma \ref{lem:Strang} provides a contribution $O(\theta^{-1}(\Delta z)^2)$ to $|\Pi_{m,l,\ell}|$ after integrating in all the other variables. 

The case $\ell=l+1$ is treated separately. There, $\partial_{s_l}Q_{l,l+1}$ is given by
\[
\hat{C}^\theta(0)e^{i\chi_0(s_l)\phi_{l+1}}\int_{[0,s_l]^{2n-l-1}_<}\hspace{-1cm}e^{i\psi_{l,l+1}}\mP_{l,l+1}d\sss_{l+1>}+\int_0^{s_l}\hspace{-0.2cm}\partial_{s_l}\hat{C}^\theta(s_l-s_{l+1})e^{i\chi_0(s_{l+1})\phi_{l+1}}\int_{[0,s_{l+1}]^{2n-l-1}_<}\hspace{-1cm}e^{i\psi_{l,l+1}}\mP_{l,l+1}d\sss_{l+1>}ds_{l+1}\,.
\]
This is uniformly bounded in $s_l$ by 
\[
|\hat{C}^\theta(0)|\int_{[0,z]^{2n-l-1}_<}\hspace{-0.1cm}|\mP_{l,l+1}|d\sss_{l+1>}+(\sup_{s_l}\int_0^{s_l}\hspace{-0.2cm}|\partial_{s_l}\hat{C}^\theta(s_l-s_{l+1})|ds_{l+1})\int_{[0,z]^{2n-l-1}_<}\hspace{-0.1cm}|\mP_{l,l+1}|d\sss_{l+1>}
\]
which gives a contribution of $O(\theta^{-1}(\Delta z)^2)$ to $|\Pi_{m,l,\ell}|$ as before. 

With $r=4$ and summing in $l,\ell$, we have that $\sum_{l,\ell}\sup_v\|\langle  v-A_m\cdot\rangle^{-r}\Pi_{m,l,\ell}\|$ is bounded by
\begin{equation*}
    \begin{aligned}
&  c^n\frac{(\Delta z)^2}{\theta}\frac{n^3}{(n-2)!}\Big(\int_{\mathbb{R}^{2d}}\langle k_1\rangle^4\langle k_2\rangle^4|\hat{C}(0,k_1,k_2)|dk_1dk_2+\int_{\mathbb{R}^{2d+1}}\hspace{-0.5cm}\langle k_1\rangle^4\langle k_2\rangle^4|\hat{C}'(s,k_1,k_2)|dsdk_1dk_2\Big)\\
\times&\Big(\int_{\mathbb{R}^{2d+1}}\langle k_1\rangle^4\langle k_2\rangle^4|\hat{C}(s,k_1,k_2)|dsdk_1dk_2\Big)^{n-1}\,.        
    \end{aligned}
\end{equation*}

Since for $\ell\ge l+2$, the pairings in $\Pi_{m,l,\ell}$ go against the time ordering in the simplex and should also be upper bounded by $O(\theta\Delta z)$ as follows. For $l\ge \ell +2$, let $\que$ be the pairing associated with $l+1$. As $s_\ell<s_{l+1}<s_l<z$, $\Pi_{m,l,\ell}$ can be upper bounded by
\begin{equation*}
    \begin{aligned}
        \Big(\int_{[0,z]^{2n-4}_<}d\sss|\mP'|\Big)\Big(\int_{0}^z\int_0^z\int_0^{s_l}\int_{0}^{s_{l+1}}|\delta\phi_{l}(s_l)||\hat{C}^\theta(s_l-s_{\ell})||\hat{C}^\theta(s_{l+1}-s_\que)|ds_\ell ds_{l+1}ds_l ds_\que\Big).
    \end{aligned}
\end{equation*}
Here, the first integral excludes the $\ell,{l+1},l,\que$ components of $\sss$ and $\kkk$ and is bounded independent of $\theta$ in total variation. The second integral is bounded by 
\[
\sfc|\phi_{l}|\Delta z\int_{0}^z\int_0^z\int_{s_{l+1}}^{z}\int_{0}^{s_{l+1}}|\hat{C}^\theta(s_l-s_{\ell})||\hat{C}^\theta(s_{l+1}-s_\que)|ds_\ell ds_{l}ds_{l+1}ds_\que \leq \sfc\theta \Delta z.
\]
Indeed we verify using $u=t-s$ and $v=\frac{t+s}2$ on the domain $0<s<z<t<\infty$ implying $\frac{t+s}2\in[z-\frac u2,z+\frac u2]$ that
$\int_z^\infty\int_0^z |\hat C^\theta(s-t)|ds dt = \int_0^\infty |\hat C^\theta(u)|udu = \theta \int_0^\infty u|\hat C(u)|du$ is bounded uniformly in $z$,
and use the fact that $\int_{\mathbb{R}^{2d+1}}\langle k_1\rangle^r\langle k_2\rangle^r|s|\hat{C}(s,k_1,k_2)|dsdk_1dk_2$ is bounded for $r$ sufficiently large.

For the case $\ell=l+1$, we distinguish the case where $\que$, the pairing of $l+2$, is equal to $l+3$. If $\que\neq l+3$, this goes against the time ordering as before and contributes $O(\theta\Delta z)$. So we only need to treat separately the case $\ell=l+1, \que=l+3$, i.e, the pairing $(l,l+1), (l+2,l+3)$. Define the antiderivative 
\[
E^\theta(s)=\int_s^\infty\hat{C}^\theta(t)dt\,.
\]
We note that $|{E}^\theta(s)|<\sfc$ uniformly in $(s,\theta)$ as a bounded measure. 
Integration by parts gives
\begin{equation*}
    \begin{aligned}
Q_{l,l+1}&=E^\theta(s_l)\mathds{1}(l=2n-1)e^{i\psi_{l,l+1}}\mP_{l,l+1}-{E}^\theta(0)e^{i\chi_0(s_l)\phi_{l+1}}\int_{[0,s_l]^{2n-l-1}_<}e^{i\psi_{l,l+1}}\mP_{l,l+1}d\sss_{l+1<}\\
&-\int_0^{s_l}E^\theta(s_l-s_{l+1})\partial_{s_{l+1}}(e^{i\chi_0(s_{l+1})\phi_{l+1}})\Big(\int_{[0,s_{l+1}]^{2n-l-1}_<}e^{i\psi_{l,l+1}}\mP_{l,l+1}d\sss_{l+1<}\Big)ds_{l+1}\\
&-\int_0^{s_{l+1}}E^\theta(s_l-s_{l+1})e^{i\chi_0(s_{l+1})\phi_{l+1}}\Big(\int_{[0,s_{l+1}]^{2n-l-2}_<}e^{i\psi_{l,l+1}}\mP_{l,l+1}\mathds{1}(s_{l+2}=s_{l+1})d\sss_{l+2<}\Big)ds_{l+1}\,\tb{,}
    \end{aligned}
\end{equation*}
\tb{where $\mathds{1}(\cdot)$ denotes the indicator function such that $\mathds{1}(s=t)f(s,t)=f(s,s)=f(t,t)$. As before, we plan to use Lemma 2.6, which involves differentiating $Q_{l,l+1}$. }The first boundary term after differentiating in $s_l$ gives $\hat{C}^\theta(s_l)$, so has to be treated separately. We note that
this term is active only when $l=2n-1$, in which case its contribution to $\Pi_{m,l,l+1}$ is upper bounded by 
\[
\int_{[0,z]^{2n-1}_<}|\mP_{2n-1,2n}|d\sss_{2n-1>}\Big|\int_0^{s_{2n-2}}\int_{s_{2n-1}}^\infty\delta\phi_{2n-1}(s_{2n-1}) \hat{C}^\theta(t)dtds_{2n-1}\Big|\,.
\]
Again, we use the boundedness of $\int_{\mathbb{R}^{2d+1}}\langle k_1\rangle^r\langle k_2\rangle^r|s|\hat{C}(s,k_1,k_2)|dsdk_1dk_2$ for $r=2$ to conclude that this boundary term contributes $O(\theta\Delta z)$. The derivative of the other three terms w.r.t. $s_l$ is
\begin{equation*}
    \begin{aligned}
      &-2E^\theta(0)\partial_{s_l}(e^{i\chi_0(s_l)\phi_{l+1}})\int_{[0,s_l]^{2n-l-1}_<}\hspace{-1.2cm}e^{i\psi_{l,l+1}}\mP_{l,l+1}d\sss_{l+1<}-2E^\theta(0)e^{i\chi_0(s_{l})\phi_{l+1}}\int_{[0,s_l]^{2n-l-2}_<}\hspace{-1.2cm}e^{i\psi_{l,l+1}}\mP_{l,l+1}\mathds{1}(s_{l+2}=s_{l+1})d\sss_{l+2<}\\
      &+\int_0^{s_l}\hat{C}^\theta(s_l-s_{l+1})\partial_{s_{l+1}}(e^{i\chi_0(s_{l+1})\phi_{l+1}})\Big(\int_{[0,s_{l+1}]^{2n-l-1}_<}\hspace{-1cm}e^{i\psi_{l,l+1}}\mP_{l,l+1}d\sss_{l+1<}\Big)ds_{l+1}\\
      &+\int_0^{s_l}\hat{C}^\theta(s_l-s_{l+1})e^{i\chi_0(s_{l+1})\phi_{l+1}}\Big(\int_{[0,s_{l+1}]^{2n-l-2}_<}\hspace{-1cm}e^{i\psi_{l,l+1}}\mP_{l,l+1}\mathds{1}(s_{l+2}=s_{l+1})d\sss_{l+2<}\Big)ds_{l+1}\,.
    \end{aligned}
\end{equation*}
Let $\mP'_{l,l+1}=\mathbb{E}\prod_{l\neq l+1\neq l+2\neq l+3}\hat{\nu}^{\theta}(s_j,k_j)$. Then the above terms are bounded as
\begin{equation*}
    \begin{aligned}
        &\sfc(|E^\theta(0)||\phi_{l+1}|\int_{[0,z]^{2n-l-1}_<}\hspace{-1cm}|\mP_{l,l+1}|d\sss_{l+2<}+|E^\theta(0)|\Big(\sup_{s_l}\int_0^{s_l}\hspace{-0.25cm}|\hat{C}^\theta(s_l-s_{l+3},k_{l+2},k_{l+3})|ds_{l+3}\Big)\int_{[0,z]^{2n-l-3}_<}\hspace{-0.4cm}|\mP'_{l,l+1}|d\sss_{l+3<}\\
        &+|\phi_{l+1}|\Big(\sup_{s_l}\int_0^{s_l}|\hat{C}^\theta(s_l-s_{l+1})|ds_{l+1}\Big)\int_{[0,z]^{2n-l-1}_<}|\mP_{l,l+1}|d\sss_{l+1<}\\
        &+\Big(\sup_{s_l}\int_0^{s_l}\int_0^{s_{l+1}}|\hat{C}^\theta(s_l-s_{l+1})||\hat{C}^\theta(s_{l+1}-s_{l+3})|ds_{l+3}ds_{l+1}\Big)\int_{[0,z]^{2n-l-3}_<}|\mP'_{l,l+1}|d\sss_{l+3<}.
    \end{aligned}
\end{equation*}
 Choosing $r=4$ and summing in $l$, 
 \[
 \sum_{l,\ell}\sup_v\|\langle v-A_m\cdot\rangle^{-r}\Pi_{m,l,\ell}\|\le \sfc^n\frac{n^3}{(n-2)!}\Delta z(\theta +\Delta z)\Big(\int_{\mathbb{R}^{2d+1}}\langle k_1\rangle^4\langle k_2\rangle^4\langle s\rangle|\hat{C}(s,k_1,k_2)|dsdk_1dk_2\Big)^{n}\,.
 \]
Using Lemma \ref{lem:estims} and \eqref{eq:estimgeom} combined with \eqref{eq:bdtv2n1} and $n^3\leq \sfc 2^n$ shows that $\| \E(\psi^\theta_{2n}-\psi^{\theta\Delta}_{2n})(z,\cdot)\|\leq  \frac{\sfc^n}{n!} \Delta z(\theta+\Delta z)$ for some constant $\sfc$ that depends on \eqref{eqn:u0_bound} and \eqref{eq:intcorr}. This is summable in $n$ and concludes the proof of the estimate when $\theta\leq \Delta z$. 

When $\theta\geq \Delta z$, the best above estimate is $\theta^{-1}(\Delta z)^2$ instead. Taking the minimum of these estimates provides $(\Delta z)^\beta$ for a choice of $\beta$ as indicated in the theorem. This confirms the interference effect between $\theta$ and $\Delta z$ that is maximized when $\Delta z=\theta^2$. Note that in practice, the natural interesting regime is $\theta\ll\Delta z$.

This concludes the derivation of the first part of Theorem~\ref{thm:momentestim} for $\sfv=\ut$.

\subsection{Strong stability estimates under medium discretization}\label{subsec:medium_path}
We now aim to prove the second part in Theorem~\ref{thm:pathwise_error} for $\sfv=\ut$ and $\sfv=\utd$. We assume that $\nu(z,x)$ is a stationary random medium with continuously differentiable power spectrum and $\nu_c(z,x)$ is its discretization described in \eqref{eq:discreterandom}. This implies that the two random media are appropriately highly correlated for $K_k$ large and $\Delta k$ small.

We first prove the following result.
\begin{lemma}\label{lem:mediumcorrelation}
    Let $\phi(k,\xi)$ be sufficiently smooth (deterministic) with second derivative bounded by 
   
    $\sfc\aver{k}^M\aver{\xi}^M$. Then for \tb{$\hat \nu_\epsilon\in\{\hat \nu, \hat \nu_c\}$}, we have
\[
      \dint_{[0,Z]^2} ds dt\Big| \int_{\Rm^{2d}} d\xi dk \E (\hat\nu^\theta(s,k) - \hat \nu^\theta_c(s,k)) \hat \nu^\theta_\epsilon(t,\xi)\phi(k,\xi)\Big|\leq \sfc \Big( (\Delta k)^2 \|w \sup_{|\alpha|\leq2} |\partial^\alpha \phi| \|_\infty +K_k^{-N} \|\phi\|_\infty \Big)\,.  
   \]
    Here the weight is $w=\aver{k}^{-M}\aver{\xi}^{-M}$.
\end{lemma}
\begin{proof}
We deduce from \eqref{eq:mediumcor} that 
\[
\begin{aligned}
    \E (\hat\nu^\theta(s,k) - \hat \nu^\theta_c(s,k)) \hat \nu^\theta(t,\xi) 
    &=
    \hat C^\theta(t-s,k) \big(\delta(k+\xi) - \chi_c(k) \delta(\xi+q_k)\big),
    \\
    \E (\hat\nu^\theta(s,k) - \hat \nu^\theta_c(s,k)) \hat \nu^\theta_c(t,\xi)
    &= \chi_c(q_k) \delta(\xi+q_k) \Big( \hat C^\theta(t-s,k) - \chi_c(q_k) \delta(k-q_k) \int_{\square_{q_k}} \hat C^\theta(t-s,\zeta)d\zeta\Big).
\end{aligned}
\]
Suppose $\hat\nu^\theta_\epsilon=\hat\nu^\theta$. The case $\hat{\nu}^\theta_\epsilon=\hat{\nu}^\theta_c$ follows a similar proof. This amounts to finding a bound for
        \begin{equation}\label{eqn:nu_dif_bound}
     \int_{\Rm^d} \chi_c(q_\xi) \hat C^\theta(t-s,\xi) (\phi(\xi,-\xi)-\phi(q_\xi,-\xi)) d\xi
     \end{equation}
as from~\eqref{eq:intcorr}, the term $\int_{[0,Z]^2}dsdt\big|\int_{\mathbb{R}^d}(1-\chi_c(q_\xi))\hat C^\theta(s,\xi)\phi(\xi,-\xi)d\xi \big|\le \sfc K_k^{-N}\|\phi\|_{\infty}$. Here $q_\xi\in(\Delta k\mathbb{Z})^d$ denotes the center of the cube containing $\xi$. To bound~\eqref{eqn:nu_dif_bound}, we write (suppressing $(s,t)$ dependence):
    \[ \begin{aligned}
        \hat C^\theta(\xi) (\phi(\xi,-\xi)-\phi(q_\xi,-\xi))
    = (\hat C^\theta(\xi)-\hat C^\theta(q_\xi)) (\phi(\xi,-\xi)-\phi(q_\xi,-\xi))  \\
    + \hat C^\theta(q_\xi) (\phi(\xi,-\xi)-\phi(q_\xi,-q_\xi) + \phi(q_\xi,-q_\xi)-\phi(q_\xi,-\xi)).
    \end{aligned}\]
    For $\psi\in C^2(\mathbb{R}^d)$,  $\hat C^\theta(q_\xi)  ( \psi(\xi)-\psi(q_\xi)) = \hat C^\theta(q_\xi) (\xi-q_\xi)\cdot \nabla \psi(q_\xi) + [(\xi-q_\xi)\cdot\nabla]^2\psi(\tilde{q}_\xi)$
        for some $\tilde{q}_\xi\in \square_{q_\xi}$. Since  $\hat{C}^\theta(q_\xi) (\xi-q_\xi)\cdot \nabla \psi(q_\xi)$ integrates to $0$ on that cube, this shows that
\[ \Big| \frac{1}{|\square_{q_\xi}|}  \int_{\square_{q_\xi}}\hspace{-0.3cm} \hat C^\theta(s-t,\xi) (\phi(\xi,-\xi)-\phi(q_\xi,-\xi))d\xi \Big| \leq  \sfc(\Delta k)^2 \aver{\xi}^M \aver{q_\xi}^M \|\hat C^\theta(s-t,\cdot)\|_{C^1(\square_{q_\xi})}\big(\sum_{|\alpha|\le 2}\|w\partial^\alpha\phi\|_{\infty}\big)\,.
\]
It remains to sum over cubes using the rapid decay of $\hat C^\theta(\xi)$ as $|\xi|\to\infty$ given by  \eqref{eq:intcorr}.
\end{proof}

We are ready to prove the \tb{strong} estimates in the second part of Theorem~\ref{thm:pathwise_error} for $\sfv\in\{\ut,\utd\}$.  Consider two solutions $\ut_{1,2}$ propagating in two different media $\nu^\theta_{1,2}$. Then we have
\[\begin{aligned}
\mE(z) := \dint_{\Rm^d}\E |\hut_1-\hut_2|^2(z,\xi)  d\xi  = \dint_{\Rm^d} d\xi\dsum_{n\geq1}\dsum_{m\geq1}  
\int_{[0,z]_<^n} \hspace{-.5cm} d\sss 
\int_{[0,z]_<^m} \hspace{-.5cm} d\ttt 
\int_{\Rm^{nd}} \hspace{-.3cm} d\kkk 
\int_{\Rm^{md}} \hspace{-.3cm} d\qqq 
\  \mI_n(z,\sss,\xi,\kkk) \mI_m^*(z,\ttt,\xi,\qqq)
\\
\E (\hat\mV^\theta_1(\sss,\kkk)-\hat\mV^\theta_2(\sss,\kkk))(\hat\mV^\theta_1(\ttt,\qqq)-\hat\mV^\theta_2(\ttt,\qqq))^*\,,
\end{aligned}\]
with notation $\hat{\mV}^\theta_\epsilon(\sss,\kkk)=\prod_{j=1}^n\hat{\nu}^\theta_\epsilon(s_j,k_j)$, $\epsilon=1,2$.
The sums over $n$ and $m$ start at $1$ since both solutions have the same ballistic component. We find
\[
\begin{aligned}
    \mE(z) &\leq \dsum_{n\geq1}\dsum_{m\geq1}  
\int_{[0,z]_<^n} \hspace{-.5cm} d\sss 
\int_{[0,z]_<^m} \hspace{-.5cm} d\ttt \  \sum_{\epsilon=1}^2 \delta\mJ_{nm\epsilon}(z,\sss,\ttt) \\
\delta\mJ_{nm\epsilon}(z,\sss,\ttt) &= \Big|\dint_{\Rm^d} d\xi\int_{\Rm^{nd}} \hspace{-.3cm} d\kkk 
\int_{\Rm^{md}} \hspace{-.3cm} d\qqq 
\  \mI_n(z,\sss,\xi,\kkk) \mI_m^*(z,\ttt,\xi,-\qqq)
\E (\hat\mV^\theta_1(\sss,\kkk)-\hat\mV^\theta_2(\sss,\kkk)) \hat\mV^\theta_\epsilon(\ttt,\qqq) \Big|.
\end{aligned}
\]
Assume $\nu_1=\nu$ a stationary Gaussian process with smooth correlation function $\hat C(t,\xi)$ and $\nu_2=\nu_c$ the discretization considered earlier.
For $\epsilon=1,2$, we next write the decomposition
\[
\delta \hat{\mV}^\theta(\sss,\kkk) \hat{\mV}^\theta_\epsilon(\ttt,\qqq) =\sum_{l=1}^n \prod_{r=1}^{l-1} \hat \nu_1^\theta(s_r,k_r) (\hat \nu_1^\theta(s_l,k_l)-\hat \nu_2^\theta(s_l,k_l)) \prod_{r=l+1}^{n} \hat \nu_2^\theta(s_r,k_r) \prod_{r=1}^m \hat \nu_\epsilon^\theta(t_r,q_r).
\]
Let $\vs=(\sss,\ttt)$ and $\vk=(\kkk,\qqq)$. In particular, $\sfs_j=s_j$ for $1\le j\le n$ and $\sfs_j=t_{j-n}$ for $n+1\le j\le m+n$. For a fixed value of $l$, the expectation of the above term is 
\[
 \sum_{l\neq\ell=1}^{n+m} \delta\hat C_{l\ell}^\theta(\sfs_l-\sfs_\ell,\sfk_l,\sfk_\ell) \mP_{l\ell},\qquad 
\mP_{l\ell}:=\mathbb{E}[\prod_{j\neq l,\ell}\hat{\nu}^\theta_{\epsilon(j)}(\sfs_j,\sfk_j)]=\sum_{P^{m+n-2}}\prod_{(j,\jay)} \mathbb{E}[\hat{\nu}^\theta_{\epsilon(j)}(\sfs_j,\sfk_j)\hat{\nu}^\theta_{\epsilon(\jay)}(\sfs_\jay,\sfk_\jay)]\,,
\]
where $(j,\jay)$ is a Gaussian pairing with $j,\jay\neq l,\ell$ and $\hat{\nu}^\theta_{\epsilon(j)}$ can be either $\hat{\nu}^\theta_1$ or $\hat{\nu}^\theta_2$ depending on the pairing. The correlation difference $\delta\hat{C}^\theta_{l\ell}$ is defined as
$\delta\hat{C}^\theta_{l\ell}(\sfs_l-\sfs_\ell,\sfk_l,\sfk_\ell)=\E[(\hat \nu_1^\theta(\sfs_l,\sfk_l)-\hat \nu_2^\theta(\sfs_l,\sfk_l))\,\hat \nu_{\epsilon(\ell)}^\theta(\sfs_\ell,\sfk_\ell)]$.
We then observe that the contribution of this term to $\delta\mJ_{nm\epsilon}$ from Lemma~\ref{lem:mediumcorrelation} requires us to bound 
\[
    \sup_j \sup_{|\alpha|\leq 2} | \dint_{\Rm^d} d\xi \partial^\alpha_{\sfk_j} (\mI_n(z,\sss,\xi,\kkk) \mI^*_m(z,\ttt,\xi,\qqq))|\,.
\]
For $\sfv=u^\theta$ from~\eqref{eq:notationDuhamel},  $|\partial_{k_j}\mI_n(z,\sss,\xi,\kkk)|\le 2^n\sup_{0\le s\le z}|\chi_0(s)|\langle\xi\rangle\prod_{j=1}^n\langle k_j\rangle|\hat{u}_0(\sk_n(\xi,\kkk))|+|\partial\hat{u}_0(\sk_n(\xi,\kkk))|$. 
In general,
\begin{equation}\label{eqn:I_deriv}
    \begin{aligned}
|\partial^\alpha_{k_j}\mI_n(z,\sss,\xi,\kkk)|
&\le c^{n|\alpha|}\langle\xi\rangle^{|\alpha|}\prod_{j=1}^n\langle k_j\rangle^{|\alpha|}\sup_{|\alpha_1|\le |\alpha|}|\partial^{\alpha_1}\hat{u}_0(\sk_n(\xi,\kkk))|\,.      
    \end{aligned}
\end{equation}
The case $\sfv=\utd$ is similar replacing $\chi_0(s)$ by $\chid_0(s)$ with same upper bound. This gives
\begin{equation}\label{eq:kderiv}
\begin{aligned}
    & \sup_j \sup_{|\alpha|\leq 2} | \dint_{\Rm^d} d\xi \partial^\alpha_{\sfk_j} (\mI_n(z,\sss,\xi,\kkk) \mI^*_m(z,\ttt,\xi,\qqq))|\\
    &\leq \sfc ^{n+m}\sup_{|\alpha+\beta|\leq 2}  \dint_{\Rm^d} d\xi \aver{\xi}^2 \prod_{j=1}^{n}\aver{k_j}^2\prod_{j=1}^{m}\aver{q_j}^2|\partial^\alpha \hat u_0(\sk_n(\xi,\kkk))|  |\partial^\beta |\hat u_0(\sk_m(\xi,-\qqq))|
  \\
  &\leq \sfc^{n+m}\prod_{j=1}^{n}\aver{k_j}^4\prod_{j=1}^{m}\aver{q_j}^4 \Big( \sup_{|\alpha|\leq 2}  \dint_{\Rm^d} \aver{\xi}^{2} |\partial^\alpha \hat u_0(\xi)|^2\Big).
\end{aligned}
\end{equation}

Now applying Lemma~\ref{lem:mediumcorrelation} (with $M=4$), this term gives an $O((\Delta k)^2+K_k^{-N})$ error after integrating in the $(\sfs_l,\sfs_\ell,\sfk_l,\sfk_\ell)$ variables. Integration in the other variables is handled in the same way and provides an estimate of the form (for $\sfc$ independent of $(n,m,\sss,\ttt)$ and including the $\hat u_0$ terms): 
\[\begin{aligned}
\mE(z) 
 & \leq 
  (\Delta k)^2\dsum_{n\geq1}\dsum_{m\geq1} \sum_{l=1}^n\sum_{l\neq\ell=1}^{n+m} \sfc^{n+m} \int_{\mathcal{S}_{l\ell}}d\vs_{l\ell}\int_{\mathbb{R}^{(n+m-2)d}}|\mP_{l\ell}|\prod_{j\neq l,\ell}\langle \sfk_j\rangle^Md\sfk_j\le \sfc(\Delta k)^2\,.
 \\
\end{aligned}
\]
Here, $\mathcal{S}_{l\ell}$ denote the simplices after removing $(\sfs_l,\sfk_l,\sfs_\ell,\sfk_\ell$) and we have applied~\eqref{eq:controlaver_sum} in Lemma~\ref{lem:estims} using the symmetry of the elements of $\mP_{l\ell}$. This proves the second part of Theorem~\ref{thm:pathwise_error} for $\sfv\in\{\ut,\utd\}$.

\subsection{Moment estimates for medium discretization}\label{subsec:medium_mom}

We now prove the third estimate in Theorem~\ref{thm:momentestim} for $\ut$ and $\utd$. For $(p,q)$ fixed, we write the Duhamel expansion \eqref {eqn:chaos_series2} as
\[ \psi(z,v) = \sum_{n\geq0} \dsum_{m=1}^{(p+q)^n}\psi_{mn}(z,v),\quad\psi_{mn}(z,v) =  \int_{[0,z]^n_<} \hspace{-.4cm} d\sss 
    \int_{\mathbb{R}^{nd}} \hspace{-.4cm} d\kkk\  \tilde{\mI}_m(\sss,v,\kkk)\hat \mV^\theta(\sss,\kkk),
\]
where $\tilde{\mI}_m(\sss,v,\kkk) = \hat\mu_{p,q}(0,v-A_m\kkk) e^{iG_m(\sss,v,\kkk)}$. We wish to estimate $ \delta\psi = \sum_{n\geq0} \sum_{m=1}^{(p+q)^n}\delta\psi_{mn}$ where
\[
 \delta\psi_{mn}  = \int_{[0,z]^n_<} \hspace{-.4cm} d\sss 
    \int_{\mathbb{R}^{nd}} \hspace{-.4cm} d\kkk\ \tilde{\mI}_m(\sss,v,\kkk)
 (\hat \mV^\theta_{1}-\hat \mV^\theta_{2})(\sss,\kkk),\quad\mV^\theta_\epsilon(\sss,\kkk)=\prod_{j=1}^n\hat{\nu}^\theta_\epsilon(s_j,k_j),\quad \epsilon=1,2.
\]
As in Section~\ref{subsec:medium_path}, we have 
\begin{align*}
  \mathbb{E} \delta\psi_{mn}=\sum_{l=1}^n\sum_{l\neq\ell=1}^n\int_{[0,z]^n_<}d\sss\int_{\mathbb{R}^{nd}}d\kkk\tilde{\mI}_m(\sss,v,\kkk)\mathbb{E}[(\hat{\nu}^\theta_1(s_l,k_l)-\hat{\nu}^\theta_2(s_l,k_l))\hat{\nu}^\theta_{\epsilon(\ell)}(s_\ell,k_\ell)]\mP_{l\ell}
\end{align*}
where $\mP_{l\ell}=\mathbb{E}[\prod_{j\neq l,\ell}\hat{\nu}^\theta_{\epsilon(j)}(s_j,k_j)]$. We can now show that $\sup_j \sup_{|\alpha|\leq 2} | \int_{\Rm^d} d\xi \partial^\alpha_{k_j} (\tilde{\mI}_m(\sss,\xi,\kkk) |$ is bounded as in~\eqref{eq:kderiv}, and use Lemma~\ref{lem:mediumcorrelation} to obtain $\|\mathbb{E}\psi_1-\mathbb{E}\psi_2\|\le \sfc[(\Delta k)^2 + K_k^{-N}]$. Phase recompensating and inverse Fourier transforming proves the second part of Theorem~\ref{thm:momentestim}.

\subsection{Strong and moment estimates of spatial concentration}

We now prove the third parts in Theorem~\ref{thm:pathwise_error} and \ref{thm:momentestim} for $\sfv\in \{\ut, \utd\}$. This is based on the following concentration result.
\begin{lemma}\label{lem:boundut}
    Let $\sfv\in \{\ut, \utd, \utc, \utdc\}$ and $r\in\Nm$ and assume that \eqref{eqn:u0_bound} holds for all $N+M\leq 2r$ and \eqref{eq:intcorr} \tb{holds} for $M=2r$. Then  we have the decay estimate $\sup_{0\leq z\leq Z} \sup_{x\in\Rm^d} \E \big[ \aver{x}^{4r} |v(z,x)|^2\big] \leq \sfc(Z,r)$.
\end{lemma}
\begin{proof}
We write the proof of $\sfv=\ut$. The other terms are treated similarly. Consider the function $|x|^{2r}\ut(z,x)$ for $r\in\Nm$, whose (partial) Fourier transform is $(-\Delta_\xi)^r\hut(z,\xi) = \sum_{n\geq0} (-\Delta_\xi)^r\hut_n(z,\xi) 
=  \sum_{n\geq0} \int_{[0,z]_<^n} d\sss \int_{\Rm^{nd}} d\kkk (-\Delta_\xi)^r \mI_n(z,\sss,\xi,\kkk) \hat{\mV}^\theta(\sss,\kkk)
$. Thus $|x|^{4r} \E |\ut(z,x)|^2=\sum_{n,m\geq0}\tilde{\mathcal{J}}_{mn}(z,x)$, where \begin{align*}
   &\tilde{\mathcal{J}}_{m,n}(z,x)=\! \! \int_{\mathbb{R}^{2d}}\hspace{-0.4cm}e^{ix\cdot(\xi-\zeta)}\frac{d\xi d\zeta}{(2\pi)^{2d}}\dint_{[0,z]^n_<} \hspace{-.5cm} d\sss\dint_{[0,z]^m_<} \hspace{-.5cm} d\ttt\int_{\Rm^{(n+m)d}} \hspace{-1cm}d\kkk d\qqq \, (-\Delta_\xi)^r\mI_n(z,\sss,\xi,\kkk)(-\Delta_\zeta)^r\mI^\ast_m(z,\ttt,\zeta,-\qqq)\mathbb{E}\hat{\mV}^\theta(\sss,\kkk)\hat{\mV}^\theta(\ttt,\qqq).
\end{align*}
Using the bound from~\eqref{eqn:I_deriv} with $\alpha=2r$, we obtain as before that $|\tilde{\mathcal{J}}_{nm}|$ is upper bounded by
\begin{align*}
 \sfc^{(n+m)r}\Big(\sup_{|\alpha|\le 2r}\int_{\mathbb{R}^{d}}\langle\xi\rangle^{2r}|\partial^\alpha\hat{u}_0|d\xi\Big)^2\dint_{[0,z]^n_<} \hspace{-.4cm} d\sss\dint_{[0,z]^m_<} \hspace{-.4cm} d\ttt\int_{\Rm^{(n+m)d}} \hspace{-0.7cm}d\kkk d\qqq\prod_{j=1}^n\langle k_j\rangle^{2r}\prod_{l=1}^m\langle q_l\rangle^{2r}|\mathbb{E}\hat{\mV}^\theta(\sss,\kkk)\hat{\mV}^\theta(\ttt,\qqq)|    \,.
\end{align*}
From Lemma~\ref{lem:estims}, these terms are summable for $n+m$ even which finishes the proof for $\sfv=\ut$. The proof for $\sfv=\{\utd,\utc,\utdc\}$ can be obtained in a similar fashion.
\end{proof}

For $\sfv\in \{\ut,\utd,\utc,\utdc\}$, we thus have that $\aver{|x|}^{4N} \E |\sfv(z,x)|^2\leq \sfc$ for $N$ sufficiently large. 
Note that $(\sfv_\sharp - \sfv)=\sum_{h\not=0} \sfv(x-hL)$
so that \eqref{eq:errorwsharp} yields $\E |\sfv_\sharp-\sfv|^2(x)\leq L^{-2N}$. After integration over $\Xm=\Tm_L^d$, we obtain the third part in Theorem~\ref{thm:pathwise_error} for $\sfv\in \{\ut,\utd\}$. Following the proof of Section \ref{sec:periodIS}, we deduce that $ |\mu_{p,q}[\utc]-\mu_{p,q}[u^{\theta}_\sharp]|(z,X,Y)+|\mu_{p,q}[\utdc]-\mu_{p,q}[\utdp]|(z,X,Y)\le \sfc_NL^{-N}$.
This proves the third part of Theorem~\ref{thm:momentestim} for $\sfv\in\{\ut,\utd\}$. 

\subsection{\tb{Strong} and moment estimates for spatial discretization}
\label{sec:pathwise}

We recall the discrete Fourier transform and orthogonal projection $\Pi=\Pi_{N_x}$ defined in \eqref{eq:discFourier}.  For $f\in L^2(\Tm_L^d)$, the polynomial $\Pi_{N_x}f$ is uniquely characterized by its values at the grid points $(\Delta x \Zm)^d \cap \Tm_L^d$.

The periodized functions $\utp$ and $\utdp$ are both solutions of the equation
\[
 \partial_z \sfv_\sharp = i \phi(z) \Delta_x \sfv_\sharp + i\nu^\theta_c(z,x) \sfv_\sharp,\quad z>0,\ x\in\Tm_L^d
\]
with $\sfv_\sharp(0,x)=u_0(x)$ and $\phi(z)=\kappa_1(z)$ for $\sfv_\sharp=\utp$ while $\phi(z)=\tbb{\tau_\gamma(z)}\kappa_1(z)$ for $\sfv_\sharp=\utdp$. The spatial discretization $\sfv_\delta$ then satisfies in both cases the defining equation \eqref{eq:discreteeq}
\[
 \partial_z \sfv_\delta = i \phi(z) \Delta_x \sfv_\delta + i  \Pi(\nu^\theta_c(z,x) \sfv_\delta),
\]
with $\sfv_\delta(0,x)=\Pi_{N_x}u_0(x)$. Indeed, $\Pi_{N_x}(\nu^\theta_c(z,x) \sfv_\delta)$ can be implemented on each grid point of the mesh by point-wise multiplication, $\Delta_x$ can be implemented locally in the Fourier variables and hence written as a finite-rank operator (matrix) acting on $\sfv_\delta(j\Delta x)$.  

This shows that $\Pi_{N_x}\sfv_\sharp$ and $\sfv_\delta$ satisfy the same evolution equation with the same initial condition, but with operators acting on different spaces (finite dimensional for $\sfv_\delta$, not for $\sfv_\sharp$). The comparison between the two solutions may then be obtained as usual by analyzing Duhamel expansions. 

Suppose $\sfv=\ut$. The Duhamel expansion in the periodic case starts in the Fourier domain from:
\[ \hat \sfv_l(z)=\hat\sfv_l(0)e^{-i\chi_0(z) |\Delta k l|^2}+i\int_0^z ds e^{-i\chi_s(z) |\Delta k l|^2}  \sum_{k\in \Zm^d} \hat \nu^\theta_{k}(s) \hat \sfv_{l-k}(s). \]
Here, $\hat \nu^\theta_{n}(s)=\hat{\nu}^\theta_c(s,n\Delta k)$ is the Fourier coefficient of the discrete periodic random medium. This shows that
\[ \hat \sfv_l(z) = \sum_{n\geq0} \int_{[0,z]^n_<} \hspace{-.4cm} d\sss  \sum_{\kkk\in\Zm^{nd}} e^{i\rG_n(z,\sss,\Delta kl,\Delta \kkk)} \hat{\mV}^\theta(\sss,\kkk) \hat{u}^0_{l-A_n\kkk},\quad\hat{\mV}^{\theta}(\sss,\kkk):=\prod_{j=1}^n\hat{\nu}^\theta_{k_j}(s_j)
\]
 for the real valued phase $\rG_n(z,\sss,\Delta kl,\Delta \kkk)=\frac{n\pi}{2}+\sum_{j=1}^n\chi_0(s_j)|\Delta k|^2(|l-\sum_{r=1}^{j-1}k_r|^2-|l-\sum_{r=1}^jk_r|^2)-\chi_0(z)|\Delta k l|^2$ and linear operator $A_n\kkk=-\sum_{j=1}^nk_j$. The Fourier coefficients of the source $\hat{u}^0_l=\hat{u}_l(0)$. Let $w=\sfv_\delta=\Pi_{N_x}w$ be the spatially discrete approximation solution of the above finite dimensional system of ordinary differential equations in $z$ when $\sfv=\utp$ or its splitting approximation when $\sfv=\utdp$. The Duhamel expansion is then obtained by replacing multiplication by $\nu^\theta(z,x)$ by the composition $\Pi_{N_x} \otimes \nu^\theta(z,x)$. In the Fourier domain, this is simply multiplication by $\mathds{1}_{|l_\epsilon|<N_x}$ following (discrete) convolution with $\hat w_{l}$. Here $\mathds{1}_{|l_\epsilon|\leq N_x}$ is the indicatrix function of indices such that $|l_\epsilon|\leq N_x$ for each $\epsilon=1,\cdots,d$. In other words,
\[ 
\begin{aligned}
    \hat w_l(z)&=\hat{w}_l(0)e^{-i \chi_0(z)|\Delta k l|^2}+i\int_0^z ds e^{-i\chi_s(z) |\Delta k l|^2}  \mathds{1}_{|l_\epsilon|\leq N_x}  \sum_{k\in \Zm^d} \hat \nu^\theta_{k}(s) \hat w_{l-k}(s),
    \\
    & = \sum_{n\geq0} \int_{[0,z]^n_<} \hspace{-.4cm} d\sss  \sum_{\kkk\in \Zm^{nd}} e^{i\rG_n(z,\sss,\Delta k l,\Delta k\kkk)} \prod_{j=0}^{n-1} \mathds{1}_{|(\sk_j(l,\kkk))_\epsilon|\leq  N_x} \hat{\mV}^\theta(\sss,\kkk) \hat u^0_{l-A_n\kkk},
\end{aligned}
\]
where we defined $\sk_j(l,\kkk)=l-\sum_{m=1}^j k_m$ while $\sk_0(l,\kkk)=l$. This implies that $(\hat v_l-\hat w_l)\aver{\Delta k l}^{2M}$ is given by
\[
 \sum_{n\geq0} \sum_{r=1}^n \int_{[0,z]^n_<} \hspace{-.4cm} d\sss  \sum_{\kkk\in \Zm^{nd}} e^{i\rG_n(z,\sss,\Delta kl,\Delta k\kkk)} \aver{l}^{2M} \mathds{1}^c_{|(\sk_r(l,\kkk))_\epsilon|\leq N_x} \prod_{j=r+1}^{n} \mathds{1}_{|(\sk_j(l,\kkk))_\epsilon|\leq  N_x} \hat{\mV}^\theta(\sss,\kkk) \hat u^0_{l-A_n\kkk},
\]
using $\mathds{1}_D^c=\mathds{1}-\mathds{1}_D$.
Now, $\mathds{1}^c_{|(\sk_j(l,\kkk))_\epsilon|\leq  N_x}\leq (\Delta k N_x)^{-N} |\Delta k\sk_j(l,\kkk)|^N$ (with $|\sk_j|$ Euclidean norm).  Using \eqref{eq:controlkl}, we bound $\aver{\Delta kl}^{2M}$ and the latter by products  of $\aver{\Delta k(l-A_n\kkk)}^N$ and $\aver{\Delta k k_j}^N$ while using \eqref{eq:controlaver} as in the derivation of Lemma \ref{lem:stabpsitpq}, we have $\E \|\sfv_\sharp - \sfv_\delta\|^2_{H^{2M}(\Tm_L^d)} \leq \sfc (\Delta x)^{2N}$. The same calculation shows that $\sfv_\sharp=\utp$ and $\sfv_\sharp=\utdp$ are also smooth functions of $x$ since $ \E |\hat \sfv_l(z)|^2 \aver{\Delta kl}^{4M} $ is given by
\[\sum_{n,m\geq0}\int_{[0,z]^n_<} \hspace{-.4cm} d\sss \int_{[0,z]^m_<}\hspace{-.4cm} d\ttt \dsum_{\kkk\in\mathbb{Z}^{nd}} \dsum_{\qqq\in\mathbb{Z}^{md}} e^{i[\rG_n(z,\sss,\Delta kl,\Delta k\kkk)-\rG_m(z,\ttt,\Delta kl,\Delta k\qqq)]}  \aver{\Delta kl}^{4M} \E \hat{\mV}^\theta(\sss,\kkk) \hat{\mV}^{\theta *}(\ttt,\qqq) \hat u^0_{l-A_n\kkk} \hat u^{0*}_{l-A_m\qqq} \]
whose analysis using \eqref{eq:controlkl} and \eqref{eq:controlaver} again provides a bound on $\E\|(-\Delta)^M\sfv_\sharp(z,x)\|^2\leq \sfc$.

For $\sfv_\sharp\in\{u^\theta_\sharp,u^{\theta\Delta}_\sharp\}$ consider the moments $\mu_{p,q}[\sfv_\sharp](X,Y)= \E \prod_j \sfv_\sharp(x_j) \prod_k \sfv_\sharp^*(y_k)$. Assume $q=0$ to simplify notation with an obvious extension to the general case. Using the H\"older inequality and Sobolev embedding, we have:
\[ \begin{aligned} &|\mu_{p}[\sfv_\sharp](X)-\mu_p[\sfv_\delta](X)| = | \sum_{k=1}^p \E \prod_{j=1}^{k-1} \sfv_\sharp(x_j) (\sfv_\sharp(x_k)-\sfv_\delta(x_k)) \prod_{j=k+1}^p \sfv_\delta(x_j)| 
\\ \leq \ &  \sum_{k=1}^p \big(\E |\prod_{j=1}^{k-1} \sfv_\sharp(x_j)|^4\big)^{\frac14}  \big(\E \prod_{j=k+1}^p \sfv_\delta(x_j)|^4\big)^{\frac14} \big(\E |\sfv_\sharp(x_j)-\sfv_\delta(x_j)|^2\big)^{\frac12}
\\ \leq \ &   p  |\mu_{4(k-1)}[\sfv_\sharp](X_{4(k-1)})|^{\frac14} |\mu_{4(p-k)}[\sfv_\delta] (X_{4(p-k)})|^{\frac14} \, \big(\E \|\sfv_\sharp-\sfv_\delta\|^2_{H^{\frac d2+1}}\big)^{\frac12}.
\end{aligned}\]
Spatial moments are all bounded as an application of the Duhamel formula as in the proof of Lemma \ref{lem:stabpsitpq}, both on $\Rm^d$ and $\Tm_L^d$ uniformly in $(z,X)$. Thus, $|\mu_{p}[\sfv_\sharp](X)-\mu_p[\sfv_\delta](X)| \leq \sfc (\Delta x)^{N}$. The extension to arbitrary $(p,q)$ is mainly notational. This concludes the proof of the final estimate in Theorem~\ref{thm:momentestim} for $\sfv_\sharp\in\{\utp,\utdp\}$.

\section{Numerical examples}
\label{sec:num}

This section illustrates our theoretical results with numerical simulations. We consider a numerical setting with final distance $Z=1$ and $\kappa_1=1$.

As a first example, we consider a two-dimensional experimental setting with lateral dimension $d=1$. We assume a Gaussian incident source profile $ u_0(x)=e^{-\frac{1}{2}x^2}$.  We first consider a first order splitting scheme with \tb{$\gamma=0$}. For the spatial discretization, we set $L=2^6$ and $\Delta x=2^{-4}$ with $N_x=2^{10}$ grid points along the $x$ direction.  We vary $\Delta z$ between $2^{-5}$ and $2^{-10}$ with $\Delta z=2^{-10}$ serving as the reference solution. 
 
 In the paraxial model, we fix $\theta=2^{-9}$. The splitting solution~\eqref{eq:splittingexplicit} evaluated at the grid points $\{z_n\}$ translates to $\utdd(z_n,x) =   e^{iW^{\theta\Delta}_n(x)}  \mG(\Delta z)\utdd(z_{n-1},x)$, where $ W^{\theta\Delta}_n(x)=\int_{z_{n-1}}^{z_n}{\nu}^\theta_c(s,x)ds$.
The action of the the discrete Laplacian $\mG(\Delta z)\utdd(z_{n-1},x)$ is implemented through a FFT/IFFT routine. Since the splitting scheme requires only integrals of the potential $W^{\theta\Delta}_n$, the sampling procedure is significantly simplified when $\nu_c$ is Gaussian as is described below. 
\paragraph{Construction of the random medium and sampling procedure} We assume that the random medium is given as in Remark~\ref{rem:nu_c_disc} by $ \nu_c(z,x)=\sum_{q\in\mathbb{Z}_\Delta}\chi_c(q)\hat{\nu}_q(z)\frac{e^{iqx}}{2\pi}$, where $\Delta k=\frac{2\pi}{L}$, $\mathbb{Z}_\Delta$ is the grid $\Delta k\mathbb{Z}$ and $\chi_c(q)$ the indicator function with cutoff $K_k=2\pi\times 2^4$, i.e, $\chi_c(q)=1$ when $-\frac{K_k}{2}\le q\le \frac{K_k}{2}$ and $0$ otherwise. $\{\hat{\nu}_q(z)\}_{q\in\mathbb{Z}_\Delta}$ are \tb{mean-zero} Gaussian random variables with covariance $ \mathbb{E}\hat{\nu}_p(s)\hat{\nu}^\ast_q(t)=\mathds{1}_{p=q}(2\pi\Delta k)\hat{C}(s-t,p)$,
where the lateral covariance $\hat{C}$ is assumed to be $\hat{C}(s,k)=\sigma e^{-\frac{|s|^2+|k|^2}{4\pi^2}}$ and the noise level $\sigma$ is varied. Due to the Gaussian assumption on $\nu_c$, the integrals $W^{\theta\Delta}_n$ can be written in terms of a finite number of Gaussian random variables as $ W^{\theta\Delta}_n(x)={\sqrt{\theta}}\sum_{q\in\mathbb{Z}_\Delta}\chi_c(q)\hat{W}_{n,q}^{\theta}\frac{e^{iqx}}{2\pi}$. Here, $ \hat{W}_{n,q}^{\theta}=\int_{\frac{z_{n-1}}{\theta}}^{\frac{z_n}{\theta}}\hat{\nu}_q(s)ds$ are \tb{mean-zero} Gaussian random variables with covariance $   \mathbb{E}\hat{W}_{m,p}^{\theta}\hat{W}_{n,q}^{\theta\ast}=\mathds{1}_{p=q}(4\pi\Delta k)\int_0^{\frac{\Delta z}{\theta}}\big(\frac{\Delta z}{\theta}-s\big)\hat{C}\big(|n-m|\Delta z +s,p\big)ds$.
This allows us to generate samples of the random variable $ W^{\theta\Delta}_n$ easily by drawing samples from a Gaussian distribution with covariance function given above. 

We denote by $u^\theta_{\rm{ref}}$ the fine grid solution. Let $\Xm$ be the torus $[-\frac12 L,\frac12 L]$. For each $\Delta z$, we approximate the strong error norm $\TP\cdot\TP_{\Xm}$ with the discrete version $\sup_{1\le n\le N_z} (\sum_{j=1}^{N_x}\mathbb{E}|(\utdd(z_n,x_j)-u^\theta_{\rm{ref}}(z_n,x_j)|^2)\Delta x)^{\frac12}$. In the first panel of Figure~\ref{fig:error_dz}, we plot the strong error from the numerical scheme on a log-log scale. In the second and third panels, we plot the numerical errors in the second and fourth moments respectively around the concentration of the beam given by $ \sup_{1\le n\le N_z}\sup_{N_x/4\le j\le 3N_x/4} |\mathbb{E}|\utdd(z_n,x_j)|^p-\mathbb{E}|u^\theta_{\rm{ref}}(z_n,x_j)|^p|$, $p=2,4$. All expectations are computed using an average of $10^4$ Monte Carlo samples. The empirical slope of all three graphs are close to 1, indicating an error rate of $O(\Delta z)$ for both strong and moment estimates as predicted from theory.

To simulate the It\^o-Schr\"{o}dinger model, we replace the coefficients $\hat{W}^{\theta}_{n,q}$ above by the \tb{mean-zero} Gaussian random variable $\hat B_{n,q}$ with covariance $ \mathbb{E}\hat B_{m,p}\hat B^\ast_{n,q}=\mathds{1}_{m=n}\mathds{1}_{p=q}(2\pi\Delta k\Delta z)\hat{R}(q)$, $\hat{R}(k)=\int_{\mathbb{R}}\hat{C}(s,k)ds=\sigma e^{-\frac{|k|^2}{4\pi^2}}$. The experimental setup in the paraxial setting is then repeated for the It\^o-Schr\"{o}dinger case in Figure~\ref{fig:error_dz}. We observe the rate of convergence of $\udd$ to the reference solution $u_{\rm{ref}}$ as $\Delta z$ varies. The slopes in all the three panels are again close to 1, indicating a convergence rate of $O(\Delta z)$. 
\begin{figure}
  \centering
    \includegraphics[scale=0.1]{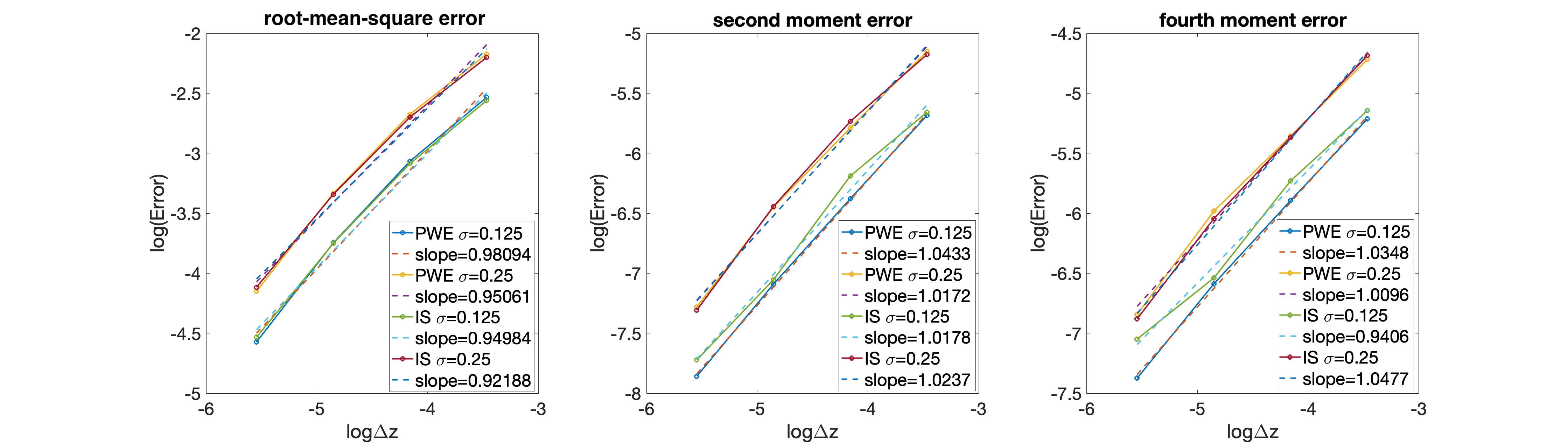}
    \caption{Error plots from first order scheme for the \tbb{paraxial wave equation (PWE)} and It\^o-Schr\"{o}dinger \tbb{(IS)} models}
    \label{fig:error_dz}         
\end{figure}

As a second example, we simulate the It\^o-Schr\"{o}dinger equation with $\gamma=\frac12$. The sampling procedure for the medium is identical as in the first order scheme. As we anticipate the need for significant Monte Carlo averaging in order to observe an error of order $(\Delta z)^2$, we use fewer grid points in $x$ to reduce the computational effort. For this,  we set $L=2^4$ and $\Delta x=2^{-2}$ so that there $N_x=2^6$ grid points in $x$. The noise level $\sigma$ is set to $0.125$ and the reference solution is computed using $\Delta z=2^{-10}$.  In the first panel of Figure~\ref{fig:error_dz_Strang}, we plot the strong error in the numerical scheme. This is still $O(\Delta z)$ as expected. In the second and third panels, we plot the errors in a few Fourier modes given by $ \sup_{1\le n\le N_z} \sum_{j=1}^{N_x}|(\mathbb{E}|\udd(z_n,x_j)|^p-\mathbb{E}|u_{\rm{ref}}(z_n,x_j)|^p)e^{i\frac{m\pi x_j}{L}}|\Delta x$, $p=2,4$  for $m=1,3,5$. All statistical averages are computed using $10^7$ realizations of the random medium. The plots indicate a convergence rate of $O(\Delta z)^2$ as expected from theory.
\begin{figure}
  \centering
     \includegraphics[scale=0.1]{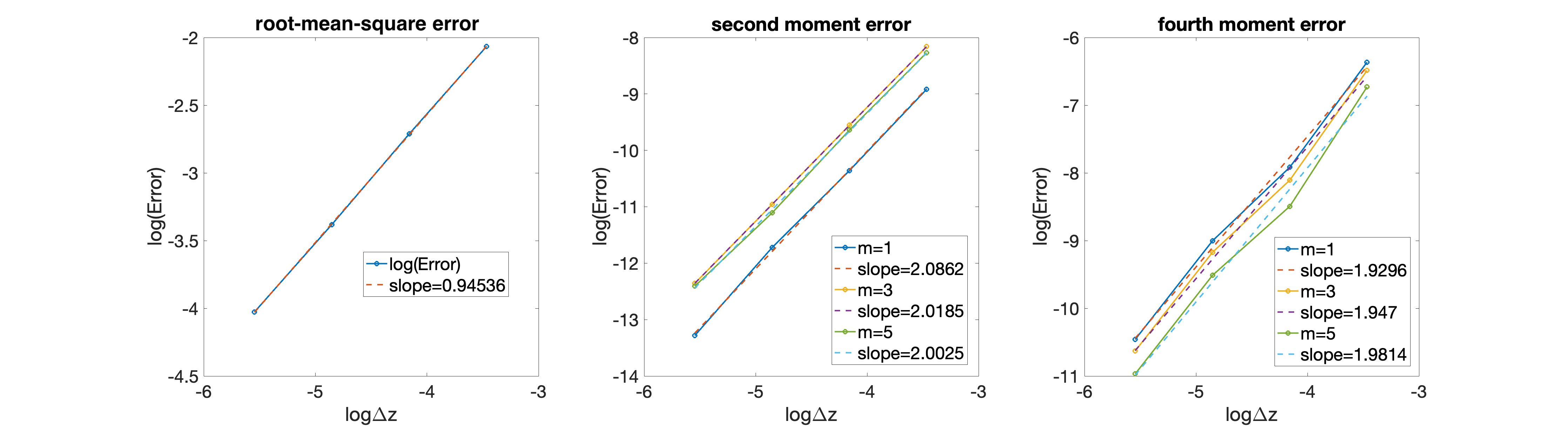}
    \caption{Rate of convergence with the second order splitting scheme under the It\^o-Schr\"{o}dinger model}
    \label{fig:error_dz_Strang}         
\end{figure}

\section{Conclusions}
\label{sec:conlu}
This paper developed an approximation theory for the (time) splitting and full spatial discretization of paraxial and \IS\ equations of wave propagation in random media modeled by short-range Gaussian processes. For the paraxial model, we confirmed surprising observations that the splitting algorithm converged even when the splitting step $\Delta z$ was not the smallest scale of the process; namely $\theta\ll\Delta z$ in practice. We obtained two different convergence results: a first-order one in the mean\tb{-}square sense, and a second-order one for centered splitting for statistical moments in the uniform sense, leveraging total variation estimates in the Fourier variables as in \cite{bal2024complex}. In both cases, we used the availability of closed-form equations satisfied by moments of the \IS\ equation and had to analyze a full Duhamel expansion for the paraxial model. It is quite possible that higher-order splitting algorithms \cite{geometric2006} can be developed to analyze moments of the \IS\ equation. For the paraxial model, we already observed interesting interactions at second-order between $\Delta z$ and $\theta$ as both parameters tend to zero. While our results apply to long distance propagation, longer distances yet may be considered in the scintillation and diffusive regimes considered in \cite{bal2024complex, bal2024long, bal2025long}, where speckle forms and scintillation builds up as briefly illustrated in the Supplementary Material.


\newpage

\section{Supplementary Material}

In Section~\ref{sec:conv_proc} of this Supplementary Material, we provide a result on the convergence of the solution to the fully discrete numerical scheme $\utdd$ of the paraxial approximation given by
\begin{equation}\label{eqn:PWE_supp}
    \begin{aligned}        \partial_z\ut&=i\kappa_1(z)\Delta_x\ut+ i\kappa_2(z)\frac1{\sqrt{\theta}}\nu\Big(\frac{z}{\theta},x\Big)\ut,\quad z>0, x\in\mathbb{R}^d ;\qquad 
        \ut(0,x)=u_0(x),
    \end{aligned}
\end{equation}
to that of the It\^o-Schr\"{o}dinger model $u$ given by
\begin{equation}\label{eqn:Ito_supp}
    \mathrm{d}u=i\kappa_1(z)\Delta_xu\mathrm{d}z-\frac{\kappa_2^2(z)R(0)}{2}u\mathrm{d}z+i\kappa_2(z)u\mathrm{d}B,\quad u(0,x)=u_0(x)\,.
\end{equation}
As in~\cite{bal2024complex, bal2024long} the proof is based on identifying the limiting distribution through its statistical moments followed by a tightness result. 

We next provide a formal derivation of the paraxial model from the Helmholtz equation in Section~\ref{app:background}. Finally in Section~\ref{sec:add_num_eg} we conclude by providing additional numerical examples which validate the numerical schemes for both paraxial and It\^{o}-Schr\"{o}dinger models using the analytical expressions available to the first two statistical moments in the It\^o-Schr\"{o}dinger case. We also provide illustrations of speckle phenomena consistent with physical observations when optical beams propagate through strong turbulence and with theoretical predictions in \cite{bal2024complex,bal2024long}.
\subsection{Convergence of finite dimensional distributions and tightness}\label{sec:conv_proc}
We have the following notions of convergence in distribution as all parameters tend to $0$.  We write these results for $\sfu^{\theta\Delta}_\delta(z,x)=\utdd(z,x)\chi_L(x)$ for concreteness. The exact same results apply to $\ut$, $\utd$, $\ud$, and $\udd$ after extending them to appropriate domains. 
We first have a result for finite dimensional distributions.
For a collection of $n$ points $(x_1,\cdots,x_n)\in\mathbb{R}^{nd}$, we define the random vector $\mathbf{M}[u](x_1,\cdots,x_n)=\big(u(x_1),\cdots,u(x_n)\big)$.  

\begin{proposition}[Convergence of finite dimensional distributions]\label{prop:conv_dist}
For fixed $0\le z\le Z$, the random vector 
\[
\mathbf{M}[\sfu^{\theta\Delta}_\delta(z,\cdot)](X) \Rightarrow \mathbf{M}[u(z,\cdot)](X)
\]
as $(\theta,\Delta z,\Delta k,K_k^{-1},\Delta x)\to (0,0,0,0,0)$ where $u(z,x_j)$ follows the It\^{o}-Schr\"{o}dinger equation~\eqref{eqn:Ito_supp}.
\begin{proof}
From~\cite{dawson1984random}, 
\[
\|u(z,x)\|_{L^2(\mathbb{R}^d)}= \|u(0,x)\|_{L^2(\mathbb{R}^d)}\quad \text{a.s.}
\]
As in the proof of Proposition~\ref{prop:tightness}, it can be shown that the process $x\to u(z,x)$ is continuous a.s., which means for a fixed $z$, $|u(z,x)|\le c$ a.s. for some $c$. This gives us that $\mu_{p,p}(z,X,Y)\le c^{2p}$, and from a Carleman criterion, the probability distribution of $u$ can be identified using its statistical moments. This along with Corollary 2.4 proves Proposition~\ref{prop:conv_dist}.   
\end{proof}
\end{proposition}
For fixed $z$, we have a result on the stochastic continuity and relative compactness of the process $x\to\sfu^{\theta\Delta}_\delta(z,x)$. 
\begin{proposition}[Tightness]\label{prop:tightness}
    The tightness criterion $  \mathbb{E}|\sfu^{\theta\Delta}_\delta(z,x+h)-\sfu^{\theta\Delta}_\delta(z,x)|^{2n}\le c(z,n)|h|^{2n}$ holds, where $c$ is a constant independent of $\theta$, $\Delta z$, $\Delta k$, $K_k$ and $\Delta x$.
    \begin{proof}
        We write the difference 
\[
\sfu^{\theta\Delta}_\delta(z,x+h)-\sfu^{\theta\Delta}_\delta(z,x)= [\utdd(z,x+h)-\utdd(z,x)]\chi_L(x)+\utdd(z,x+h)[\chi_L(x+h)-\chi_L(x)] 
\]
so that
\begin{align*}
    |\sfu^{\theta\Delta}_\delta(z,x+h)-\sfu^{\theta\Delta}_\delta(z,x)|^{2n}&\le C(n) (|\utdd(z,x+h)-\utdd(z,x)|^{2n}|\chi_L(x)|^{2n}\\
    &+|\utdd(z,x+h)|^{2n}|\chi_L(x+h)-\chi_L(x)|^{2n})\,.
\end{align*}
From the smoothness of $\chi_L$ and the boundedness of $\utdd$, the last term after taking expectation is bounded by $c(n)\|\chi_L\|_{1,\infty}|h|^{2n}$. 
From Section~\ref{sec:pathwise}, the difference
\begin{equation*}
    \utdd(z,x+h)-\utdd(z,x)=\sum_{l\in \mathbb{Z}^d}\mathds{1}_{|l_\mu|\le N_x}\hat{v}_l(z)(e^{i\Delta kl\cdot(x+h)}-e^{i\Delta k l\cdot x})\,.
\end{equation*}
For $\mu=1,\cdots,d$, let $\alpha_{j,\mu}, \beta_{j,\mu}$ denote the $\mu$th component of $\alpha_j,\beta_j$. For $n\ge 0$ this gives
\begin{equation*}
\begin{aligned}
    &\mathbb{E}| \utdd(z,x+h)-\utdd(z,x)|^{2n}\\
    &=\sum_{\alpha,\beta\in\mathbb{Z}^{nd}}\mathbb{E}\prod_{j=1}^n\hat{v}_{\alpha_j}(z)\hat{v}^\ast_{\beta_j}(z)\mathds{1}_{|\alpha_{j,\mu}|\le N_x}\mathds{1}_{|\beta_{j,\mu}|\le N_x}(e^{i\Delta k\alpha_j\cdot(x+h)}-e^{i\Delta k\alpha_j\cdot x})(e^{-i\Delta k\beta_j\cdot(x+h)}-e^{-i\Delta k\beta_j\cdot x})\\
    &\le |h|^{2n}\sum_{\alpha,\beta\in\mathbb{Z}^{nd}}|\mathbb{E}[\prod_{j=1}^n\hat{v}_{\alpha_j}(z)\hat{v}^\ast_{\beta_j}(z)]|\prod_{j=1}^n\mathds{1}_{|\alpha_{j,\mu}|\le N_x}\mathds{1}_{|\beta_{j,\mu}|\le N_x}|\Delta k \alpha_j||\Delta k \beta_j|\le C_{n}|h|^{2n}
\end{aligned}
    \end{equation*} 
    due to the boundedness of the Fourier coefficients in the TV sense. 
    \end{proof}
\end{proposition}
The combination of these two results (convergence of finite dimensional moments and tightness) directly leads to the following convergence in distribution result~\cite{kunita1997stochastic}.
\begin{theorem}[Convergence of processes]\label{thm:conv_proc}
For $0\le z\le Z$ fixed and $\alpha\in(0,1)$, the process $x\mapsto \sfu^{\theta\Delta}_\delta(z,x)$ converges in law to the solution of the It\^{o}-Schr\"{o}dinger equation~\eqref{eqn:Ito_supp} on $C^{0,\alpha}(\mathbb{R}^d)$ as $(\theta,\Delta z,\Delta k,K_k^{-1},\Delta x)\to (0,0,0,0,0)$. 
\end{theorem}

\subsection{Formal derivation of paraxial model from Helmholtz model}\label{app:background}

We start with the scalar Helmholtz equation
\begin{equation*}
    \partial_z^2p+\Delta_xp+k_0^2n^2(z,x)p=\delta'_0(z)u_0(x),
\end{equation*}
where $u_0(x)$ denotes the incident beam at $z=0$ and the equation is posed on the domain $z>0$ and $x\in\Rm^d$.
We are interested in the high frequency, long propagation regime, and so assume $k_0$ large while we rescale the $z$ coordinate to consider a more convenient scaling $ k_0\to\frac{k_0}{\theta},\quad z \to \frac{z}{\theta}$. The parameter $\theta$ may be interpreted as the ratio of the typical wavelength with the typical correlation length of the turbulent medium, with a typical value in applications of order $10^{-6} / 10^{-3} = 10^{-3}\ll1$. It is therefore natural to consider the $\theta\ll1$ as well as the $\theta\to0$ regimes. 

We also assume weak turbulence fluctuations about a slowly varying mean so that $  n^2(z,x)=n_0^2(\theta z)\big(1+\theta^{3/2}\nu(z,x)\big)$. While {\em weak} locally, the influence of the turbulence is of order $O(1)$ after long-distance propagation.

We now consider the more slowly varying {\em envelope} $u$ given by $ p(z,x)=u(\theta z,x)e^{\frac{i\phi(z)}{\theta}}$. Substituting this in the original Helmholtz equation gives
\begin{equation*}
\begin{aligned}
    2i\partial_z\phi(z)\partial_su(\theta z,x)-\frac{1}{\theta^2}\big(\partial_z\phi(z)\big)^2u(\theta z,x)+\frac{i}{\theta}\partial_{z}^2\phi(z)u(\theta z,x)&\\
    +\Delta_xu(\theta z,x)+\frac{k_0^2}{\theta^2}n_0^2(\theta z)[1+\theta^{3/2}\nu(z,x)]u(\theta z,x)&=0,\quad z>0;\qquad 
    u(0,x)&=u_0(x),
\end{aligned}
    \end{equation*}
    where we have formally ignored the backscattering term $\theta^2\partial_z^2u$. Comparing terms at $O(\theta^{-2})$, we set $ \big(\partial_z\phi(z)\big)^2=k_0^2n_0^2(\theta z)$ so that $\partial_z^2\phi(z)=\pm\theta k_0\partial_sn_0(\theta z)=\pm \theta k_0\partial_{z'}n_0(z')$.     In particular, when $n_0=1$, we get back the classical paraxial ansatz with $\phi(z)=\pm k_0z$. This gives (keeping the plus sign for instance and dropping the primes on $z$)
    \begin{equation*}
        \begin{aligned}
            2ik_0n_0(z)\partial_z\ut+\Delta_x\ut+ik_0\partial_zn_0(z)\ut+\frac{k_0^2}{\theta^{1/2}}n_0^2(z)\nu\Big(\frac{z}{\theta},x\Big)\ut&=0,\quad z>0;\qquad 
            \ut(0,x)&=u_0(x).
        \end{aligned}
    \end{equation*}
    For the transformation $v^\theta(z,x)=\ut(z,x)e^{\vartheta(z)}$, where $ \vartheta(z)=\frac{1}{2}\int_0^z\frac{n_0'(s)}{n_0(s)}\mathrm{d}s=\frac{1}{2}\log\Big(\frac{n_0(z)}{n_0(0)}\Big)$,
    we have
    \begin{equation*}
        \begin{aligned}
            \partial_zv^\theta&=\frac{i}{2k_0n_0(z)}\Delta_xv^\theta+\frac{ik_0}{2\theta^{1/2}}n_0(z)\nu\Big(\frac{z}{\theta},x\Big)v^\theta,\quad v^\theta(0,x)=u_0(x)\,.
        \end{aligned}
    \end{equation*}
For simplicity we set $k_0=1/2$ and rescale $\nu\to 4\nu$.  This justifies the general class of equations of the form \eqref{eqn:PWE_supp} analyzed in this paper.

\subsection{Additional numerical examples}\label{sec:add_num_eg}
For Gaussian initial conditions of the form $ u_0(x)=e^{-\frac{1}{2}x^2}$, the  first and second moment solutions in the It\^o-Schr\"{o}dinger regime are given analytically as~\cite{garnier2016fourth, bal2024complex} 
\begin{align*}
    \mu_{1,0}(z,x)&= e^{-\sqrt{\pi}z/2}A_0(z,x)\\
    \mu_{1,1}(z,x,x)&= \frac{1}{\sqrt{4\pi}}\int_{\mathbb{R}}e^{-|\xi|^2(z^2+\frac{1}{4})}e^{i\xi\cdot x}e^{(\sigma\sqrt{\pi}\int_0^z(e^{-4\pi^2|\xi|^2 s^2}-1)\mathrm{d}s)}\mathrm{d}\xi\,,
\end{align*}
where $A_0(z,x)=\frac{1}{\vartheta(z)}e^{-\frac{|x|^2}{2\vartheta^2(z)}}, \vartheta(z)=(1+2iz)^{1/2}$ denotes the solution to the paraxial approximation in a homogeneous medium. In Figure \ref{fig:SM1}, we display the statistical averages from the simulation of the paraxial and It\^o-Schr\"{o}dinger models and compare them with the analytical solutions. In the first panel, we compare the real part of the mean field at $z=Z$ computed using the reference numerical solution $u_{\rm{ref}}$ with the corresponding values from the analytical expression. The numerical approximation agrees well with the analytical solution. Similarly in the second panel, we plot the second moment at $z=Z$ computed using $u_{\rm{ref}}$. This agrees well with the analytical solution as well. 
\begin{figure}
  \centering
    \includegraphics[width=0.45\textwidth]{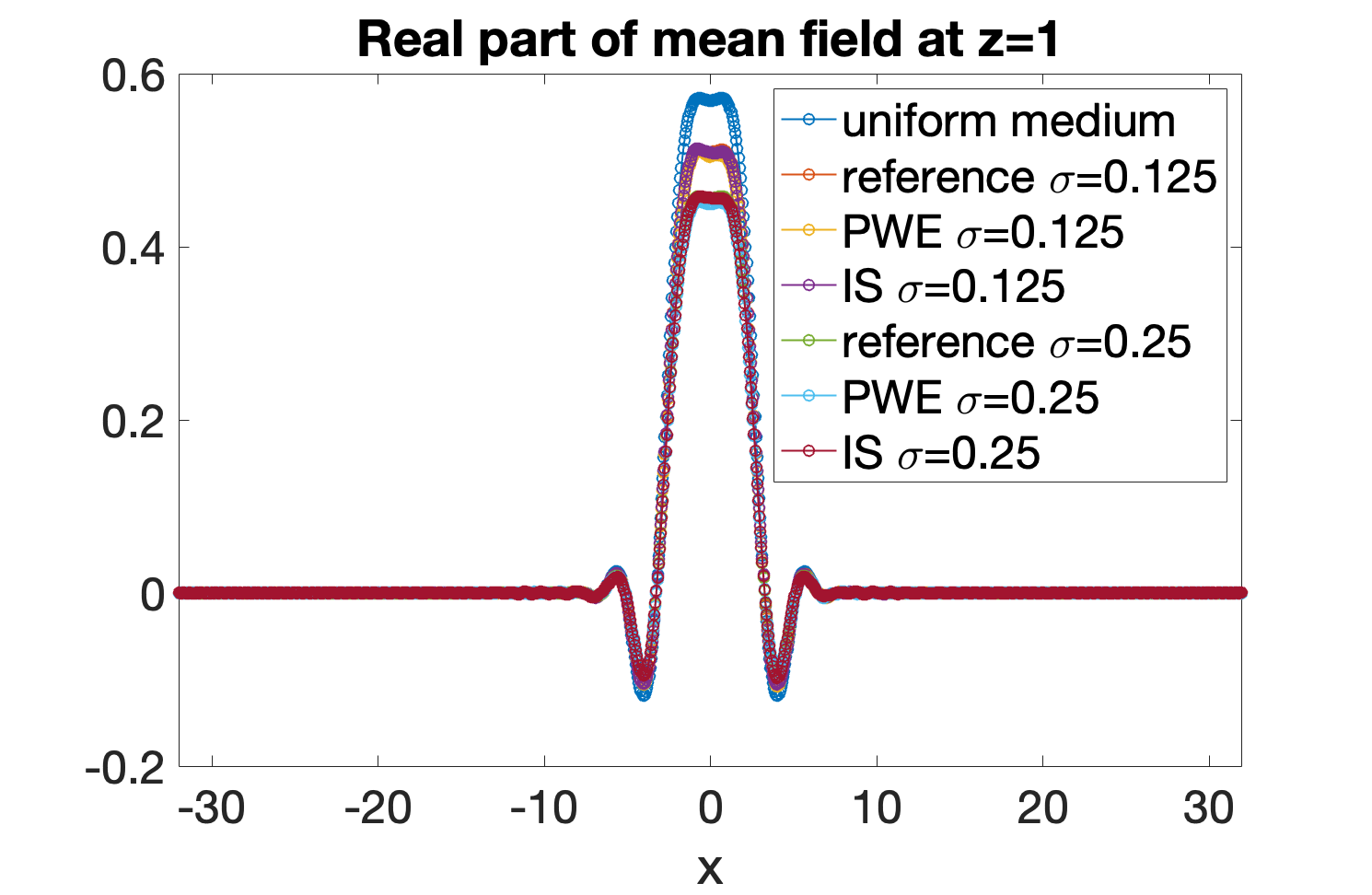}
        \includegraphics[width=0.45\textwidth]{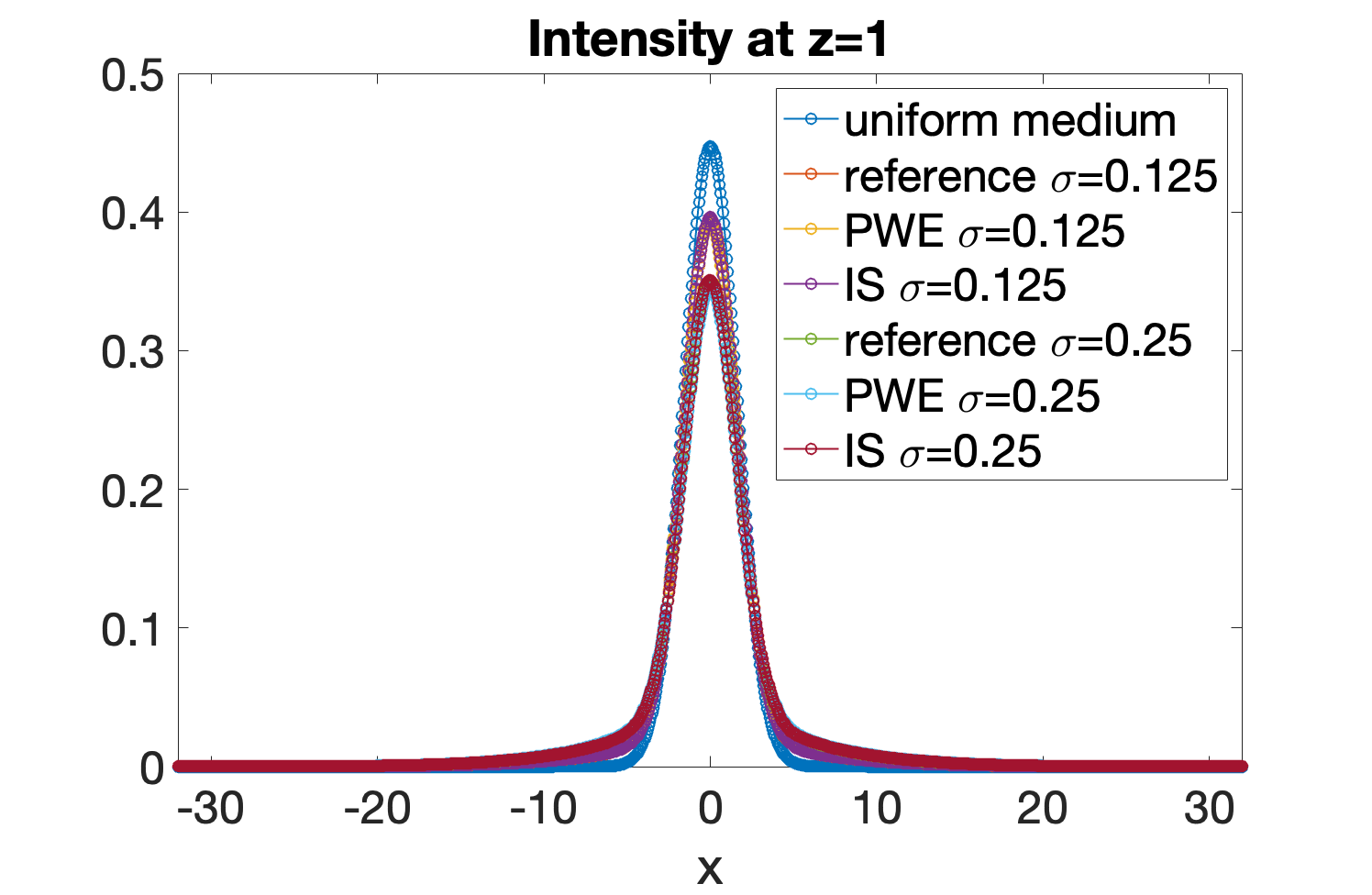}
    \caption{Statistical averages from the first-order simulations for the paraxial approximation and It\^o-Schr\"{o}dinger models. 
    The left and right panels display the real part of wavefield mean value and its intensity, respectively.}   \label{fig:SM1}
\end{figure}

\paragraph{Speckle phenomena}

The first and third panels of Figure~\ref{fig:evolution_Ito_sigma} plot the beam profile (absolute value) as a function of $x$ and $z$ for increasing values of $\sigma$. 
\begin{figure}[htp]
\centering
    \includegraphics[scale=0.1525]{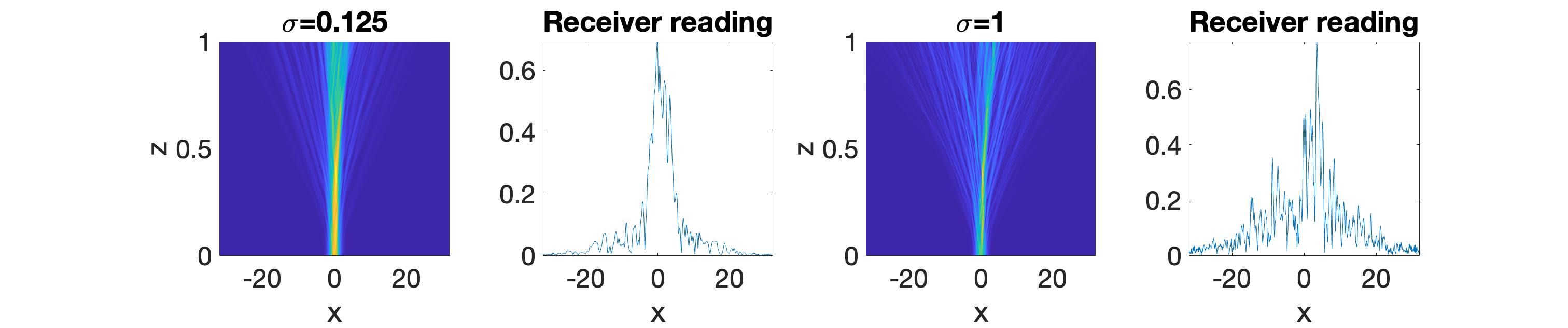}
                      \includegraphics[width=0.35\textwidth]{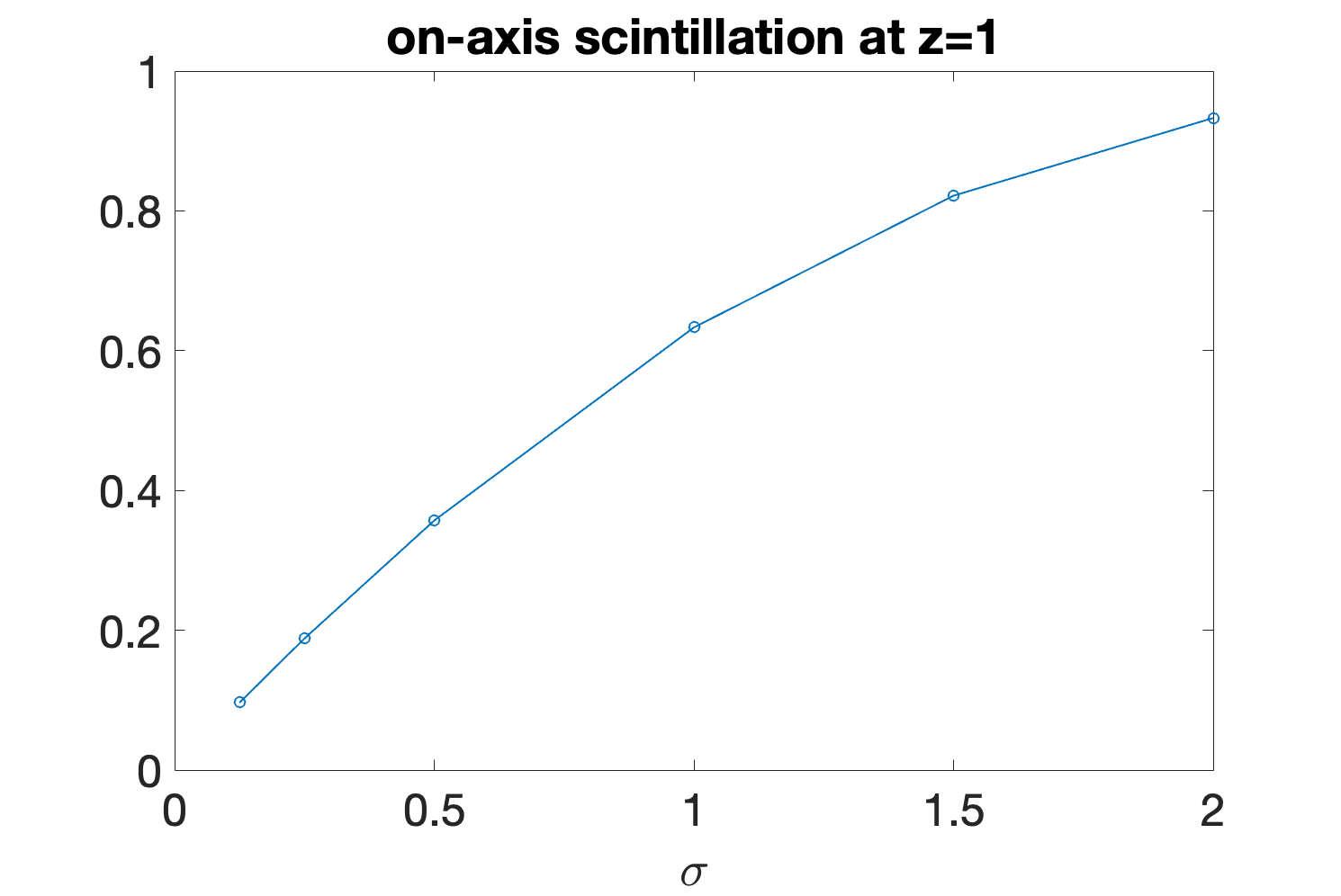}
     \caption{Plots showing evolution of the beam profile and cross-section of the reference solution at $z=1$ (absolute value) for different $\sigma$. The first panel corresponds to the beam evolution for $\sigma=0.125$ while the second panel plots the absolute value of the signal at $z=Z$. The third and fourth panels correspond to the same plots for $\sigma=1$. The last panel plots the scintillation index at the beam's center.}
      \label{fig:evolution_Ito_sigma}           
\end{figure}
The corresponding receiver readings at $z=Z$ are also plotted on the second and fourth panels. The cross section of the beam appears more jagged as $\sigma$ increases. This is consistent with speckle formation in highly turbulent regimes where the beam is expected to rapidly lose coherence \cite{andrews2023laser}. The rest of the panels plot the difference in solutions $|\udd-u_{\rm{ref}}|$ for varying $\Delta z$. The last panel plots the normalized variance of intensity (scintillation) $S=\frac{\mathbb{E}|\udd|^4}{(\mathbb{E}|\udd|^2)^2}-1$ computed from It\^o-Schr\"{o}dinger simulations at the center of the receiver. Scintillation is a commonly used metric for comparing the quality of optical signals and is a useful physical object~\cite{andrews2023laser}. In appropriate strong turbulence scalings, scintillation is expected to saturate to unity~\cite{garnier2016fourth, bal2024complex}. This is consistent with the trend in the simulations observed here.

We conclude with the numerical simulation of a beam in three dimensions corresponding to $d=2$. The source profile is taken as $u_0(x,y)=e^{-\frac{(x^2+y^2)}{2}}$. We plot the cross section of the beam profile (absolute value) at the final receiver location in Figure~\ref{fig:evol_3D} for varying $\sigma$. Here, $\Delta z=2^{-10}$ and the final distance $Z=1$. $\Delta x=\Delta y=2^{-3}$ and $L=2^6$ so that there are $2^9\times 2^9$ grid points in the lateral dimensions. As expected in high turbulence regimes, the beam profile develops fine scale variations consistent with speckle formation. 
\begin{figure}
  \centering
     \includegraphics[scale=0.15]{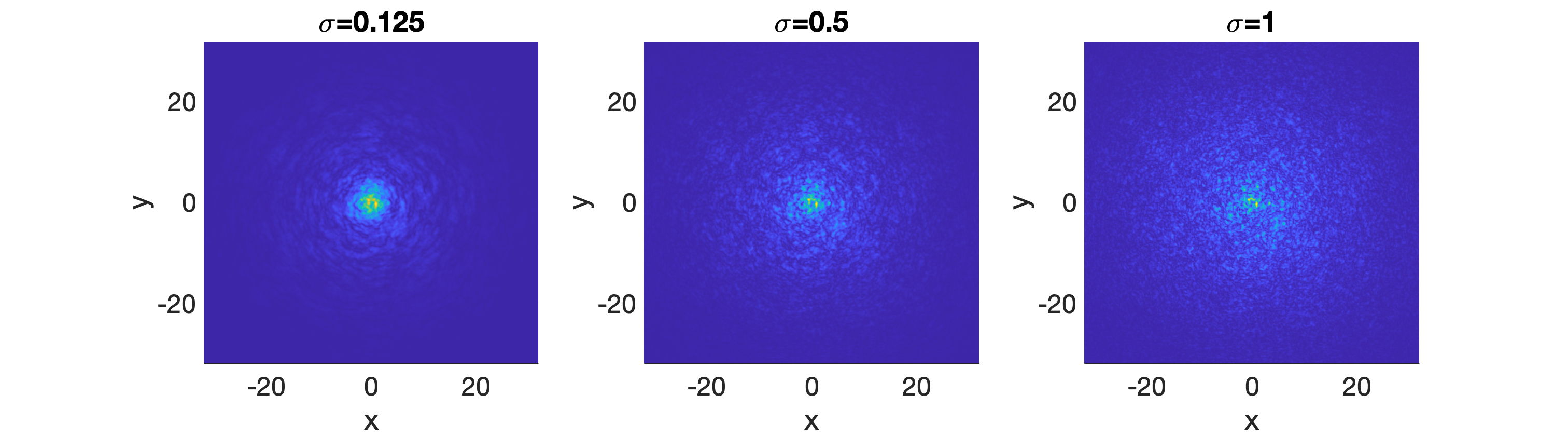}
    \caption{Receiver reading at $z=1$ for increasing $\sigma$.}
    \label{fig:evol_3D}         
\end{figure}

\bibliographystyle{siam}
\bibliography{Reference}

\end{document}